\documentclass[11pt]{article}

\usepackage[margin=1in]{geometry}

\usepackage{amsfonts}
\usepackage{amsmath}
\usepackage{amsthm} %
\usepackage{amssymb}
\usepackage{mathtools,empheq}
\usepackage{epsfig, graphics}
\usepackage{url}
\usepackage[dvipsnames]{xcolor}
\usepackage{setspace}
\usepackage{float}
\usepackage{colortbl}
\usepackage{bm}
\usepackage{bbm}
\usepackage{dsfont}
\usepackage[square, numbers, comma, sort&compress]{natbib}
\usepackage{tikz}
\usepackage{enumerate}   
\usetikzlibrary{shapes.geometric, arrows}
\usepackage{algorithm}
\usepackage{algpseudocode}
\usepackage{enumitem}
\usepackage{subcaption}
\usepackage{caption}

\DeclareMathAlphabet\mathbfcal{OMS}{cmsy}{b}{n}

\algrenewcommand\algorithmicrequire{\textbf{Input:}}
\algrenewcommand\algorithmicensure{\textbf{Output:}}

\tikzstyle{arrow} = [thick,->,>=stealth]

\makeatletter
\def\blfootnote{\gdef\@thefnmark{}\@footnotetext}
\makeatother

\title{Extended Graphon Mean-Field Games in Discrete Time%
\blfootnote{H.A. and A.S. acknowledge financial support from the Institut Europlace de Finance.}}

\newcommand{\cB}{\mathcal{B}}

\newcommand{\ioN}{\frac{i}{N}}
\newcommand{\joN}{\frac{j}{N}}

\newcommand{\R}{\mathbb{R}}
\newcommand{\RR}{\mathbb{R}}

\newcommand{\N}{\mathbb{N}}
\newcommand{\cW}{\mathcal{W}}
\newcommand{\cM}{\mathcal{M}}

\newcommand{\E}{\mathbb{E}}
\newcommand{\cK}{\mathcal{K}}
\newcommand{\cI}{\mathcal{I}}

\newcommand{\cA}{\mathcal{A}}

\newcommand{\cL}{\mathcal{L}}
\newcommand{\cP}{\mathcal{P}}

\newcommand{\cF}{\mathcal{F}}

\newcommand{\cJ}{\mathcal{J}}

\newcommand{\cX}{\mathcal{X}}
\newcommand{\cT}{\mathcal{T}}

\newcommand{\cH}{\mathcal{H}}

\newcommand{\bfT}{\mathbf{T}}

\newcommand{\bfPi}{\mathbf{\Pi}}
\newcommand{\G}{\mathbb{G}}
\newcommand{\ocT}{\overline\cT}

\newcommand{\ucM}{{\mathcal{\underline M}}}
\newcommand{\uPi}{{\underline\Pi}}
\newcommand{\ucZ}{{\mathcal{\underline Z}}}

\DeclareMathAlphabet\mathbfcal{OMS}{cmsy}{b}{n}

\newcommand{\bcZ}{{\mathbfcal{Z}}}
\newcommand{\bZ}{{\mathbf{Z}}}
\newcommand{\bucM}{{\mathbfcal{\underline M}}}
\newcommand{\buPi}{{\boldsymbol{\underline\Pi}}}
\newcommand{\bucZ}{{\mathbfcal{\underline Z}}}
\newcommand{\buZ}{{\mathbf{\underline Z}}}
\newcommand{\bV}{{\boldsymbol{\mathbf{V}}}}
\newcommand{\bmu}{{\boldsymbol{\mu}}}
\newcommand{\bpi}{{\boldsymbol{\pi}}}
\newcommand{\bumu}{{\boldsymbol{\underline\mu}}}
\newcommand{\umu}{{\underline\mu}}
\newcommand{\ueta}{{\underline\eta}}
\newcommand{\bupi}{{\boldsymbol{\underline\pi}}}
\newcommand{\upi}{{\underline\pi}}
\newcommand{\unu}{{\underline\nu}}

\newcommand{\bupsi}{{\boldsymbol{\underline\psi}}}
\newcommand{\buV}{{\boldsymbol{\underline{\mathbf{V}}}}}
\newcommand{\bunu}{{\boldsymbol{\underline\nu}}}

\usepackage{hyperref}
\hypersetup{
    colorlinks = true,
    linkcolor = {blue},
    citecolor = {blue}
}

\newtheorem{theorem}{Theorem}[section]
\newtheorem{remark}[theorem]{Remark}

\newtheorem{example}[theorem]{Example}
\newtheorem{assumption}{Assumption}
\newtheorem{definition}[theorem]{Definition}

\newtheorem{proposition}[theorem]{Proposition}
\newtheorem{lemma}[theorem]{Lemma}

\usepackage[normalem]{ulem}

\author{
Hamed Amini
\footnote{Center for Applied Optimization, Department of Industrial and Systems Engineering, University of Florida, Gainesville, FL, USA (\href{mailto:aminil@ufl.edu}{aminil@ufl.edu}).}
\and 
Zhongyuan Cao
\footnote{School of Science and Engineering, The Chinese University of Hong Kong (Shenzhen), Guangdong, 518172, China (\href{mailto:caozhongyuan@cuhk.edu.cn}{caozhongyuan@cuhk.edu.cn}).}
\and 
G\"ok\c ce Dayan{\i}kl{\i} 
\footnote{Department of Statistics, University of Illinois at Urbana-Champaign, Champaign, IL 61820, USA (\href{mailto:gokced@illinois.edu}{gokced@illinois.edu}).}
\and 
Mathieu Lauri\`ere
\footnote{Shanghai Center for Data Science; NYU-ECNU Institute of Mathematical Sciences, NYU Shanghai, Shanghai, 200126, People’s Republic of China (\href{mailto:mathieu.lauriere@nyu.edu}{mathieu.lauriere@nyu.edu}).}
\and 
Kexin Shao
\footnote{INRIA  Paris,  48 rue Barrault, CS 61534
75647 Paris Cedex, France (\href{mailto:coco.shao@nyu.edu}{coco.shao@nyu.edu}).}
\and 
Agnès Sulem
\footnote{INRIA  Paris,  48 rue Barrault, CS 61534
75647 Paris Cedex, France (\href{mailto:agnes.sulem@inria.fr}{agnes.sulem@inria.fr}).}
}

\date{}

\begin{document}

\maketitle %

\begin{abstract}
	In this paper, we study games involving a continuum of heterogeneous players in the discrete-time setting with finite state spaces and continuous action spaces. We introduce a new model that incorporates joint state-action interactions within the graphon-weighted aggregate, described by a coupled forward-backward system. We establish the existence of graphon mean-field equilibria and characterize them through this forward-backward system. Additionally, we provide uniqueness results under monotonicity and contraction conditions. To illustrate the practical relevance of our framework, we solve an example of portfolio liquidation with price impact. We provide numerical results for three different graphons and two different initial distributions, illustrating the impact of the network's structure and the distribution heterogeneity on the distribution and the policy.
\end{abstract}

\section{Introduction}

\paragraph{Motivation.}
Large-scale multi-agent systems are common in many real-life applications appearing in domains such as financial markets, energy systems, transportation networks, and epidemic control. In these settings, each individual agent makes decisions over time to optimize their objectives while interacting with a large population. Analyzing such systems is challenging because the complexity grows rapidly with the number of agents and the possible heterogeneity of their interactions.
Mean-Field Games (MFGs), introduced simultaneously by~\cite{lasrylions2007} and~\cite{huang2006}, provide a powerful approximation framework for the Nash equilibrium by considering the limit of infinitely many identical agents that interact symmetrically, i.e., each agent interacts with the population through a homogeneous aggregate involving the population distribution of states or actions. This assumption does not encompass applications where interactions are heterogeneous. For example, in financial markets, portfolio managers adjust their strategies based not only on market averages but also on the performance of specific peer groups or competitors they interact with. Similarly, in epidemic control, individuals’ decisions about social distancing depend on their local contact network such as their neighbors rather than the entire population. This creates heterogeneity in interactions and in agents’ model parameters.

To capture these heterogeneities, Graphon Mean-Field Games (GMFGs), introduced in~\cite{parise2019graphon,parise2021analysis}, extend the classical MFG framework by incorporating weighted interactions through a graphon which is the limit of large dense graphs. This allows modeling of complex network effects and heterogeneities in large populations. Most existing GMFG studies focus on the continuous-time setting or restrict both state and action spaces to be finite in the discrete-time setting. These assumptions limit applicability in scenarios where decisions are made in discrete time or when the learning algorithms that require time discretization are used and actions vary over a continuous domain which is a common feature in epidemic control~\cite{aurell2022finite,amini2022epidemic}, financial decision-making~\cite{tangpizhou2024,laurieretanpizhou2024deep}, opinion dynamics~\cite{opinion}, and energy market modeling~\cite{fraccarolo2025graphon}.
Moreover, existing formulations typically aggregate interactions based solely on states, ignoring the interactions through actions. Indeed, in many practical contexts, such as competitive investment strategies, demand response in energy markets, or epidemic spread, the influence of other agents depends jointly on the state and action distribution, or possibly on the action distribution only. Incorporating such state–action dependent interactions into the GMFG framework is essential for accurately modeling strategic behavior in populations with heterogeneous interactions. MFG models involving the action distribution are sometimes referred to as extended MFGs~\cite{gomes2016extended,alasseur2020extended} or MFG of controls~\cite{cardaliaguet2018mean,achdou2020mean,kobeissi2022classical}. For simplicity, we will refer to graphon games with action distribution as \emph{extended GMFGs}.

\paragraph{Literature review.}
Graphon games can be classified based on their time scale and state space. They were initially introduced for approximating Nash equilibria in \textit{static} and deterministic network games~\cite{parise2019graphon,parise2021analysis}. The static time setting is extended to incorporate stochasticity in~\cite{carmona2021}.

Dynamic (and stochastic) graphon mean-field games (GMFGs) can be further classified into continuous-time and discrete-time frameworks. Some work in the continuous-time setting includes the study of graphon mean-field interacting systems~\cite{bayraktar2020graphon, wu2022, bayraktar2022stationarity,amicaosu2022graphonbsde} and game settings based on them~\cite{caines2019graphon, caines2020graphon, caines2021,amicaosu2023graphoncontrol}. These studies analyze classical results such as propagation of chaos, existence and uniqueness of \emph{graphon mean-field equilibrium} (GMFE) and the approximation results for the finite player counterpart. Among continuous-time graphon games, the \emph{linear-quadratic (LQ) setting} has gained particular attention. \cite{gao2021} derived approximate NE for finite-player games on large graphs under the LQ setting. This work was later extended and refined by~\cite{StochGraphonGame2,lacker2023,hu2024finite}. In the continuous-time and discrete-state space framework,~\cite{aurell2022optimal} conducted a theoretical analysis of finite-state extended graphon games with applications to epidemic dynamics. Their work also proposes a numerical approach based on machine learning methods.

In the discrete-time setting,~\cite{cui2022b} analyzed GMFGs with continuum players, where both the state and action spaces are finite. They demonstrated the existence of Nash equilibria and approximate equilibria under the assumption of Lipschitz continuity in transition kernels and graphons.~\cite{fabian2023learning} extended this analysis to sparse GMFGs within the same framework. \cite{cui2022b} and~\cite{fabian2023learning} established asymptotic convergence.
We emphasize that none of these works studied discrete-time, finite-state, compact-action \textit{extended} graphon mean-field games.

Given that equilibrium solutions for GMFGs are often intractable, recent research efforts have increasingly focused on learning-based methods to approximate equilibria. This active area of study includes contributions on scalable algorithms and convergence guarantees~\cite{chen2023learning, cui2022b, aurell2022finite,fabian2023learning,zhou2024graphon}.

\paragraph{Contributions and structure.}

Our contributions are three-fold. First, the paper extends the discrete-time GMFG literature by considering interactions through the {joint distribution of states and actions} of other players. This allows us to incorporate aggregate interactions based jointly on states and actions for more extensive applications. Second, the paper presents a complete theoretical analysis of the discrete-time, discrete space extended mean-field games by providing {equilibrium characterization results}, {existence and uniqueness results}, and an {approximate equilibrium result} for the finite-player game. Finally, a {numerical illustration} that is motivated by the traders' investment problem in financial markets is presented and numerical results are shown under different graphon and initial distribution settings.

The structure of the paper is as follows. The model setting and technical definitions and assumptions are stated in Section~\ref{sec_model}. The characterization of the graphon mean-field equilibrium with a forward-backward system is presented in Section~\ref{sec_link}. The existence and uniqueness results for the graphon mean-field equilibrium are stated in Sections~\ref{sec_existence} and~\ref{sec_uniqueness}, respectively. The analysis of approximation of the finite player game with the graphon game is given in Section~\ref{sec_epsNash}. Finally, we present numerical experiments on an extended GMFG with application to portfolio liquidation in Section~\ref{sec:num1}.

\section{Graphon Mean-Field Games}\label{sec_model}
\subsection{Preliminaries on graphons}

Let $I:=[0,1]$ be the unit interval equipped with the Euclidean distance. We denote the Lebesgue measure and Borel $\sigma$-field on $I$ by Unif$(I)$ and $\cB(I)$ respectively. The set $I$ is an index set labeling a continuum of players in the graphon game.

Given an $N$-tuple $(x^1,\dots,x^N)$ and an index $i\in[N]$, we denote by $x^{-i}$ the $(N-1)$-tuple of the $x^j$ with $j \neq i$ and $j\in[N]$. Furthermore, with a slight abuse of notation, we identify $(x^i,x^{-i})$ with the $N$-tuple $(x^1,\dots,x^{i-1}, x^i, x^{i+1},\dots, x^N)$.

A graphon is defined as a symmetric measurable function $G: I \times I \rightarrow [0,1]$, which is regarded as the limits of adjacency matrices of weighted dense graphs when the size of graphs (number of vertices) tends to infinity. We denote by $\cW$ the space of graphons.
A graphon can also be viewed  as an operator from $L^{\infty}(I)$ to $L^1(I)$, mapping any $\varphi\in L^\infty(I)$
to:
$
	G\phi(u):=\int_I G (u,v) \phi(v) dv.
$
If $G \in L^p(I \times I)$, we denote by $\|G\|_p = \left(\int_{I \times I} G(u,v)^p du dv \right)^{1/p}$ its $L^p(I \times I)$ norm.
The cut norm of a graphon is defined by
\[
	\|G\|_{\square}:=\sup_{A,B\in\mathcal{B}(I)}\left| \int_{A\times B}G(u,v) \,du\,dv\right|.
\]
By Lov\'asz \cite[Lemma 8.11]{lovasz2012}, the operator norm is equivalent to the cut norm
$
	\| G\|_{\square}\leq  \|G \|_{\infty\to 1}\leq 4\| G\|_{\square},
$
with
$
	\|G \|_{\infty\to 1}:=\sup_{|\phi|\leq 1} \|G\phi\|_{L^1}.
$

\subsection{The model}\label{sec:model}

We consider a discrete-time, discrete-state setting with a finite horizon
$T\in\mathbb{N}$. We denote by
\[
\cT:=\{0,1,\ldots,T-1\}
\qquad\text{and}\qquad
\ocT:=\cT\cup\{T\}=\{0,1,\ldots,T\}
\]
the sets of decision times and state times, respectively.
For a random variable $X$ taking values in a measurable space, we write
$\cL(X)$ for its law.

\begin{definition}
	A finite-horizon, discrete-time graphon mean-field game (GMFG) is
	defined by a tuple $(T,G,\cX,\cA,P,r,r_T,\varphi,\mu_{\mathrm{init}})$,
	where
	\begin{itemize}\itemsep0em
		\item $T$ is the terminal time;
		\item $G$ is a graphon;
		\item $\cX$ is a finite state space;
		\item $\cA$ is an action set, which is a convex and compact subset
		      of $\RR$; we let $C_{\cA}>0$ be such that
		      $|a|\leq C_{\cA}$ for every $a\in\cA$;
		\item  $P:\cT\times I\times\cX\times\cA\times\RR
		      \longrightarrow\cP(\cX)$
		      is a measurable one-step transition probability kernel;
		\item
		      $r:\cT\times I\times\cX\times\cA\times\RR
		      \longrightarrow\RR$
		      is a measurable running reward function;
		\item
		      $r_T:I\times\cX\longrightarrow\RR$
		      is a measurable terminal reward function;
		\item
		       $\varphi:\cT\times I\times\cX\times\cA
		      \longrightarrow\RR$
		      is a measurable interaction function;
		\item
		      $\mu_{\mathrm{init}}
		      =(\mu_{\mathrm{init}}^u)_{u\in I}$ is a measurable family
		      of initial distributions, with
		      $\mu_{\mathrm{init}}^u\in\cP(\cX)$ for every $u\in I$.
	\end{itemize}
\end{definition}

We use bold letters to denote temporal sequences and underlined letters
to denote profiles indexed by the player label.
Given two sets $A$ and $B$, we denote by $A^B$ the set of functions
from $B$ to $A$. We introduce the following notation:
\begin{itemize}\itemsep0em
	\item $\cM := \cP(\cX)$, $\ucM := L^2(I; \cP(\cX))$ and $\bucM := \ucM^{\ocT}$ denote, respectively, the set of state distributions, the
	      set of graphon mean-fields, and the set of their temporal
	      sequences. For $\bumu\in\bucM$, we use the notation $(t,u,x) \mapsto \mu_t^u(x) := \mu(t,u,x)$, representing the probability of agents with index $u$ being at state $x$ at time $t$.
	\item  $\bfPi:=\cA^{\cT\times\cX}$
	      denotes the set of deterministic Markov policies for one
	      player. For policy profiles, we set
	      \[
	      \Pi:=L^2(I;\cA),\qquad
	      \uPi:=\Pi^{\cX},\qquad
	      \buPi:=\uPi^{\cT}.
	      \]
	      For $\bupi\in\buPi$, we use the notation
	      $
	      (t,u,x)\longmapsto
	      \pi_t^u(x):=\pi(t,u,x)\in\cA,
	      $
	      which specifies the action taken in state $x$ at time $t$ by
	      the agent with label $u$.
	      	\item
	      $
	      \bcZ:=\RR^{\cT},
	      \ucZ:=L^2(I),
	      \bucZ:=\ucZ^{\cT}
	      $
	      denote, respectively, the set of aggregate paths for one
	      player, the set of graphon-weighted aggregate profiles at one
	      time, and the set of their temporal sequences.
      
\end{itemize}

The spaces above consist of equivalence classes of measurable profiles.
Throughout the continuum model, all equations, fixed-point relations,
and optimality conditions involving the label $u$ are understood to hold
for Lebesgue-a.e.\ $u\in I$. Modifying a profile on a set of Lebesgue
measure zero does not change the induced mean-field and aggregate
profiles as elements of the corresponding $L^2$ spaces. Whenever a
later result requires evaluation at prescribed labels, such as $u=i/N$,
we will explicitly specify a representative satisfying the required
pointwise regularity and optimality properties; these properties do not
follow merely from choosing an arbitrary measurable representative.

We will also study value functions, which are elements of
$L^2(I)^{\ocT\times\cX}$. We endow this space with the norm
\begin{equation}
	\label{eq_def-norm-infinf2}
	\|\bupsi\|_{\infty,\infty,2}
	=
	\max_{t\in\ocT,\,x\in\cX}
	\left(\int_I |\psi_t^u(x)|^2\,du\right)^{1/2},
	\qquad
	\bupsi\in L^2(I)^{\ocT\times\cX}.
\end{equation}
We endow $\bucZ$ with the norm
\begin{equation}
	\label{eq_def-norm-inf2}
	\|\bupsi\|_{\infty,2}
	=
	\max_{t\in\cT}
	\left(\int_I |\psi_t^u|^2\,du\right)^{1/2},
	\qquad
	\bupsi\in\bucZ.
\end{equation}

We now define the mean-field and aggregate profiles induced by a policy
profile.

\begin{definition}\label{def_GMFG_mu_Z}
	Given $\bupi\in\buPi$, the induced graphon mean-field
	$\bumu^\bupi\in\bucM$ and aggregate
	$\buZ^\bupi\in\bucZ$ are defined recursively as follows. For a.e.\
	$u\in I$,
	\begin{equation}\label{eq_piZ0}
		\mu_0^{\bupi,u}(x)
		=
		\mu_{\mathrm{init}}^u(x),
		\qquad x\in\cX.
	\end{equation}
	For every $t\in\cT$ and a.e.\ $u\in I$,
	\begin{equation}\label{eq_piZt}
		\begin{dcases}
			Z_t^{\bupi,u}
			=
			\displaystyle
			\int_I G(u,v)
			\sum_{x\in\cX}
			\varphi_t^u\!\left(x,\pi_t^v(x)\right)
			\mu_t^{\bupi,v}(x)\,dv,
			\\[1.2em]
			\mu_{t+1}^{\bupi,u}(x)
			=
			\displaystyle
			\sum_{x'\in\cX}
			\mu_t^{\bupi,u}(x')
			P_t^u\!\left(
				x\mid x',
				\pi_t^u(x'),
				Z_t^{\bupi,u}
			\right),
			\qquad x\in\cX.
		\end{dcases}
	\end{equation}
\end{definition}

The total reward of player $u\in I$, using a policy
$\bpi^u\in\bfPi$ against an aggregate path
$\bZ^u\in\bcZ$, is
\begin{equation}\label{eq:reward}
	J^u(\bpi^u;\bZ^u)
	=
	\E\left[
		\sum_{t\in\cT}
		r_t^u(X_t^u,a_t^u,Z_t^u)
		+
		r_T^u(X_T^u)
	\right],
\end{equation}
where the state-action dynamics satisfy
\begin{equation}
	\label{eq_GraphonG-dynamics-Xu}
	X_0^u\sim\mu_{\mathrm{init}}^u,
	\qquad
	a_t^u=\pi_t^u(X_t^u),
	\qquad
	X_{t+1}^u
	\sim
	P_t^u\!\left(
		\cdot\mid X_t^u,a_t^u,Z_t^u
	\right),
	\qquad t\in\cT.
\end{equation}
Thus, actions, running rewards, and graphon-weighted aggregates are
defined at the decision times $t\in\cT$, while $r_T^u(X_T^u)$ is a
terminal state reward and no action is chosen at time $T$.

\begin{definition}[Graphon mean-field equilibrium]
	\label{def_Nash_GMFG}
	A pair
	$(\widehat\bupi,\widehat\buZ)\in\buPi\times\bucZ$
	is called a graphon mean-field equilibrium (GMFE) if:
	\begin{enumerate}[label=\textnormal{(\roman*)}]
		\item for a.e.\ $u\in I$, $\widehat\bpi^u$ is optimal against the
		      aggregate path $\widehat\bZ^u$, i.e.,
		      \[
			      J^u(\widehat\bpi^u;\widehat\bZ^u)
			      \geq
			      J^u(\widetilde\bpi;\widehat\bZ^u),
			      \qquad
			      \forall\,\widetilde\bpi\in\bfPi;
		      \]
		\item the aggregate profile is consistent with the policy
		      profile:
		      \[
			      \widehat\buZ=\buZ^{\widehat\bupi}.
		      \]
	\end{enumerate}
	A policy profile $\widehat\bupi\in\buPi$ is called a GMFE policy if
	$(\widehat\bupi,\buZ^{\widehat\bupi})$ is a GMFE pair.
\end{definition}

\subsection{Characterization of graphon mean-field equilibria}
\label{sec_link}

In this section, we characterize graphon mean-field equilibria through a
forward-backward system in which the forward equation represents the
evolution of the mean-field and the backward equation describes the
value function. For $\mu\in\cP(\cX)$, let
$\operatorname{supp}(\mu):=\{x\in\cX:\mu(x)>0\}$.

\begin{theorem}\label{thm_GMFE_sys}
	Consider the GMFG model of Section~\ref{sec:model}. Assume that the
	running rewards and the terminal reward are bounded and that, for
	every $t\in\cT$, $u\in I$, $x,x'\in\cX$, and $z\in\RR$, the
	mappings $a\longmapsto r_t^u(x,a,z)$ and $a\longmapsto P_t^u(x'\mid x,a,z)$
	are continuous on $\cA$. A pair
	$(\widehat\bupi,\widehat\buZ)\in\buPi\times\bucZ$
	is a GMFE if and only if there exist $(\buV,\bumu) \in L^2(I)^{\ocT\times\cX}\times\bucM$
	satisfying the following forward-backward system with
	$(\widehat\bupi,\widehat\buZ)=(\bupi,\buZ)$:
	\begin{subequations}\label{eq_sys_V_mu}
		\begin{empheq}[left=\empheqlbrace]{align}
			\mu_{t+1}^{u}(x)
			=&
			\sum_{x'\in\cX}
			\mu_t^{u}(x')
			P_t^u\left(
				x\mid x',
				\pi_t^{u}(x'),
				Z_t^u
			\right),
			\quad t\in\cT,
			\label{eq_sys_V_mu_a}
			\\
			V_t^{u}(x)
			=&
			\max_{a\in\cA}
			\left\{
				r_t^u(x,a,Z_t^u)
				+
				\sum_{x'\in\cX}
				P_t^u(x'\mid x,a,Z_t^u)
				V_{t+1}^{u}(x')
			\right\},
			\quad t\in\cT,
			\label{eq_sys_V_mu_b}
			\\
			Z_t^{u}
			=&
			\int_I
			G(u,v)
			\sum_{x\in\cX}
			\varphi_t^u\left(x,\pi_t^v(x)\right)
			\mu_t^v(x)\,dv,
			\quad t\in\cT,
			\label{eq_sys_V_mu_c}
			\\
			\pi_t^{u}(x)
			\in&
			\operatorname*{Argmax}_{a\in\cA}
			\left\{
				r_t^u(x,a,Z_t^u)
				+
				\sum_{x'\in\cX}
				P_t^u(x'\mid x,a,Z_t^u)
				V_{t+1}^{u}(x')
			\right\}, \
			t\in\cT, \
			x\in\operatorname{supp}(\mu_t^u),
			\label{eq_sys_V_mu_d}
			\\
			\mu_0^u(x)
			=&
			\mu_{\mathrm{init}}^u(x),
			\qquad
			V_T^u(x)=r_T^u(x),
			\label{eq_sys_V_mu_e}
		\end{empheq}
	\end{subequations}
for a.e.\ $u\in I$, with all equations holding for every
	$x\in\cX$ and all indicated times, except that the optimality
	condition~\eqref{eq_sys_V_mu_d} is imposed only for
	$x\in\operatorname{supp}(\mu_t^u)$.
\end{theorem}

\begin{proof}
	\textbf{(i) From GMFE to forward-backward system.}

	Assume $(\widehat\bupi,\widehat\buZ)$ is a GMFE, and let
	$\widehat\bumu:=\bumu^{\widehat\bupi}$ denote the mean-field flow induced
	by $\widehat\bupi$. Note that since the action set $\cA$ is compact,
	and by assumption the running reward $r$ and the transition kernel
	$P$ are continuous with respect to $a$, we have, for every
	$(t,u,x,Z,V)$, that the optimization problem $a\mapsto r_t^u(x,a,Z) + \sum_{x'\in\cX} P_t^u(x'\mid x,a,Z)V(x')$
	admits a maximizer. By Definition~\ref{def_Nash_GMFG}, we have the
	following two properties:
	\begin{enumerate}
		\item For a.e.\ $u\in I$, $\widehat\bpi^u$ is an optimal policy for
		      the player of index $u$ against the aggregate
		      $\widehat\bZ^u$. Let $\widehat\bV^u$ denote the optimal value
		      function against this aggregate path. By the Bellman
		      optimality principle for Markov decision processes
		      (see, e.g.,
		      \cite[Lemma 8.7, p.~206]{bertsekas1996stochastic}),
		      $\widehat\bV^u$ satisfies: $\widehat V_T^u(x)=r_T^u(x)$ for all $x\in\cX$, 
		      and
		      \[
			      \widehat V_t^u(x)
			      =
			      \max_{a\in\cA}
			      \left\{
				      r_t^u(x,a,\widehat Z_t^u)
				      +
				      \sum_{x'\in\cX}
				      P_t^u(x'\mid x,a,\widehat Z_t^u)
				      \widehat V_{t+1}^u(x')
			      \right\},
			      \qquad
			      x\in\cX,\quad t\in\cT.
		      \]
		      Moreover,
		      \[
			      \widehat\pi_t^u(x)
			      \in
			      \operatorname*{Argmax}_{a\in\cA}
			      \left\{
				      r_t^u(x,a,\widehat Z_t^u)
				      +
				      \sum_{x'\in\cX}
				      P_t^u(x'\mid x,a,\widehat Z_t^u)
				      \widehat V_{t+1}^u(x')
			      \right\},
			      \qquad
			      x\in\operatorname{supp}(\widehat\mu_t^u),
			      \quad t\in\cT.
		      \]
		      Indeed, if this condition failed at some
		      $(t,x)$ satisfying $\widehat\mu_t^u(x)>0$, one could replace
		      the continuation policy at $(t,x)$ by a Bellman-optimal
		      continuation. This change would weakly improve the
		      continuation payoff from every state at time $t$ and
		      strictly improve it from $x$. Since
		      $\widehat\mu_t^u(x)>0$, it would strictly increase the total
		      expected reward, contradicting the optimality of
		      $\widehat\bpi^u$. By the measurable maximum theorem and the boundedness of
		      the rewards, $\widehat\buV$ is measurable and belongs to
		      $L^2(I)^{\ocT\times\cX}$.

		\item Moreover,
		      $\widehat\buZ=\buZ^{\widehat\bupi}$, using the notation introduced
		      in Definition~\ref{def_GMFG_mu_Z}. By
		      \eqref{eq_piZ0}--\eqref{eq_piZt}, this means that
		      $\widehat\bumu$ and $\widehat\buZ$ satisfy
		      \begin{equation*}
			      \begin{dcases}
				      \widehat\mu_0^u(x)
				      =
				      \mu_{\mathrm{init}}^u(x),
				      \qquad x\in\cX,
				      \\[0.4em]
				      \widehat\mu_{t+1}^u(x)
				      =
				      \displaystyle
				      \sum_{x'\in\cX}
				      \widehat\mu_t^u(x')
				      P_t^u\left(
					      x\mid x',
					      \widehat\pi_t^u(x'),
					      \widehat Z_t^u
				      \right),
				      \qquad x\in\cX,\quad t\in\cT,
				      \\[0.8em]
				      \widehat Z_t^u
				      =
				      \displaystyle
				      \int_I
				      G(u,v)
				      \sum_{x\in\cX}
				      \varphi_t^u\left(
					      x,\widehat\pi_t^v(x)
				      \right)
				      \widehat\mu_t^v(x)\,dv,
				      \qquad t\in\cT.
			      \end{dcases}
		      \end{equation*}
		      Combining the above equations, we obtain the
		      forward-backward system~\eqref{eq_sys_V_mu}.
	\end{enumerate}

	\textbf{(ii) From forward-backward system to GMFE.}

	Let $(\bumu,\buV)$ be a solution to the forward-backward
	system~\eqref{eq_sys_V_mu}, with associated $(\bupi,\buZ)$. Fix
	$u\in I$ for which the system holds.
	By the Bellman equation, for every $t\in\cT$, $x\in\cX$, and
	$a\in\cA$,
	\[
		V_t^u(x)
		\geq
		r_t^u(x,a,Z_t^u)
		+
		\sum_{x'\in\cX}
		P_t^u(x'\mid x,a,Z_t^u)V_{t+1}^u(x').
	\]
	Furthermore, the optimality condition for $\bpi$ gives equality
	when $a=\pi_t^u(x)$ for every
	$x\in\operatorname{supp}(\mu_t^u)$. Therefore,
	\[
		\sum_{x\in\cX}
		\mu_t^u(x)V_t^u(x)
		=
		\sum_{x\in\cX}
		\mu_t^u(x)
		r_t^u\left(x,\pi_t^u(x),Z_t^u\right)
		+
		\sum_{x\in\cX}
		\mu_{t+1}^u(x)V_{t+1}^u(x),
	\]
	where we used the forward equation satisfied by $\bumu$.
	Iterating this identity from $t=0$ to $T-1$ and using
	$V_T^u=r_T^u$ yields
	$
		J^u(\bpi^u;\bZ^u)
		=
		\sum_{x\in\cX}
		\mu_{\mathrm{init}}^u(x)V_0^u(x).
	$

	Now let $\widetilde\bpi\in\bfPi$ be any alternative policy, and let
	$\widetilde\bumu^u$ denote the corresponding state-distribution
	sequence when the player uses $\widetilde\bpi$ against the fixed
	aggregate path $\bZ^u$. Applying the Bellman inequality with
	$a=\widetilde\pi_t(x)$, multiplying by $\widetilde\mu_t^u(x)$, summing over
	$x\in\cX$, and iterating from $t=0$ to $T-1$, we obtain
	$
		J^u(\widetilde\bpi;\bZ^u)
		\leq
		\sum_{x\in\cX}
		\mu_{\mathrm{init}}^u(x)V_0^u(x)
		=
		J^u(\bpi^u;\bZ^u).
	$
	Hence, $\bpi^u$ is a best response against $\bZ^u$ from the
	prescribed initial distribution.
	Finally, by the equations satisfied by $\bumu$ and $\buZ$, we have
	$\bumu=\bumu^{\bupi}$ and
	$\buZ=\buZ^{\bupi}$, i.e., $\buZ$ is the aggregate generated by
	$\bupi$. Therefore, the conditions in
	Definition~\ref{def_Nash_GMFG} are satisfied, and
	$(\bupi,\buZ)$ is a GMFE.
\end{proof}

\begin{remark}\label{rem:bellman-representative}
	The support condition in \eqref{eq_sys_V_mu_d} is necessary because
	Definition~\ref{def_Nash_GMFG} imposes optimality only from the
	prescribed initial distribution. The policy may therefore be
	arbitrary at state-time pairs that occur with zero probability.
	Moreover, under the strong-concavity assumptions introduced in
	Section~\ref{sec_existence}, the Bellman maximizer is unique.
	Consequently, a GMFE policy can be redefined at zero-mass states by
	the unique Bellman maximizer without changing its induced
	mean-field flow, aggregate, or payoff. The resulting canonical
	Bellman representative satisfies \eqref{eq_sys_V_mu_d} for every
	$x\in\cX$.
\end{remark}

\section{Existence of graphon mean-field equilibria}\label{sec_existence}

We now study existence of a GMFE by proving existence of a solution to the
forward--backward system~\eqref{eq_sys_V_mu}. We first introduce the
standing assumptions used throughout this section.

\subsection{Assumptions and preliminary results}
\label{sec_assumptionGMFG}

We use the following assumptions.

\begin{assumption}[Graphon regularity]
	\label{aasp_graphon-assumption}
	The mapping $u\mapsto G(u,\cdot)$ is continuous from $I$ into
	$L^2(I)$, that is,
	\[
		\lim_{|u-u'|\to0}
		\|G(u,\cdot)-G(u',\cdot)\|_{L^2(I)}
		=0.
	\]
\end{assumption}

Since $I$ is compact, this implies uniform continuity. We denote by
$\omega^G(\cdot)$ the modulus of continuity of this mapping:
\[
	\omega^G(h)
	:=
	\sup_{|u-v|\leq h}
	\left(
		\int_I |G(u,w)-G(v,w)|^2\,dw
	\right)^{1/2},
\]
with $\omega^G(h)\to0$ as $h\to0$.

\begin{assumption}[Interaction function regularity]
	\label{asm:varphi-regularity}
	Assume that there is a modulus of continuity
	$\omega^\varphi:\RR_+\to\RR_+$ such that
	\[
		\max_{\substack{t\in\cT,\,x\in\cX}}
		\sup_{a\in\cA}
		\left|
			\varphi_t^u(x,a)-\varphi_t^v(x,a)
		\right|
		\leq
		\omega^\varphi(|u-v|),
		\qquad u,v\in I,
	\]
	and $\omega^\varphi(h)\to0$ as $h\to0$.
\end{assumption}

We define, for all
$(Z,V,t,u,x,a)
	\in
	\RR\times\RR^\cX\times\cT\times I\times\cX\times\cA$,
\begin{equation}\label{eq_def-f-rplusP}
	f_t^u(Z,V,x,a)
	:=
	r_t^u(x,a,Z)
	+
	\sum_{x'\in\cX}
	P_t^u(x'\mid x,a,Z)V(x').
\end{equation}

\begin{assumption}\label{assp_diff_concave_faz}
	The coefficients $r$, $P$, and $f$ satisfy:
	\begin{enumerate}[label=\textnormal{(\arabic*)}]
		\item
		      For every $t\in\cT$, the running reward $r_t$ and the
		      transition kernel $P_t$ are twice differentiable with
		      respect to $a$. Moreover, there exist
		      $C_r,C_P,\epsilon_r>0$ such that, for all
		      $(t,u,x,x',a,Z) \in \cT\times I\times\cX\times\cX \times\cA\times\RR$,
		      \begin{itemize}
			      \item $|r_t^u(x,a,Z)|\leq C_r, \quad |r_T^u(x)|\leq C_r$;
			      \item $\left| \partial_{aa}^2 P_t^u(x'\mid x,a,Z) \right| \leq C_P;$
			      \item $\partial_{aa}^2r_t^u(x,a,Z) < -C_PC_r|\cX|T-\epsilon_r.$
		      \end{itemize}
		      \label{assp_atof_concave}

		\item
		      The mapping
		      $(Z,V)\mapsto\partial_af_t^u(Z,V,x,a)$ is Lipschitz
		      continuous, uniformly in
		      $a\in\cA$, $t\in\cT$, $u\in I$, and $x\in\cX$, on the
		      continuation-value domain $\RR\times[-TC_r,TC_r]^\cX$.
		      More precisely, there exist $L_{f,Z},L_{f,V}\geq0$ such
		      that, for all $(Z,V),(\widetilde Z,\widetilde V) \in \RR\times[-TC_r,TC_r]^\cX$,
		      \[
			      \max_{\substack{
				      t\in\cT,\ u\in I,\\
				      a\in\cA,\ x\in\cX
			      }}
			      \left|
				      \partial_af_t^u(Z,V,x,a)
				      -
				      \partial_af_t^u(
					      \widetilde Z,\widetilde V,x,a
				      )
			      \right|
			      \leq
			      L_{f,Z}|Z-\widetilde Z|
			      +
			      L_{f,V}
			      \max_{x\in\cX}
			      |V(x)-\widetilde V(x)|.
		      \]
		      \label{assp_ZVtodf_lip}
	\end{enumerate}
\end{assumption}

\begin{example}
	Let the transition kernel be given by the softmax function
	\[
		P_t^u(x'\mid x,a,Z)
		=
		\frac{
			\exp\bigl(
				\beta_{t,x,x'}^u
				+\gamma_{t,x,x'}^ua
				+\eta_{t,x,x'}^uZ
			\bigr)
		}{
			\sum_{y\in\cX}
			\exp\bigl(
				\beta_{t,x,y}^u
				+\gamma_{t,x,y}^ua
				+\eta_{t,x,y}^uZ
			\bigr)
		},
	\]
	where the coefficients
	$\beta_{t,x,x'}^u$, $\gamma_{t,x,x'}^u$, and
	$\eta_{t,x,x'}^u$ are uniformly bounded in $(t,x,x',u)$.
	Since the softmax function is $C^\infty$ in its arguments, the
	mapping $a\mapsto P_t^u(\cdot\mid x,a,Z)$
	is $C^2$, and its second derivative is uniformly bounded.
	Let
	$$
		r_t^u(x,a,Z)
		=
		b_t^u(x,Z)-\frac{\kappa}{2}a^2,
		\qquad
		|b_t^u(x,Z)|\leq C_b.
	$$
	Since $|a|\leq C_\cA$, $|r_t^u(x,a,Z)| \leq C_b+\frac{\kappa}{2}C_\cA^2$.
	Assume that $C_P|\cX|TC_\cA^2<2$
	and choose
	\[
		\kappa
		>
		\frac{
			C_PC_b|\cX|T+\epsilon_r
		}{
			1-\frac12C_P|\cX|TC_\cA^2
		}.
	\]
	Setting $C_r := C_b+\frac{\kappa}{2}C_\cA^2$
	and choosing a terminal reward satisfying
	$|r_T^u(x)|\leq C_r$, we obtain
	\[
		\partial_{aa}^2r_t^u(x,a,Z)
		=
		-\kappa
		<
		-C_PC_r|\cX|T-\epsilon_r.
	\]
	Moreover, the bounded softmax coefficients imply that
	$\partial_af_t^u$ is Lipschitz continuous in $(Z,V)$ on
	$\RR\times[-TC_r,TC_r]^\cX$, uniformly in $(t,u,x,a)$.
	Therefore, Assumption~\ref{assp_diff_concave_faz} holds.
\end{example}

We first state some useful lemmas needed in the proof of existence result (Theorem~\ref{thm_Vmu} below).

\begin{lemma}\label{lemma_f_stronglyConcave}
	Suppose Assumption~\ref{assp_diff_concave_faz} holds. Then the
	function $\cA\ni a \longmapsto f(\theta,a) := f_t^u(Z,V,x,a)$
	is continuously differentiable and strongly concave in $a$,
	uniformly in $\theta := (t,u,Z,V,x) \in \cT\times I\times\RR \times[-TC_r,TC_r]^\cX\times\cX$.
	In other words, there exists $\lambda>0$ such that, for every
	$a,\widetilde a\in\cA$, $\epsilon\in[0,1]$, and every such
	$\theta$,
	\[
		f\bigl(
			\theta,\epsilon a+(1-\epsilon)\widetilde a
		\bigr)
		\geq
		\epsilon f(\theta,a)
		+
		(1-\epsilon)f(\theta,\widetilde a)
		+
		\frac{\lambda}{2}
		\epsilon(1-\epsilon)
		|a-\widetilde a|^2.
	\]
\end{lemma}

\begin{proof}
	See Section~\ref{proof:lemma_f_stronglyConcave}.
\end{proof}

\begin{lemma}\label{lemma_ZtoA_lip}
	Suppose Assumption~\ref{assp_diff_concave_faz} holds.
	For all $\theta=(t,u,Z,V,x) \in \cT\times I\times\RR \times[-TC_r,TC_r]^\cX\times\cX$, the function $\cA\ni a \longmapsto f(\theta,a) = f_t^u(Z,V,x,a)$	admits a unique maximizer, which we denote by
	$\widehat a_t^u(Z,V,x)$. Moreover, the mapping $\RR\times[-TC_r,TC_r]^\cX \ni (Z,V) \longmapsto \widehat a_t^u(Z,V,x) \in\cA$
	is Lipschitz continuous uniformly in
	$(t,u,x)\in\cT\times I\times\cX$: there exist constants
	$L_{a,Z},L_{a,V}\geq0$ such that
	\[
		\max_{x\in\cX}
		\left|
			\widehat a_t^u(Z,V,x)
			-
			\widehat a_t^u(
				\widetilde Z,\widetilde V,x
			)
		\right|
		\leq
		L_{a,Z}|Z-\widetilde Z|
		+
		L_{a,V}
		\max_{x\in\cX}
		|V(x)-\widetilde V(x)|.
	\]
\end{lemma}

\begin{proof}
	See Section~\ref{proof:lemma_ZtoA_lip}.
\end{proof}

\begin{assumption}[Regularity of coefficients]
	\label{assp_P_lip}
	For every $t\in\cT$ and $x\in\cX$, we assume:
	\begin{enumerate}[label=\textnormal{(\arabic*)}]
		\item the mappings $I\times\cA\times\RR \ni (u,a,Z) \longmapsto P_t^u(\cdot\mid x,a,Z) \in\cP(\cX)$ and $ I\times\cA\times\RR \ni (u,a,Z) \longmapsto r_t^u(x,a,Z) \in\RR$ are jointly measurable, and the mapping $u\mapsto r_T^u(x)$ is measurable;
		\item\label{assp_z_to_P_lip}
		      the mapping
		      $
			      \RR\ni Z
			      \longmapsto
			      P_t^u(\cdot\mid x,a,Z)
			      \in\RR^{\cX}
		      $
		      is Lipschitz continuous with respect to the sup norm,
		      with Lipschitz constant $L_{P,Z}$, uniformly in
		      $t,u,x,a$. More precisely, for all
		      $Z,\widetilde Z\in\RR$,
		      \[
			      \max_{x'\in\cX}
			      \left|
				      P_t^u(x'\mid x,a,Z)
				      -
				      P_t^u(x'\mid x,a,\widetilde Z)
			      \right|
			      \leq
			      L_{P,Z}|Z-\widetilde Z|;
		      \]

		\item\label{assp_a_to_P_lip}
		      the mapping
		      $
			      \cA\ni a
			      \longmapsto
			      P_t^u(\cdot\mid x,a,Z)
			      \in\RR^{\cX}
		      $
		      is Lipschitz continuous with respect to the sup norm,
		      with Lipschitz constant $L_{P,a}$, uniformly in
		      $t,u,x,Z$. More precisely, for all
		      $a,\widetilde a\in\cA$,
		      \[
			      \max_{x'\in\cX}
			      \left|
				      P_t^u(x'\mid x,a,Z)
				      -
				      P_t^u(x'\mid x,\widetilde a,Z)
			      \right|
			      \leq
			      L_{P,a}|a-\widetilde a|;
		      \]

		\item\label{assp_z_to_r_lip}
		      the mapping
		      $
			      \RR\ni Z
			      \longmapsto
			      r_t^u(x,a,Z)
		      $
		      is Lipschitz continuous with Lipschitz constant
		      $L_{r,Z}$, uniformly in $t,u,x,a$, i.e., for all
		      $Z,\widetilde Z\in\RR$,
		      \[
			      \left|
				      r_t^u(x,a,Z)
				      -
				      r_t^u(x,a,\widetilde Z)
			      \right|
			      \leq
			      L_{r,Z}|Z-\widetilde Z|.
		      \]
	\end{enumerate}
\end{assumption}

\begin{lemma}\label{lemma_ztoVz_lip}
	For $u\in I$ and
	$\bZ\in\bcZ=\RR^\cT$, let us define
	\begin{equation}\label{eq_VZ}
		\begin{dcases}
			V_T^{\bZ,u}(x)
			=
			r_T^u(x),
			\quad x\in\cX,
			\\[0.4em]
			V_t^{\bZ,u}(x)
			=
			\displaystyle
			\max_{a\in\cA}
			\left\{
				r_t^u(x,a,Z_t)
				+
				\sum_{x'\in\cX}
				P_t^u(x'\mid x,a,Z_t)
				V_{t+1}^{\bZ,u}(x')
			\right\}, \quad
			x\in\cX,\quad t\in \cT.
		\end{dcases}
	\end{equation}
	Then, under Assumptions~\ref{assp_diff_concave_faz}
	and~\ref{assp_P_lip}:
	\begin{enumerate}[label=\textnormal{(\roman*)}]
		\item
		      \label{lemma_sub_ZtoV_lip_infinf}
		      the mapping $\bcZ\ni\bZ \longmapsto \bV^{\bZ,u} \in \RR^{\ocT\times\cX}$
		      is Lipschitz continuous uniformly in $u$, i.e., there
		      exists $K_{V,Z}\geq0$ such that
		      \begin{equation}\label{eq_Vdiff-Z}
			      \left\|
				      \bV^{\bZ,u}
				      -
				      \bV^{\widetilde\bZ,u}
			      \right\|_{\infty,\infty}
			      \leq
			      K_{V,Z}
			      \|\bZ-\widetilde\bZ\|_\infty,
		      \end{equation}
		      		      where $ \left \|\bV^{\bZ,u}\right \|_{\infty,\infty} = \max_{t \in \ocT,x \in \cX} \,|V^{\bZ,u}_t(x) | $ and $ \left\| \bZ \right\|_\infty = \max_{t\in \cT} \left| Z_t \right|$.
			
		\item
		      \label{lemma_sub_ZtoV_lip_infinf2}
		      the mapping $\bucZ\ni\buZ \longmapsto \buV^{\buZ} := \bigl(\bV^{\bZ^u,u}\bigr)_{u\in I} \in L^2(I)^{\ocT\times\cX}$
		      is Lipschitz continuous with the same, possibly enlarged,
		      constant $K_{V,Z}$:
		      \begin{equation}\label{eq_ZtoV_lip_infinf2}
			      \left\|
				      \buV^{\buZ}
				      -
				      \buV^{\widetilde\buZ}
			      \right\|_{\infty,\infty,2}
			      \leq
			      K_{V,Z}
			      \|\buZ-\widetilde\buZ\|_{\infty,2},
		      \end{equation}
		      where the norms $\|\cdot\|_{\infty,\infty,2}$ and $\|\cdot\|_{\infty,2}$ are introduced in~\eqref{eq_def-norm-infinf2} and~\eqref{eq_def-norm-inf2} respectively.

		\item
		      \label{lemma_sub_V_lessTCr}
		      For every $u\in I$, $x\in\cX$, and $t\in\ocT$,
		      \[
			      |V_t^{\bZ,u}(x)|
			      \leq
			      (T-t+1)C_r.
		      \]
	\end{enumerate}
\end{lemma}

\begin{proof}
	See Section~\ref{proof:lemma_ztoVz_lip}.
\end{proof}

\begin{assumption}[Lipschitz continuity and boundedness of $\varphi$]
	\label{assp_varphi_lip}
	There exist constants $L_\varphi,C_\varphi>0$ such that, for all
	$t\in\cT$, $u\in I$, $a,a'\in\cA$, and $x\in\cX$,
	\[
		|\varphi_t^u(x,a)-\varphi_t^u(x,a')|
		\leq
		L_\varphi|a-a'|,
		\qquad
		|\varphi_t^u(x,a)|
		\leq
		C_\varphi.
	\]
\end{assumption}

Under Assumption~\ref{assp_varphi_lip}, the interaction function
$\varphi$ is uniformly bounded by $C_\varphi$, and therefore every
aggregate profile $\buZ$ generated by a policy satisfies $\|\buZ\|_{\infty,2} \leq C_\varphi\|G\|_2 \leq C_\varphi$,
where the last inequality follows from $0\leq G\leq1$. In particular,
$\buZ$ belongs to the bounded subset
\[
	\bucZ_{C_\varphi}
	:=
	\left\{
		\bupsi\in\bucZ:
		\|\bupsi\|_{\infty,2}\leq C_\varphi
	\right\}.
\]

\begin{lemma}\label{lemma_ZtomuZ_lip}
	For $\buZ\in\bucZ_{C_\varphi}$, let us define
	\begin{equation}\label{eq_muZ}
		\begin{dcases}
			\mu_{t+1}^{\buZ,u}(x)
			=
			\displaystyle
			\sum_{x'\in\cX}
			\mu_t^{\buZ,u}(x')
			P_t^u\left(
				x\mid x',
				\pi_t^{*,\buZ,\buV^{\buZ},u}(x'),
				Z_t^u
			\right),\quad
			x\in\cX, u\in I, t\in\cT,
			\\[0.4em]
			\mu_0^{\buZ,u}(x)
			=
			\mu_{\mathrm{init}}^u(x),
			\quad x\in\cX, u\in I,
		\end{dcases}
	\end{equation}
	where $\buV^{\buZ}$ is defined in
	Lemma~\ref{lemma_ztoVz_lip}, and $\bupi^{*,\buZ,\buV^{\buZ}} \in L^2(I;\cA)^{\cT\times\cX}$
	is defined by
	\begin{equation}\label{eq_pi*ZVu}
			\pi_t^{*,\buZ,\buV^{\buZ},u}(x)
			=
			\widehat a_t^u
			\left(
				Z_t^u,
				V_{t+1}^{\buZ,u},
				x
			\right)
			=
			\operatorname*{argmax}_{a\in\cA}
			\left\{
				r_t^u(x,a,Z_t^u)
				+
				\sum_{x'\in\cX}
				P_t^u(x'\mid x,a,Z_t^u)
				V_{t+1}^{\buZ,u}(x')
			\right\},
	\end{equation}
	for every $x\in\cX$, $u\in I$, and $t\in\cT$. Then, under Assumptions~\ref{assp_diff_concave_faz},
	\ref{assp_P_lip}, and~\ref{assp_varphi_lip}:
	\begin{enumerate}[label=\textnormal{(\roman*)}]
		\item
		      \label{lemma_sub_ZtoPi_lip}
		      the mapping $\bucZ_{C_\varphi} \ni \buZ \longmapsto \bupi^{*,\buZ,\buV^{\buZ}} \in L^2(I;\cA)^{\cT\times\cX}$		      is Lipschitz continuous: there exists
		      $K_{\pi,Z}\geq0$ such that
		      \begin{equation*}%
			      \left\|
				      \bupi^{*,\buZ,\buV^{\buZ}}
				      -
				      \bupi^{*,\widetilde\buZ,
				      \buV^{\widetilde\buZ}}
			      \right\|_{\infty,\infty,2}
			      \leq
			      K_{\pi,Z}
			      \|\buZ-\widetilde\buZ\|_{\infty,2}.
		      \end{equation*}

		\item
		      \label{lemma_sub_ZtoMu_lip}
		      the mapping $\bucZ_{C_\varphi} \ni \buZ \longmapsto \bumu^{\buZ} := (\bmu^{\buZ,u})_{u\in I} \in L^2(I)^{\ocT\times\cX}$
		      is Lipschitz continuous: there exists
		      $K_{\mu,Z}\geq0$ such that
		      \begin{equation*}%
			      \|\bumu^{\buZ}
			      -
			      \bumu^{\widetilde\buZ}\|_{\infty,\infty,2}
			      \leq
			      K_{\mu,Z}
			      \|\buZ-\widetilde\buZ\|_{\infty,2}.
		      \end{equation*}
	\end{enumerate}
\end{lemma}

\begin{proof}
	See Section~\ref{proof:lemma_ZtomuZ_lip}.
\end{proof}

\begin{assumption}[Smallness of coefficients]
	\label{assp_joint_coef_less1}
	$C_\Phi := L_\varphi\|G\|_2L_{a,Z} <1$,	where $L_{a,Z}$ is defined in
	Lemma~\ref{lemma_ZtoA_lip}.
\end{assumption}

\begin{assumption}[Continuity with respect to the player label]
	\label{assp_u_cont}
	There exists a modulus of continuity
	$\omega_u:\RR_+\to\RR_+$, with
	$\omega_u(h)\to0$ as $h\to0$, such that, for all
	$t\in\cT$, $u,v\in I$, $x,x'\in\cX$, $a\in\cA$, and
	$Z\in\RR$,
	\[
		|r_t^u(x,a,Z)-r_t^v(x,a,Z)|
		\leq
		\omega_u(|u-v|),
	\]
	\[
		|P_t^u(x'\mid x,a,Z)-P_t^v(x'\mid x,a,Z)|
		\leq
		\omega_u(|u-v|).
	\]
	Moreover, for all $u,v\in I$ and $x\in\cX$,
	\[
		|r_T^u(x)-r_T^v(x)|
		\leq
		\omega_u(|u-v|),
	\]
	and
	\[
		\max_{x\in\cX}
		|\mu_{\mathrm{init}}^u(x)
		-
		\mu_{\mathrm{init}}^v(x)|
		\leq
		\omega_u(|u-v|).
	\]
\end{assumption}

\subsection{Existence theorem}

Based on the above assumptions and preliminary results, we have the
following existence result.

\begin{theorem}\label{thm_Vmu}
	Suppose Assumptions~\ref{aasp_graphon-assumption}--\ref{assp_u_cont}
	hold. Then there exist $(\buV,\bumu,\bupi,\buZ) \in L^2(I)^{\ocT\times\cX} \times\bucM\times\buPi\times\bucZ$
	satisfying the forward-backward system~\eqref{eq_sys_V_mu}.
	In particular, $(\bupi,\buZ)$ is a GMFE.
\end{theorem}

\begin{proof}%
	The proof of Theorem~\ref{thm_Vmu} relies on an application of
	Schauder's fixed point theorem and is divided into three steps.

	\noindent
	\textbf{Step 1. Definition of the solution space.}
	We recall that
	$L^2(I)^{\ocT\times\cX}$, endowed with the norm
	$\|\cdot\|_{\infty,\infty,2}$ defined
	in~\eqref{eq_def-norm-infinf2}, is a Banach space.
	To simplify the notation, we denote
	\begin{equation}\label{eq_cKnorm}
		\|(\buV,\bumu)\|_{\cK}
		=
		\max\left\{
			\|\buV\|_{\infty,\infty,2},
			\|\bumu\|_{\infty,\infty,2}
		\right\}.
	\end{equation}
	Let $C_1 := \max\left\{(T+1)C_r,1\right\}$,
	and define
	\begin{equation}\label{def_KC1}
		\cK_{C_1}
		:=
		\left\{
			(\buV,\bumu)
			\in
			L^2(I)^{\ocT\times\cX}\times\bucM:
			|V_t^u(x)|
			\leq
			(T-t+1)C_r
			\text{ for all }t,x
			\text{ and a.e.\ }u
		\right\}.
	\end{equation}
	Since $\bumu$ takes values in the simplex $\cP(\cX)$, we have $\|(\buV,\bumu)\|_{\cK} \leq C_1$ for every $(\buV,\bumu)\in\cK_{C_1}$.
	The set $\cK_{C_1}$ is a nonempty closed convex subset of the
	Banach space $L^2(I)^{\ocT\times\cX} \times L^2(I)^{\ocT\times\cX}$.
	Closeness follows from the closed constraints defining the
	time-dependent intervals $[-(T-t+1)C_r,(T-t+1)C_r]$
	and the simplex $\cP(\cX)$.

	\medskip

	\noindent
	\textbf{Step 2. Definition of the aggregate mapping.}
	Recall that we endow $\bucZ$ with the norm
	$\|\cdot\|_{\infty,2}$ defined
	in~\eqref{eq_def-norm-inf2}. We denote
	\[
		\bucZ_{C_\varphi}
		:=
		\left\{
			\buZ\in\bucZ:
			\|\buZ\|_{\infty,2}
			\leq C_\varphi
		\right\}.
	\]

	For
	$(\buV,\bumu)\in\cK_{C_1}$ and
	$\buZ\in\bucZ_{C_\varphi}$, define $\pi_t^{*,\buZ,\buV,u}(x) := \widehat a_t^u \left( Z_t^u,V_{t+1}^u,x \right)$,  $t\in\cT, x\in\cX$.
	Notice that $|V_{t+1}^u(x)| \leq (T-t)C_r \leq TC_r$,
	so that the selector
	$\widehat a_t^u$ from
	Lemma~\ref{lemma_ZtoA_lip} is well defined.

	We apply the Banach fixed point theorem to the map $\Phi^{(\buV,\bumu)} : \bucZ_{C_\varphi} \longrightarrow \bucZ$
	defined by
	\begin{equation}
		\label{eq_def-Phi-for-Z}
		\Phi^{(\buV,\bumu)}(\buZ)_t^u
		=
		\int_I
		G(u,v)
		\sum_{x\in\cX}
		\varphi_t^u
		\left(
			x,\pi_t^{*,\buZ,\buV,v}(x)
		\right)
		\mu_t^v(x)\,dv,
		\quad
		t\in\cT, u\in I.
	\end{equation}
	The mapping is measurable by the measurability of the coefficients
	and of the unique maximizer established above.
	Under Assumption~\ref{assp_varphi_lip}, $\left| \Phi^{(\buV,\bumu)}(\buZ)_t^u \right| \leq C_\varphi$.
	Consequently, $\Phi^{(\buV,\bumu)} \left( \bucZ_{C_\varphi} \right) \subseteq \bucZ_{C_\varphi}$.
	Furthermore, $\bucZ_{C_\varphi}$ is a closed subset of the Banach
	space $\bucZ$, and is therefore complete.

	We next verify that
	$\Phi^{(\buV,\bumu)}$ is a strict contraction.
	For
	$\buZ,\widetilde\buZ\in\bucZ_{C_\varphi}$, the Lipschitz
	continuity of $\varphi$ and
	Lemma~\ref{lemma_ZtoA_lip} imply, for every $t\in\cT$ and
	a.e.\ $u\in I$,
	\begin{align*}
		\left|
			\Phi^{(\buV,\bumu)}(\buZ)_t^u
			-
			\Phi^{(\buV,\bumu)}(\widetilde\buZ)_t^u
		\right|
		&\leq
		\int_I
		G(u,v)
		\sum_{x\in\cX}
		L_\varphi
		\left|
			\pi_t^{*,\buZ,\buV,v}(x)
			-
			\pi_t^{*,\widetilde\buZ,\buV,v}(x)
		\right|
		\mu_t^v(x)\,dv
		\\
		&\leq
		L_\varphi L_{a,Z}
		\int_I
		G(u,v)
		|Z_t^v-\widetilde Z_t^v|\,dv.
	\end{align*}
	Using the Hilbert--Schmidt bound for the graphon operator, we
	obtain
	\[
		\left\|
			\Phi^{(\buV,\bumu)}(\buZ)_t
			-
			\Phi^{(\buV,\bumu)}(\widetilde\buZ)_t
		\right\|_{L^2(I)}
		\leq
		L_\varphi L_{a,Z}\|G\|_2
		\|Z_t-\widetilde Z_t\|_{L^2(I)}.
	\]
	Taking the maximum over $t\in\cT$ gives
	\[
		\left\|
			\Phi^{(\buV,\bumu)}(\buZ)
			-
			\Phi^{(\buV,\bumu)}(\widetilde\buZ)
		\right\|_{\infty,2}
		\leq
		L_\varphi\|G\|_2L_{a,Z}
		\|\buZ-\widetilde\buZ\|_{\infty,2}.
	\]
	Since $L_\varphi\|G\|_2L_{a,Z} =:C_\Phi<1$
	by Assumption~\ref{assp_joint_coef_less1}, the Banach fixed point
	theorem implies that
	$\Phi^{(\buV,\bumu)}$ has a unique fixed point in
	$\bucZ_{C_\varphi}$. We denote it by $\widehat\buZ^{(\buV,\bumu)}$.

	\medskip

	\noindent
	\textbf{Step 3. Application of Schauder's theorem.}
	For
	$(\buV,\bumu)\in\cK_{C_1}$, Step~2 provides a unique fixed point
	$\widehat\buZ^{(\buV,\bumu)}$.
	Conversely, given
	$\buZ\in\bucZ_{C_\varphi}$, we define
	$\buV^{\buZ}$ and $\bumu^{\buZ}$ as in
	\eqref{eq_VZ} and~\eqref{eq_muZ}, respectively.

	Let $\Psi: \cK_{C_1} \longrightarrow \cK_{C_1}$
	be defined by $\Psi(\buV,\bumu) := \Bigl( \buV^{\widehat\buZ^{(\buV,\bumu)}}, \bumu^{\widehat\buZ^{(\buV,\bumu)}} \Bigr)$.
	By Lemma~\ref{lemma_ztoVz_lip}%
	\ref{lemma_sub_V_lessTCr}, the first component satisfies the
	required value bounds, while the second component takes values in
	$\cP(\cX)$. Hence, $\Psi(\cK_{C_1}) \subseteq \cK_{C_1}$.
	We will apply Schauder's fixed point theorem
	(see, for instance,
	\cite[Theorem (Schauder), p.~179]{brezis2011functional}) after
	verifying:
	\begin{enumerate}[label=(\roman*)]
		\item
		      $\Psi:\cK_{C_1}\to\cK_{C_1}$ is continuous;
		      \label{cond_Schauder_continuous}
		\item
		      $\Psi(\cK_{C_1})$ is relatively compact in $L^2(I)^{\ocT\times\cX} \times L^2(I)^{\ocT\times\cX}$.
		      \label{cond_Schauder_compact}
	\end{enumerate}

	\medskip

	\noindent
	\textbf{Continuity of $\Psi$.}
	We first establish the continuity of $(\buV,\bumu) \longmapsto \widehat\buZ^{(\buV,\bumu)}$.
		Let $\buZ = \widehat\buZ^{(\buV,\bumu)}$ and $\widetilde\buZ = \widehat\buZ^{(\widetilde\buV,\widetilde\bumu)}$
	denote the corresponding fixed points. Since
	\[
		\buZ
		=
		\Phi^{(\buV,\bumu)}(\buZ),
		\qquad
		\widetilde\buZ
		=
		\Phi^{(\widetilde\buV,\widetilde\bumu)}
		(\widetilde\buZ),
	\]
	we have
	\begin{align*}
		\|\buZ-\widetilde\buZ\|_{\infty,2}
		&\leq
		\left\|
			\Phi^{(\buV,\bumu)}(\buZ)
			-
			\Phi^{(\buV,\bumu)}(\widetilde\buZ)
		\right\|_{\infty,2}
		+
		\left\|
			\Phi^{(\buV,\bumu)}(\widetilde\buZ)
			-
			\Phi^{(\widetilde\buV,\widetilde\bumu)}
			(\widetilde\buZ)
		\right\|_{\infty,2}.
	\end{align*}
	The first term is bounded by $C_\Phi \|\buZ-\widetilde\buZ\|_{\infty,2}$.
	For the second term, decompose the difference into the change in
	the greedy policy and the change in the mean-field. For fixed
	$\widetilde\buZ$, Lemma~\ref{lemma_ZtoA_lip} yields
	\[
		\left|
			\pi_t^{*,\widetilde\buZ,\buV,v}(x)
			-
			\pi_t^{*,\widetilde\buZ,
			\widetilde\buV,v}(x)
		\right|
		\leq
		L_{a,V}
		\max_{y\in\cX}
		|V_{t+1}^v(y)-\widetilde V_{t+1}^v(y)|.
	\]
	Since
	\[
		\left(
			\int_I
			\max_{y\in\cX}
			|V_{t+1}^v(y)-\widetilde V_{t+1}^v(y)|^2\,dv
		\right)^{1/2}
		\leq
		\sqrt{|\cX|}
		\|\buV-\widetilde\buV\|_{\infty,\infty,2},
	\]
	and
	\[
		\left\|
			\sum_{x\in\cX}
			|\mu_t^\cdot(x)-\widetilde\mu_t^\cdot(x)|
		\right\|_{L^2(I)}
		\leq
		|\cX|
		\|\bumu-\widetilde\bumu\|_{\infty,\infty,2},
	\]
	the Hilbert--Schmidt bound gives
	\begin{align*}
		\left\|
			\Phi^{(\buV,\bumu)}(\widetilde\buZ)
			-
			\Phi^{(\widetilde\buV,\widetilde\bumu)}
			(\widetilde\buZ)
		\right\|_{\infty,2}
		\leq
		L_\varphi L_{a,V}
		\sqrt{|\cX|}\,
		\|G\|_2
		\|\buV-\widetilde\buV\|_{\infty,\infty,2}
		+
		C_\varphi|\cX|\,
		\|G\|_2
		\|\bumu-\widetilde\bumu\|_{\infty,\infty,2}.
	\end{align*}
	Consequently,
	\begin{align}
		\left\|
			\widehat\buZ^{(\buV,\bumu)}
			-
			\widehat\buZ^{(\widetilde\buV,\widetilde\bumu)}
		\right\|_{\infty,2}
		\leq
		\frac{\|G\|_2}{1-C_\Phi}
		\Bigl(
			L_\varphi L_{a,V}\sqrt{|\cX|}
			\|\buV-\widetilde\buV\|_{\infty,\infty,2}
			+
			C_\varphi|\cX|
			\|\bumu-\widetilde\bumu\|_{\infty,\infty,2}
		\Bigr).
		\label{eq:Zhat-continuity}
	\end{align}

	We next use the Lipschitz continuity of
	$\buZ\mapsto\buV^{\buZ}$ and
	$\buZ\mapsto\bumu^{\buZ}$. Recalling
	\eqref{eq_cKnorm}, we have
	\begin{align}
		\left\|
			\Psi(\buV,\bumu)
			-
			\Psi(\widetilde\buV,\widetilde\bumu)
		\right\|_{\cK}
		&=
		\max
		\Bigl\{
			\|
				\buV^{\widehat\buZ^{(\buV,\bumu)}}
				-
				\buV^{\widehat\buZ^{(\widetilde\buV,
				\widetilde\bumu)}}
			\|_{\infty,\infty,2},
			\|
				\bumu^{\widehat\buZ^{(\buV,\bumu)}}
				-
				\bumu^{\widehat\buZ^{(\widetilde\buV,
				\widetilde\bumu)}}
			\|_{\infty,\infty,2}
		\Bigr\}
		\notag
		\\
		&\leq
		\max\{K_{V,Z},K_{\mu,Z}\}
		\left\|
			\widehat\buZ^{(\buV,\bumu)}
			-
			\widehat\buZ^{(\widetilde\buV,\widetilde\bumu)}
		\right\|_{\infty,2}.
		\label{eq_Psi_lipVmuZ}
	\end{align}
	Combining
	\eqref{eq:Zhat-continuity} and
	\eqref{eq_Psi_lipVmuZ}, we conclude that
	$\Psi:\cK_{C_1}\to\cK_{C_1}$ is continuous.

	\medskip

	\noindent
	\textbf{Definition of $\cF$ and proof of relative compactness.}
	We now prove Condition~\ref{cond_Schauder_compact}. Let $\cF := \Psi(\cK_{C_1})$.
	We will show that the elements of $\cF$ admit representatives that
	are uniformly bounded and uniformly equicontinuous as functions of
	the label $u$. The result will then follow from the
	Arzelà--Ascoli theorem.

	We first establish a common modulus of continuity for the aggregate
	profiles
	$\widehat\buZ^{(\buV,\bumu)}$. Since
	$\widehat\buZ^{(\buV,\bumu)}$ is a fixed point of
	$\Phi^{(\buV,\bumu)}$, we choose its representative given by the
	right-hand side of~\eqref{eq_def-Phi-for-Z}. For
	$u,u'\in I$ and $t\in\cT$, we have
	\begin{align*}
		\left|
			\widehat Z_t^{(\buV,\bumu),u}
			-
			\widehat Z_t^{(\buV,\bumu),u'}
		\right|
		&\leq
		\int_I
		|G(u,v)-G(u',v)|
		\left|
			\sum_{x\in\cX}
			\varphi_t^u
			\left(
				x,\pi_t^{*,\widehat\buZ^{(\buV,\bumu)},
				\buV,v}(x)
			\right)
			\mu_t^v(x)
		\right|\,dv
		\\
		+&
		\int_I
		G(u',v)
		\sum_{x\in\cX}
		\left|
			\varphi_t^u
			\left(
				x,\pi_t^{*,\widehat\buZ^{(\buV,\bumu)},
				\buV,v}(x)
			\right)
			-
			\varphi_t^{u'}
			\left(
				x,\pi_t^{*,\widehat\buZ^{(\buV,\bumu)},
				\buV,v}(x)
			\right)
		\right|
		\mu_t^v(x)\,dv
		\\
		&\leq
		C_\varphi
		\|G(u,\cdot)-G(u',\cdot)\|_{L^1(I)}
		+
		\omega^\varphi(|u-u'|)
		\\
		&\leq
		C_\varphi\omega^G(|u-u'|)
		+
		\omega^\varphi(|u-u'|).
	\end{align*}
	Define $\omega_Z(h) := C_\varphi\omega^G(h) + \omega^\varphi(h)$.
	Then
	\[
		\max_{t\in\cT}
		\left|
			\widehat Z_t^{(\buV,\bumu),u}
			-
			\widehat Z_t^{(\buV,\bumu),u'}
		\right|
		\leq
		\omega_Z(|u-u'|),
	\]
	uniformly in
	$(\buV,\bumu)\in\cK_{C_1}$, and
	$\omega_Z(h)\to0$ as $h\to0$.

	We next obtain uniform moduli for the value functions, greedy
	policies, and state distributions generated by these aggregates.

	Define recursively $\omega_{V,T}(h) := \omega_u(h)$,
	and, for $t=T-1,\ldots,0$,
	\begin{align*}
		\omega_{V,t}(h)
		:={}&
		\omega_u(h)
		+
		L_{r,Z}\omega_Z(h)
		+
		|\cX|(T-t)C_r
		\left(
			\omega_u(h)
			+
			L_{P,Z}\omega_Z(h)
		\right)
		+
		\omega_{V,t+1}(h).
	\end{align*}
	Using the inequality between maxima, Assumption~\ref{assp_u_cont},
	the Lipschitz continuity in the aggregate from
	Assumption~\ref{assp_P_lip}, and backward induction, we obtain
	\[
		\max_{x\in\cX}
		\left|
			V_t^{\widehat\bZ^{(\buV,\bumu),u},u}(x)
			-
			V_t^{\widehat\bZ^{(\buV,\bumu),u'},u'}(x)
		\right|
		\leq
		\omega_{V,t}(|u-u'|).
	\]
	The terminal step uses the continuity of
	$u\mapsto r_T^u(x)$ contained in
	Assumption~\ref{assp_u_cont}.

	For $t\in\cT$, set
	\begin{align*}
		\delta_t(h)
		:={}&
		\omega_u(h)
		+
		L_{r,Z}\omega_Z(h)
		+
		|\cX|(T-t)C_r
		\left(
			\omega_u(h)
			+
			L_{P,Z}\omega_Z(h)
		\right)
		+
		\omega_{V,t+1}(h).
	\end{align*}
	The corresponding Bellman objectives for labels $u$ and $u'$ differ
	uniformly in $a$ by at most $\delta_t(|u-u'|)$. Since both
	objectives are $\lambda$-strongly concave, their unique maximizers
	satisfy
	\[
		\max_{x\in\cX}
		\left|
			\pi_t^{*,\widehat\buZ^{(\buV,\bumu)},
			\buV^{\widehat\buZ^{(\buV,\bumu)}},u}(x)
			-
			\pi_t^{*,\widehat\buZ^{(\buV,\bumu)},
			\buV^{\widehat\buZ^{(\buV,\bumu)}},u'}(x)
		\right|
		\leq
		\omega_{\pi,t}(|u-u'|),
	\]
	where $\omega_{\pi,t}(h) := 2\sqrt{\frac{\delta_t(h)}{\lambda}}$.

	Indeed, if $Q^u$ and $Q^{u'}$ denote the two Bellman objectives
	and $a^u,a^{u'}$ their respective maximizers, strong concavity
	implies
	\[
		\frac{\lambda}{2}|a^u-a^{u'}|^2
		\leq
		Q^u(a^u)-Q^u(a^{u'})
		\leq
		2\sup_{a\in\cA}
		|Q^u(a)-Q^{u'}(a)|.
	\]

	Finally, define $\omega_{\mu,0}(h) := \omega_u(h)$,
	and, for $t\in\cT$,
	\begin{equation*}
		\omega_{\mu,t+1}(h)
		:=
		|\cX|\omega_{\mu,t}(h)
		+
		\omega_u(h)
		+
		L_{P,a}\omega_{\pi,t}(h)
		+
		L_{P,Z}\omega_Z(h).
	\end{equation*}
	Using the forward equation, Assumption~\ref{assp_u_cont}, and
	Assumption~\ref{assp_P_lip}, a forward induction gives
	\[
		\max_{x\in\cX}
		\left|
			\mu_t^{\widehat\buZ^{(\buV,\bumu)},u}(x)
			-
			\mu_t^{\widehat\buZ^{(\buV,\bumu)},u'}(x)
		\right|
		\leq
		\omega_{\mu,t}(|u-u'|).
	\]

	All the moduli
	$\omega_Z$, $\omega_{V,t}$, $\omega_{\pi,t}$, and
	$\omega_{\mu,t}$ tend to zero as $h\to0$, uniformly in
	$(\buV,\bumu)\in\cK_{C_1}$. Hence the elements of $\cF$, viewed as functions from $I$ into $\RR^n, n=2|\cX||\ocT|$,
	admit uniformly bounded and uniformly equicontinuous
	representatives. By the Arzelà--Ascoli theorem, $\cF$ is relatively
	compact in $C(I;\RR^n)$.
	Since the embedding $C(I;\RR^n) \hookrightarrow L^2(I;\RR^n)$
	is continuous, $\cF$ is also relatively compact in $L^2(I)^{\ocT\times\cX} \times L^2(I)^{\ocT\times\cX}$.
	This proves Condition~\ref{cond_Schauder_compact}.

	Since
	$\Psi:\cK_{C_1}\to\cK_{C_1}$ is continuous and has relatively
	compact image, Schauder's fixed point theorem implies that
	$\Psi$ admits a fixed point, denoted by
	$(\buV^*,\bumu^*)$. Let $\buZ^* := \widehat\buZ^{(\buV^*,\bumu^*)}$
	and define
	\[
		\pi_t^{*,u}(x)
		:=
		\widehat a_t^u
		\left(
			Z_t^{*,u},
			V_{t+1}^{*,u},
			x
		\right).
	\]
	The fixed-point identity for $\Psi$ gives the backward value
	equation and the forward mean-field equation, while the fixed-point
	identity for
	$\Phi^{(\buV^*,\bumu^*)}$ gives the aggregate consistency equation.
	The policy $\bupi^*$ satisfies the Bellman optimality condition at
	every state and therefore, in particular, on the support of
	$\bumu^*$. Thus $(\buV^*,\bumu^*,\bupi^*,\buZ^*)$
	satisfies~\eqref{eq_sys_V_mu}. By
	Theorem~\ref{thm_GMFE_sys},
	$(\bupi^*,\buZ^*)$ is a GMFE.
\end{proof}

\section{Uniqueness Results} \label{sec_uniqueness}

We establish uniqueness results through two different approaches: monotonicity and contraction.

\subsection{Uniqueness by monotonicity}
\label{sec:unique-mono}

For every family of state-action distributions
$\unu=(\nu^u)_{u\in I}$ such that the mapping
$I\ni u\mapsto \nu^u\in\cP(\cX\times\cA)$ is measurable, and every
$u\in I$, the aggregate $Z_t^{\unu,u}\in\RR$ is defined as
\begin{equation}\label{def_Znu}
	Z_t^{\unu,u}
	=
	\int_{v\in I}
	G(u,v)
	\sum_{x\in\cX}
	\int_\cA
	\varphi_t^u(x,a)\nu^v(x,da)\,dv.
\end{equation}
Given $\mu\in\cP(\cX)$ and $\pi\in\cA^\cX$, we denote by
$\nu^{\pi,\mu}\in\cP(\cX\times\cA)$ the joint state-action distribution
defined as
\begin{equation*}
	\nu^{\pi,\mu}(x,da)
	=
	\mu(x)\,\delta_{\pi(x)}(da).
\end{equation*}
We extend these notations to time- and index-dependent objects. For
instance, $\bunu^{\bupi,\bumu}$ satisfies $\nu_t^{\bupi,\bumu,u}(x,da) = \mu_t^u(x)\,\delta_{\pi_t^u(x)}(da),$
for all $u\in I$ and $t\in\cT$.

Given $\bupi=(\pi_t^u)_t$, we recall that $\bumu^\bupi$ denotes the
mean-field generated by $\bupi$. To simplify notation, we denote $\bunu^\bupi := \bunu^{\bupi,\bumu^\bupi}$.
Furthermore, we denote $Z_t^{\bunu^\bupi,u} := Z_t^{\bupi,u}$, for all $t\in\cT, u\in I$,
where $\left(Z_t^{\bupi,u}\right)_{u\in I,t\in\cT} \in\bucZ$
is introduced in Definition~\ref{def_GMFG_mu_Z}.

We need the following monotonicity assumption.

\begin{assumption}[Monotonicity condition]
	\label{assp_LLmonotone}
	The transitions do not involve interactions, that is, $P$ is
	constant with respect to $Z$. Moreover, the running reward $r$ is
	monotone in the following sense: for every pair of measurable
	state-action distribution profiles $\bunu=(\nu_t^u)_{t\in\cT,u\in I}$ and  $\widetilde\bunu = (\widetilde\nu_t^u)_{t\in\cT,u\in I}$,
	we have
	\[
		\sum_{t\in\cT}
		\int_{u\in I}
		\sum_{x\in\cX}
		\int_\cA
		\left(
			r_t^u(x,a,Z_t^{\bunu,u})
			-
			r_t^u(x,a,Z_t^{\widetilde\bunu,u})
		\right)
		\left(
			\nu_t^u(x,da)
			-
			\widetilde\nu_t^u(x,da)
		\right)\,du
		\leq0.
	\]
	We assume moreover that equality can occur only if
	$\bunu=\widetilde\bunu$ up to null sets for the product of
	counting measure on $\cT$ and Lebesgue measure on $I$.
\end{assumption}

\begin{theorem}\label{thm_UniqueNash_LLmonotone}
	Suppose Assumption~\ref{assp_LLmonotone} holds. Then there exists
	at most one graphon state-action flow and aggregate pair
	$(\widehat\bunu,\widehat\buZ)$, up to $du$-a.e.\
	indistinguishability.
\end{theorem}

\begin{proof}
	Suppose there exist two GMFEs, denoted by
	$(\bupi,\buZ)$ and
	$(\widetilde\bupi,\widetilde\buZ)$. By definition,
	\begin{align*}
		\int_{u\in I}
		\left(
			J^u(\bpi^u;\bZ^u)
			-
			J^u(\widetilde\bpi^u;\bZ^u)
		\right)\,du
		\geq0,
		\quad
		\int_{u\in I}
		\left(
			J^u(\widetilde\bpi^u;\widetilde\bZ^u)
			-
			J^u(\bpi^u;\widetilde\bZ^u)
		\right)\,du
		\geq0.
	\end{align*}
	Moreover, by Definition~\ref{def_Nash_GMFG}, $\bZ^u=\bZ^{\bupi,u}$ and $\widetilde\bZ^u = \bZ^{\widetilde\bupi,u}$
	for a.e.\ $u\in I$. Notice that
	$ \buZ^{\bupi} = \buZ^{\bunu^\bupi}, $ where $\bunu^\bupi$ is the state-action distribution profile generated by $\bupi$. Likewise, $ \buZ^{\widetilde\bupi} = \buZ^{\bunu^{\widetilde\bupi}}$.	
	
	Since Assumption~\ref{assp_LLmonotone} makes the transition
	kernel independent of the aggregate, the state-action law generated
	by a fixed policy does not change when the same policy is evaluated
	against a different aggregate path. The terminal reward terms also
	cancel in each comparison because $r_T^u$ is independent of the
	aggregate and the terminal state law generated by a fixed policy is
	unchanged. Therefore, the following reward differences can be
	written using the equilibrium state-action distribution profiles of
	$\bupi$ and $\widetilde\bupi$. For all $u\in I$, let $\nu_t^u:=\nu_t^{\bupi,u}$ and  $\widetilde\nu_t^u := \nu_t^{\widetilde\bupi,u}$.
	Combining the two inequalities above, we obtain
	\begin{align*}
		0
		&\leq
		\int_{u\in I}
		\Big(
			J^u(\bpi^u;\bZ^{\bupi,u})
			-
			J^u(\bpi^u;\bZ^{\widetilde\bupi,u})
			+
			J^u(\widetilde\bpi^u;
			\bZ^{\widetilde\bupi,u})
			-
			J^u(\widetilde\bpi^u;
			\bZ^{\bupi,u})
		\Big)\,du
		\\
		&=
		\int_{u\in I}
		\sum_{t\in\cT}
		\sum_{x\in\cX}
		\Big[
			\left(
				r_t^u(x,\pi_t^u(x),Z_t^{\bupi,u})
				-
				r_t^u(x,\pi_t^u(x),
				Z_t^{\widetilde\bupi,u})
			\right)
			\mu_t^u(x)
		\\
		&\hspace{11em}
			+
			\left(
				r_t^u(x,\widetilde\pi_t^u(x),
				Z_t^{\widetilde\bupi,u})
				-
				r_t^u(x,\widetilde\pi_t^u(x),
				Z_t^{\bupi,u})
			\right)
			\widetilde\mu_t^u(x)
		\Big]\,du
		\\
		&=
		\int_{u\in I}
		\sum_{t\in\cT}
		\sum_{x\in\cX}
		\int_\cA
		\Big[
			\left(
				r_t^u(x,a,Z_t^{\bupi,u})
				-
				r_t^u(x,a,Z_t^{\widetilde\bupi,u})
			\right)
			\mu_t^u(x)\delta_{\pi_t^u(x)}(da)
		\\
		&\hspace{11em}
			+
			\left(
				r_t^u(x,a,Z_t^{\widetilde\bupi,u})
				-
				r_t^u(x,a,Z_t^{\bupi,u})
			\right)
			\widetilde\mu_t^u(x)
			\delta_{\widetilde\pi_t^u(x)}(da)
		\Big]\,du
		\\
		&=
		\int_{u\in I}
		\sum_{t\in\cT}
		\sum_{x\in\cX}
		\int_\cA
		\Big[
			\left(
				r_t^u(x,a,Z_t^{\bunu,u})
				-
				r_t^u(x,a,Z_t^{\widetilde\bunu,u})
			\right)
			\nu_t^u(x,da)
		\\
		&\hspace{11em}
			+
			\left(
				r_t^u(x,a,Z_t^{\widetilde\bunu,u})
				-
				r_t^u(x,a,Z_t^{\bunu,u})
			\right)
			\widetilde\nu_t^u(x,da)
		\Big]\,du
		\\
		&=
		\int_{u\in I}
		\sum_{t\in\cT}
		\sum_{x\in\cX}
		\int_\cA
		\left(
			r_t^u(x,a,Z_t^{\bunu,u})
			-
			r_t^u(x,a,Z_t^{\widetilde\bunu,u})
		\right)
		\left[
			\nu_t^u(x,da)
			-
			\widetilde\nu_t^u(x,da)
		\right]\,du.
	\end{align*}
By Assumption~\ref{assp_LLmonotone}, the right-hand side is non-positive. Since it is also non-negative, it must vanish. By the equality case in Assumption~\ref{assp_LLmonotone}, it follows that $\bunu=\widetilde\bunu$, and therefore $\buZ^{\bunu} = \buZ^{\widetilde\bunu}$ for a.e.\ $u\in I$ and every $t\in\cT$. Hence, the graphon state-action flow and the associated aggregate are unique, up to $du$-a.e.\ indistinguishability.
\end{proof}

\subsection{Uniqueness by contraction}

\begin{theorem}\label{thm_UniqueNash_Banach}
	Suppose Assumptions~\ref{aasp_graphon-assumption}--\ref{assp_u_cont}
	hold. Moreover, suppose that
	\begin{equation}
		\label{eq:thm-unique-constants}
		C_\Psi
		:=
		\max\{K_{V,Z},K_{\mu,Z}\}
		\frac{
			\|G\|_2
			\left(
				L_\varphi L_{a,V}\sqrt{|\cX|}
				+
				C_\varphi|\cX|
			\right)
		}{
			1-L_\varphi L_{a,Z}\|G\|_2
		}
		<1,
		\qquad\mbox{and}\qquad
		L_\varphi L_{a,Z}\|G\|_2<1,
	\end{equation}
	where we recall that $L_{a,Z}$ and $L_{a,V}$ are defined in
	Lemma~\ref{lemma_ZtoA_lip}, $K_{V,Z}$ is defined in
	Lemma~\ref{lemma_ztoVz_lip}, $K_{\mu,Z}$ is defined in
	Lemma~\ref{lemma_ZtomuZ_lip}, and $C_\varphi,L_\varphi$ are
	defined in Assumption~\ref{assp_varphi_lip}.

	Then there exists a unique graphon mean-field flow and aggregate
	pair. Moreover, the associated canonical Bellman GMFE policy is
	unique. Equivalently, the GMFE is unique up to modifications of
	the policy at state-time pairs carrying zero equilibrium mass.
\end{theorem}

\begin{proof}%
	By Theorem~\ref{thm_GMFE_sys}, every GMFE satisfies the
	forward-backward system~\eqref{eq_sys_V_mu}, with the optimality
	condition imposed on the support of its equilibrium state
	distribution. Under Assumption~\ref{assp_diff_concave_faz}, the
	Bellman maximizer is unique. Thus, any GMFE policy can be replaced
	at zero-mass state-time pairs by the unique Bellman maximizer
	without changing its induced mean-field flow, aggregate, or payoff.
	It is therefore sufficient to prove the existence and uniqueness of
	a canonical bounded solution $(\buV,\bumu)$ in $L^2(I)^{\ocT\times\cX}\times\bucM$
	to the system~\eqref{eq_sys_V_mu} in three steps:
	\vspace{-0.1cm}
	\begin{enumerate}
		\item We define the solution space and restrict
		      $(\buV,\bumu)$ to be in a feasible set
		      $\cK_{C_1}$ as in~\eqref{def_KC1}.

		\item We show that, for every
		      $(\buV,\bumu)\in\cK_{C_1}$, the map
		      $\Phi^{(\buV,\bumu)}$ defined as
		      in~\eqref{eq_def-Phi-for-Z} has a unique fixed point,
		      denoted by
		      $\widehat\buZ^{(\buV,\bumu)}$.

		\item We prove, by the Banach fixed point theorem, that the map
		      \[
			      \Psi:
			      \cK_{C_1}\longrightarrow\cK_{C_1},
			      \qquad
			      (\buV,\bumu)
			      \longmapsto
			      \left(
				      \buV^{\widehat\buZ^{(\buV,\bumu)}},
				      \bumu^{\widehat\buZ^{(\buV,\bumu)}}
			      \right),
		      \]
		      admits a unique fixed point
		      $(\widehat\buV,\widehat\bumu)$.
	\end{enumerate}

	The first two steps are identical to the proof of
	Theorem~\ref{thm_Vmu}. We focus on Step~3, using the extra
	assumption~\eqref{eq:thm-unique-constants} on the coefficients.

        Let $(\buV,\bumu), (\widetilde\buV,\widetilde\bumu) \in\cK_{C_1}$,
	and recall~\eqref{eq_Psi_lipVmuZ}. Moreover, for the right-hand
	side of~\eqref{eq_Psi_lipVmuZ}, we have
	\begin{align*}
		&
		\left\|
			\widehat\buZ^{(\buV,\bumu)}
			-
			\widehat\buZ^{(\widetilde\buV,\widetilde\bumu)}
		\right\|_{\infty,2}
		\\
		&\leq
		\max_{t\in\cT}
		\Bigg(
		\int_{u\in I}
		\Bigg|
			\int_{v\in I}
			G(u,v)
			\sum_{x\in\cX}
			\Big(
				\varphi_t^u
				\left(
					x,
					\pi_t^{*,
					\widehat\buZ^{(\buV,\bumu)},
					\buV,v}(x)
				\right)
		\\
		&\hspace{16em}
				-
				\varphi_t^u
				\left(
					x,
					\pi_t^{*,
					\widehat\buZ^{(\widetilde\buV,
					\widetilde\bumu)},
					\widetilde\buV,v}(x)
				\right)
			\Big)
			\mu_t^v(x)\,dv
		\Bigg|^2du
		\Bigg)^{1/2}
		\\
		&\quad+
		\max_{t\in\cT}
		\Bigg(
		\int_{u\in I}
		\Bigg|
			\int_{v\in I}
			G(u,v)
			\sum_{x\in\cX}
			\varphi_t^u
			\left(
				x,
				\pi_t^{*,
				\widehat\buZ^{(\widetilde\buV,
				\widetilde\bumu)},
				\widetilde\buV,v}(x)
			\right)
			\left(
				\mu_t^v(x)
				-
				\widetilde\mu_t^v(x)
			\right)\,dv
		\Bigg|^2du
		\Bigg)^{1/2}.
	\end{align*}

	For the first term, since $\mu_t^v$ is a probability distribution
	and $\varphi$ is $L_\varphi$-Lipschitz continuous,
	\begin{align*}
		&
		\left|
			\sum_{x\in\cX}
			\Big(
				\varphi_t^u
				(x,\pi_t^{*,\widehat\buZ^{(\buV,\bumu)},
				\buV,v}(x))
				-
				\varphi_t^u
				(x,\pi_t^{*,\widehat\buZ^{(\widetilde\buV,
				\widetilde\bumu)},\widetilde\buV,v}(x))
			\Big)
			\mu_t^v(x)
		\right|
		\\
		&\leq
		L_\varphi
		\max_{x\in\cX}
		\left|
			\pi_t^{*,\widehat\buZ^{(\buV,\bumu)},
			\buV,v}(x)
			-
			\pi_t^{*,\widehat\buZ^{(\widetilde\buV,
			\widetilde\bumu)},\widetilde\buV,v}(x)
		\right|.
	\end{align*}

	Moreover, by Lemma~\ref{lemma_ZtoA_lip}, for every
	$t\in\cT$,
	\begin{align*}
		&
		\left(
			\int_I
			\max_{x\in\cX}
			\left|
				\pi_t^{*,\widehat\buZ^{(\buV,\bumu)},
				\buV,v}(x)
				-
				\pi_t^{*,\widehat\buZ^{(\widetilde\buV,
				\widetilde\bumu)},\widetilde\buV,v}(x)
			\right|^2dv
		\right)^{1/2}
		\\
		&\qquad \leq
		L_{a,Z}
		\left\|
			\widehat Z_t^{(\buV,\bumu)}
			-
			\widehat Z_t^{(\widetilde\buV,
			\widetilde\bumu)}
		\right\|_{L^2(I)}
	+
		L_{a,V}
		\left(
			\int_I
			\max_{x\in\cX}
			|V_{t+1}^v(x)
			-
			\widetilde V_{t+1}^v(x)|^2\,dv
		\right)^{1/2}
		\\
		&\qquad \leq
		L_{a,Z}
		\left\|
			\widehat\buZ^{(\buV,\bumu)}
			-
			\widehat\buZ^{(\widetilde\buV,
			\widetilde\bumu)}
		\right\|_{\infty,2}
		+
		L_{a,V}\sqrt{|\cX|}
		\|\buV-\widetilde\buV\|_{\infty,\infty,2}.
	\end{align*}
	Using the Hilbert--Schmidt bound for the graphon operator, the
	first term in the preceding decomposition is therefore bounded by
	\[
		L_\varphi\|G\|_2
		\left(
			L_{a,Z}
			\left\|
				\widehat\buZ^{(\buV,\bumu)}
				-
				\widehat\buZ^{(\widetilde\buV,
				\widetilde\bumu)}
			\right\|_{\infty,2}
			+
			L_{a,V}\sqrt{|\cX|}
			\|\buV-\widetilde\buV\|_{\infty,\infty,2}
		\right).
	\]

	For the second term, using
	$|\varphi_t^u(x,a)|\leq C_\varphi$, we have
	\[
		\left|
			\sum_{x\in\cX}
			\varphi_t^u(x,a_x)
			\left(
				\mu_t^v(x)
				-
				\widetilde\mu_t^v(x)
			\right)
		\right|
		\leq
		C_\varphi
		\sum_{x\in\cX}
		|\mu_t^v(x)-\widetilde\mu_t^v(x)|.
	\]
Hence, the second term is bounded by $C_\varphi|\cX|\|G\|_2 \|\bumu-\widetilde\bumu\|_{\infty,\infty,2}$. 

Substituting these estimates and rearranging terms, recalling that $L_\varphi L_{a,Z}\|G\|_2<1$, yields
	\begin{align*}
		\left\|
			\widehat\buZ^{(\buV,\bumu)}
			-
			\widehat\buZ^{(\widetilde\buV,
			\widetilde\bumu)}
		\right\|_{\infty,2}
		&\leq
		\frac{\|G\|_2}{
			1-L_\varphi L_{a,Z}\|G\|_2
		}
		\Big(
			L_\varphi L_{a,V}\sqrt{|\cX|}
			\|\buV-\widetilde\buV\|_{\infty,\infty,2}
			+
			C_\varphi|\cX|
			\|\bumu-\widetilde\bumu\|_{\infty,\infty,2}
		\Big)
		\\
		&\leq
		\frac{
			\|G\|_2
			\left(
				L_\varphi L_{a,V}\sqrt{|\cX|}
				+
				C_\varphi|\cX|
			\right)
		}{
			1-L_\varphi L_{a,Z}\|G\|_2
		}
		\left\|
			(\buV,\bumu)
			-
			(\widetilde\buV,\widetilde\bumu)
		\right\|_{\cK}.
	\end{align*}

	Finally, using the Lipschitz continuity derived
	in~\eqref{eq_Psi_lipVmuZ}, we have
	\begin{align*}
		\left\|
			\Psi(\buV,\bumu)
			-
			\Psi(\widetilde\buV,\widetilde\bumu)
		\right\|_{\cK}
		\leq
		\max\{K_{V,Z},K_{\mu,Z}\}
		\frac{
			\|G\|_2
			\left(
				L_\varphi L_{a,V}\sqrt{|\cX|}
				+
				C_\varphi|\cX|
			\right)
		}{
			1-L_\varphi L_{a,Z}\|G\|_2
		}
		\left\|
			(\buV,\bumu)
			-
			(\widetilde\buV,\widetilde\bumu)
		\right\|_{\cK}.
	\end{align*}
	Thus, provided the constant $C_\Psi$ defined
	in~\eqref{eq:thm-unique-constants} is strictly less than $1$, the
	mapping $\Psi$ is a contraction. By the Banach fixed point theorem,
	$\Psi$ admits a unique fixed point $(\widehat\buV,\widehat\bumu) \in\cK_{C_1}$. 

Let $\widehat\buZ := \widehat\buZ^{(\widehat\buV,\widehat\bumu)}$ and define the canonical Bellman policy by $\widehat\pi_t^u(x) = \widehat a_t^u \left( \widehat Z_t^u, \widehat V_{t+1}^u, x \right)$, for all $t\in\cT, x\in\cX$.
Then
$(\widehat\buV,\widehat\bumu,\widehat\bupi,\widehat\buZ)$
satisfies~\eqref{eq_sys_V_mu}, and hence
$(\widehat\bupi,\widehat\buZ)$ is a GMFE.

Conversely, any GMFE can be modified at zero-mass state-time pairs
by replacing its policy with the unique Bellman maximizer. This does
not change its mean-field flow, aggregate, or payoff, and the resulting
canonical representative yields a fixed point of $\Psi$. By uniqueness
of this fixed point, every GMFE has the same mean-field flow and
aggregate and agrees with $\widehat\bupi$ on the support of the
equilibrium flow. This proves the result.
\end{proof}

\section{Approximate Equilibrium in Finite Player Games}\label{sec_epsNash}

We consider now a large finite network game with heterogeneous
interactions. Let $N\in\mathbb{N}$ be the number of players.
An $N$-player network game is specified by the tuple
$(\cT,\zeta^N,\cX,\cA,P,r,r_T,\varphi,\mu_{\mathrm{init}})$,
where the interaction structure is given by a symmetric matrix
$\zeta^N=(\zeta_{ij}^N)_{i,j\in[N]}$ and
$\zeta_{ij}^N\in[0,1]$ quantifies the influence of player $j$ on
player $i$. In this section, we assume that
$(\mu_{\mathrm{init}},P,r,r_T,\varphi)$
are the same for every player, and heterogeneity arises only
through the interaction matrix $\zeta^N$.

To study the large-$N$ limit, we associate to each interaction matrix
$\zeta^N$ a step-graphon $G_N$ defined on $I=[0,1]$:
\begin{equation}\label{eq_step_graphon}
	G_N(u,v)
	:=
	\sum_{i,j=1}^N
	\zeta^N_{ij}\,
	\mathds{1}_{u\in\left(\frac{i-1}{N},\frac{i}{N}\right]}
	\mathds{1}_{v\in\left(\frac{j-1}{N},\frac{j}{N}\right]}.
\end{equation}

\begin{assumption}\label{assp_cutnorm_convergence}
	The sequence of step-graphons $(G_N)_{N\in\N}$ converges to a continuous graphon
	$G\in\cW$ in the operator norm $\|\cdot\|_{\infty\to1}$.
\end{assumption}
Note that equivalently, up to universal constants, one may formulate this assumption in the cut norm $\|\cdot\|_\square$.

By \cite[Theorem 11.59]{lovasz2012}, a sufficient condition for
Assumption~\ref{assp_cutnorm_convergence} is that the interaction
matrices $\zeta^N$ arise from a sequence of dense graphs
$(\G_N)_{N\in\N}$ that converges to a continuous graphon $G$ in the
sense of \cite{lovasz2012}.\footnote{
	More precisely, $(\G_N)$ is said to converge to $G$ if, for every
	finite simple graph $F=(V_F,E_F)$, the homomorphism densities
	satisfy $t(F,\G_N)\to t(F,G)$ as $N\to\infty$. Here,
	$t(F,\G_N)$ denotes the probability that a uniformly random map
	$\phi:V_F\to V(\G_N)$ sends every edge of $F$ to an edge of
	$\G_N$, and
	\[
		t(F,G)
		=
		\int_{[0,1]^k}
		\prod_{\{i,j\}\in E_F}
		G(x_i,x_j)\,dx_1\cdots dx_k,
	\]
	with $k=|V_F|$, is the homomorphism density of $F$ into the
	graphon $G$, i.e., the probability that a random map from $V_F$
	to $[0,1]$ preserves all edges of $F$.
}
This is equivalent to the existence of a relabeling of the vertices of
$\G_N$ such that the associated step-graphons $G_N$ satisfy
$\|G_N-G\|_\square\to0$.

We recall that the set of policies for one player is $\bfPi=\cA^{\cT\times\cX}$. A policy $\bpi^i\in\bfPi$, $(t,x)\longmapsto \pi_t^i(x):=\pi^i(t,x)\in\cA$, provides an action for each state $x$ at each decision time $t$ for player $i\in[N]:=\{1,\dots,N\}$. Given policies $\bm{\pi}^{(N)} = (\bpi^1,\dots,\bpi^N) \in\bfPi^N$,
we define the aggregate for player $i$ as
\begin{equation}
	\label{eq:ZNpii}
	Z_t^{N,\bm{\pi}^{(N)},i}
	:=
	\frac{1}{N}
	\sum_{j=1}^N
	\zeta_{ij}^N\,
	\varphi_t
	\left(
		X_t^j,\pi_t^j(X_t^j)
	\right),
	\qquad
	i\in[N],\quad t\in\cT,
\end{equation}
where player $i$ evolves according to
\begin{equation*}
	X_0^i\sim\mu_{\mathrm{init}},
	\qquad
	a_t^i=\pi_t^i(X_t^i),
	\qquad
	X_{t+1}^i
	\sim
	P_t
	\left(
		\cdot\mid
		X_t^i,a_t^i,
		Z_t^{N,\bm{\pi}^{(N)},i}
	\right),
	\qquad
	i\in[N],\quad t\in\cT.
\end{equation*}
Throughout this section, we assume that
$X_0^1,\dots,X_0^N$ are independent with common distribution
$\mu_{\mathrm{init}}$ and that, conditionally on the current state
vector, the players' next-state transitions are independent.
Note that the evolution of the players' states is coupled through the
aggregate. The total reward of player $i$ is defined, for
$\bpi^i\in\bfPi$ and $\bpi^{-i}\in\bfPi^{N-1}$, by
\begin{equation*}
	\cJ^{N,i}(\bpi^i;\bpi^{-i})
	:=
	\E\left[
		\sum_{t\in\cT}
		r_t
		\left(
			X_t^i,a_t^i,
			Z_t^{N,\bm{\pi}^{(N)},i}
		\right)
		+
		r_T(X_T^i)
	\right].
\end{equation*}

\begin{definition}[$(\epsilon,p)$-Nash equilibrium on a large subset]
	Fix $N\in\mathbb{N}$ and $\epsilon,p>0$. A vector of policies $\bm{\widehat\pi}^{(N)} = (\widehat\bpi^1,\dots,\widehat\bpi^N) \in\bfPi^N$
	is an $(\epsilon,p)$-Nash equilibrium on a large subset if there
	exists a subset $\cI_N\subseteq[N]$ with
	$|\cI_N|\geq(1-p)N$ such that
	\[
		\sup_{\widetilde\bpi\in\bfPi}
		\cJ^{N,i}
		(\widetilde\bpi;\widehat\bpi^{-i})
		-
		\cJ^{N,i}
		(\bm{\widehat\pi}^{(N)})
		\leq\epsilon,
		\qquad
		\forall\,i\in\cI_N.
	\]
\end{definition}

We show that, under some regularity assumptions, a GMFE policy in the
GMFG, Definition~\ref{def_Nash_GMFG}, induces an
$(\epsilon,p)$-Nash equilibrium on a large subset of the
$N$-player network game, for any $\epsilon,p>0$ when $N$ is large
enough.

For an $N$-player policy profile
$
	\bm{\pi}^{(N)}
	=
	(\bpi^1,\dots,\bpi^N)
	\in\bfPi^N,
$
we define the corresponding graphon step policy as
\begin{equation*}
	\pi_t^{N,\bm{\pi}^{(N)},u}(x)
	:=
	\sum_{i\in[N]}
	\mathds{1}_{u\in
	\left(\frac{i-1}{N},\frac{i}{N}\right]}
	\pi_t^i(x),
	\qquad
	x\in\cX,\quad u\in I,\quad t\in\cT.
\end{equation*}

Similarly, the associated empirical state distribution can be extended
to all $u\in I$ by
\begin{equation}\label{eq_muN}
	\mu_t^{N,\bm{\pi}^{(N)},u}
	:=
	\sum_{i\in[N]}
	\mathds{1}_{u\in
	\left(\frac{i-1}{N},\frac{i}{N}\right]}
	\delta_{X_t^i},
	\qquad
	u\in I,\quad t\in\ocT.
\end{equation}

Let $L_\pi$ be a positive constant. We denote by
$\buPi^{L_\pi}$ the set of policy profiles admitting a chosen
representative such that, for every $x\in\cX$ and $t\in\cT$, the map
$u\longmapsto\pi_t^u(x)$
is Lipschitz continuous with Lipschitz constant $L_\pi$. Elements of
$\buPi^{L_\pi}$ are used with this representative, so pointwise
evaluation is well-defined. Conversely, from such a regular graphon
policy representative, we can define a policy for the finite-player
game. We define the sampling map
$
	\Lambda^N:
	\buPi^{L_\pi}\longrightarrow\bfPi^N
$
by
\[
	\Lambda^N(\bupi)
	=
	\left(
		\bpi^{\frac1N},
		\bpi^{\frac2N},
		\dots,
		\bpi^{\frac NN}
	\right)
	\in\bfPi^N,
	\qquad
	\bupi\in\buPi^{L_\pi}.
\]

We now state an approximation result for the convergence of the state
and aggregate processes. We recall that the notations
$\bunu^\bupi$ and $\buZ^\bunu$ are defined in
Section~\ref{sec:unique-mono}. This Lipschitz condition is imposed only
on the background graphon policy profile; unilateral finite-player
deviations below range over the full finite-player policy space
$\bfPi$.

\begin{proposition}\label{lemma_XZconvergence}
	Suppose
	Assumptions~\ref{aasp_graphon-assumption}--
	\ref{assp_varphi_lip}
	and~\ref{assp_cutnorm_convergence} hold. Let
	$\bupi\in\buPi^{L_\pi}$.
	Let $C_h,L_h>0$. Define $\cH$ as the set of functions
	$h:\cX\to\R$ bounded by $C_h$, and let
	$\widetilde\cH$ be the family of measurable functions
$\mathfrak h: \cX\times\RR\longrightarrow\R$
	bounded by $C_h$ and Lipschitz continuous with respect to the
	second argument, with Lipschitz constant $L_h$ uniformly in
	$x\in\cX$.

	For any $\epsilon,p>0$, there exists $N_0$ such that, for all
	$N\geq N_0$, there exists a subset
	$\cI_N\subseteq[N]$ with
	$|\cI_N|\geq(1-p)N$ such that, for every $i\in\cI_N$, the
	following two estimates hold simultaneously:
	\begin{equation}\label{eq_X_convergence}
		\sup_{\widetilde\bpi\in\bfPi}
		\sup_{h\in\cH}
		\left|
			\E\left[h(X_t^i)\right]
			-
			\E\left[h(X_t^{\ioN})\right]
		\right|
		<\epsilon,
		\qquad
		\forall\,t\in\ocT,
	\end{equation}
	and
	\begin{equation}\label{eq_XZ_convergence}
		\sup_{\widetilde\bpi\in\bfPi}
		\sup_{\mathfrak h\in\widetilde\cH}
		\left|
			\E\left[
				\mathfrak h
				\left(
					X_t^i,
					Z_t^{N,(\widetilde\bpi,
					\Lambda^N(\bupi)^{-i}),i}
				\right)
			\right]
			-
			\E\left[
				\mathfrak h
				\left(
					X_t^{\ioN},
					Z_t^{\bunu^\bupi,\ioN}
				\right)
			\right]
		\right|
		<\epsilon,
		\qquad
		\forall\,t\in\cT.
	\end{equation}

	Here, $X^i$ denotes the state of player $i$ in the $N$-player
	game when player $i$ uses $\widetilde\bpi\in\bfPi$ and the
	other players use $\Lambda^N(\bupi)^{-i}$, while
	$X^{\ioN}$ denotes the state of the player with label $i/N$ in
	the graphon game when this player uses $\widetilde\bpi$ and the
	population uses $\bupi$. Their respective dynamics are
	\begin{align*}
		X_0^i
		&\sim\mu_{\mathrm{init}},
		\qquad
		a_t^i=\widetilde\pi_t(X_t^i),
		\qquad
		X_{t+1}^i
		\sim
		P_t
		\left(
			\cdot\mid
			X_t^i,a_t^i,
			Z_t^{N,(\widetilde\bpi,
			\Lambda^N(\bupi)^{-i}),i}
		\right),
		\qquad t\in\cT,
		\\
		X_0^{\ioN}
		&\sim\mu_{\mathrm{init}},
		\qquad
		a_t^{\ioN}
		=
		\widetilde\pi_t(X_t^{\ioN}),
		\qquad
		X_{t+1}^{\ioN}
		\sim
		P_t
		\left(
			\cdot\mid
			X_t^{\ioN},a_t^{\ioN},
			Z_t^{\bunu^\bupi,\ioN}
		\right),
		\qquad t\in\cT.
	\end{align*}
\end{proposition}

\begin{proof}
	See Section~\ref{proof:lemma_XZconvergence}.
\end{proof}

It is convenient to rewrite the total reward~\eqref{eq:reward} in terms
of the state-action distribution profile. Let $\bunu = (\nu_t^u)_{u\in I,t\in\cT}$, with $\nu_t^u\in\cP(\cX\times\cA)$,
and let $\buZ^\bunu$ be the corresponding graphon-weighted aggregate
defined in~\eqref{def_Znu}. The total reward of a representative player
with label $u\in I$ under policy $\bpi^u$ is
\begin{equation*}
	\cJ^{\bunu,u}(\bpi^u)
	:=
	\E\left[
		\sum_{t\in\cT}
		r_t
		\left(
			X_t^u,a_t^u,
			Z_t^{\bunu,u}
		\right)
		+
		r_T(X_T^u)
	\right],
	\qquad
	u\in I,
\end{equation*}
subject to
\begin{equation*}
	X_0^u\sim\mu_{\mathrm{init}},
	\qquad
	a_t^u=\pi_t^u(X_t^u),
	\qquad
	X_{t+1}^u
	\sim
	P_t
	\left(
		\cdot\mid
		X_t^u,a_t^u,
		Z_t^{\bunu,u}
	\right),
	\qquad
	t\in\cT.
\end{equation*}

We have the following approximation result, which states that unilateral
deviations in the finite-player game and in the graphon game yield very
similar values.

\begin{theorem}\label{thm_JioN_Jmu}
	Suppose
	Assumptions~\ref{aasp_graphon-assumption}--
	\ref{assp_varphi_lip}
	and~\ref{assp_cutnorm_convergence} hold. Let
	$\bupi\in\buPi^{L_\pi}$. Let
	$\bunu^\bupi = (\nu_t^{\bupi,u})_{u\in I,t\in\cT}$, $\nu_t^{\bupi,u}\in\cP(\cX\times\cA)$, be the corresponding state-action distribution profile generated
	by $\bupi$.

	For any $\epsilon,p>0$, there exists $N_0$ such that, for all
	$N\geq N_0$, there exists a subset
	$\cI_N\subseteq[N]$ with
	$|\cI_N|\geq(1-p)N$ such that, for every $i\in\cI_N$,
	\[
		\sup_{\widetilde\bpi\in\bfPi}
		\left|
			\cJ^{N,i}
			(\widetilde\bpi;\Lambda^N(\bupi)^{-i})
			-
			\cJ^{\bunu^\bupi,\ioN}
			(\widetilde\bpi)
		\right|
		<\epsilon.
	\]
\end{theorem}

\begin{proof}
Define the shorthand notation $r_t^\pi(x,z) := r_t(x,\pi(x),z)$, $\pi\in\cA^\cX$. Fix $\epsilon,p>0$ and set $\delta := \frac{\epsilon}{T+1}$.
	By Proposition~\ref{lemma_XZconvergence}, applied with
	$C_h=C_r$, $L_h=L_{r,Z}$, and tolerance $\delta$, there exists
	$N_0$ such that, for all $N\geq N_0$, there exists a subset
	$\cI_N\subseteq[N]$ with
	$|\cI_N|\geq(1-p)N$ for which
	\eqref{eq_X_convergence} and~\eqref{eq_XZ_convergence} hold
	simultaneously for every $i\in\cI_N$ and all indicated times.

	For every $i\in\cI_N$ and every
	$\widetilde\bpi\in\bfPi$, the triangle inequality yields
\begin{align*}
	&
	\left|\,
	\cJ^{N,i}
	(\widetilde\bpi;\Lambda^N(\bupi)^{-i})
	-
	\cJ^{\bunu^\bupi,\ioN}(\widetilde\bpi)
	\,\right|
	\\
	&=
	\Bigg|
	\E\Bigg[
		\sum_{t\in\cT}
		r_t
		\left(
			X_t^i,
			\widetilde\pi_t(X_t^i),
			Z_t^{N,(\widetilde\bpi,\Lambda^N(\bupi)^{-i}),i}
		\right)
		+
		r_T(X_T^i)
	\Bigg]
	\\
	&\qquad\qquad-
	\E\Bigg[
		\sum_{t\in\cT}
		r_t
		\left(
			X_t^{\ioN},
			\widetilde\pi_t(X_t^{\ioN}),
			Z_t^{\bunu^\bupi,\ioN}
		\right)
		+
		r_T(X_T^{\ioN})
	\Bigg]
	\Bigg|
	\\
	&=
	\Bigg|
	\E\Bigg[
		\sum_{t\in\cT}
		r_t^{\widetilde\pi_t}
		\left(
			X_t^i,
			Z_t^{N,(\widetilde\bpi,\Lambda^N(\bupi)^{-i}),i}
		\right)
		+
		r_T(X_T^i)
	\Bigg]
	\\
	&\qquad\qquad-
	\E\Bigg[
		\sum_{t\in\cT}
		r_t^{\widetilde\pi_t}
		\left(
			X_t^{\ioN},
			Z_t^{\bunu^\bupi,\ioN}
		\right)
		+
		r_T(X_T^{\ioN})
	\Bigg]
	\Bigg|
	\\
	&\leq
	\sum_{t\in\cT}
	\left|
		\E\left[
			r_t^{\widetilde\pi_t}
			\left(
				X_t^i,
				Z_t^{N,(\widetilde\bpi,\Lambda^N(\bupi)^{-i}),i}
			\right)
		\right]
		-
		\E\left[
			r_t^{\widetilde\pi_t}
			\left(
				X_t^{\ioN},
				Z_t^{\bunu^\bupi,\ioN}
			\right)
		\right]
	\right|
	\\
	&\qquad+
	\left|
		\E\left[r_T(X_T^i)\right]
		-
		\E\left[r_T(X_T^{\ioN})\right]
	\right|
	\\
	&<
	T\delta+\delta
	=
	\epsilon.
\end{align*}
Indeed, for every $t\in\cT$, the function
$(x,z)\longmapsto r_t^{\widetilde\pi_t}(x,z)$
is bounded by $C_r$ and is Lipschitz continuous with respect to $z$
with Lipschitz constant $L_{r,Z}$, uniformly in
$\widetilde\bpi\in\bfPi$. Hence, each running-reward difference is bounded
using~\eqref{eq_XZ_convergence}. The terminal reward
$x\mapsto r_T(x)$ is bounded by $C_r$, so the terminal difference is
bounded using~\eqref{eq_X_convergence} at time $T$.

Since the subset $\cI_N$ and all the preceding estimates are uniform
with respect to $\widetilde\bpi\in\bfPi$, taking the supremum over
	$\widetilde\bpi$ proves the result.
\end{proof}

As a corollary of Theorem~\ref{thm_JioN_Jmu}, we obtain the main approximation result of this section.

\begin{theorem}\label{thm_epsNash}
	Suppose Assumptions~\ref{aasp_graphon-assumption}-\ref{assp_varphi_lip}
	and~\ref{assp_cutnorm_convergence} hold.
	Let $\widehat\bupi\in\buPi^{L_\pi}$ be a graphon mean-field equilibrium policy.
	For any $\epsilon,p>0$, there exists $N_0$ such that for all
	$N\geq N_0$, $\Lambda^N(\widehat\bupi)$ is an
	$(\epsilon,p)$-Nash equilibrium for the $N$-player game, i.e.,
	there exists a subset $\cI_N\subseteq[N]$ of indices with
	cardinality $|\cI_N|\geq(1-p)N$ such that
	\[
		\sup_{\widetilde\bpi\in\bfPi}
		\cJ^{N,i}
		(\widetilde\bpi;\Lambda^N(\widehat\bupi)^{-i})
		-
		\cJ^{N,i}(\Lambda^N(\widehat\bupi))
		\leq\epsilon,
		\qquad
		\forall\,i\in\cI_N.
	\]
\end{theorem}

\begin{proof}
Let
$
	\bunu
	=
	(\nu_t^{\widehat\bupi,\widehat\bumu,u})_{u\in I,t\in\cT},
$
with
$\nu_t^{\widehat\bupi,\widehat\bumu,u}\in\cP(\cX\times\cA)$,
denote the state-action distribution profile associated with the
equilibrium policy $\widehat\bupi$, where $\widehat\bumu$ is the corresponding
state-distribution sequence induced from the initial distribution
$\mu_{\mathrm{init}}$.

We first verify that the chosen Lipschitz representative $\widehat\bupi$ is
optimal at every label $u\in I$, including the sampled labels $i/N$.
Define
\[
	R(u)
	:=
	\sup_{\widetilde\bpi\in\bfPi}
	\left[
		\cJ^{\bunu,u}(\widetilde\bpi)
		-
		\cJ^{\bunu,u}(\widehat\bupi^u)
	\right],
	\qquad u\in I.
\]
By Assumption~\ref{aasp_graphon-assumption} and the boundedness of
$\varphi$,
\[
	|Z_t^{\bunu,u}-Z_t^{\bunu,v}|
	\leq
	C_\varphi
	\|G(u,\cdot)-G(v,\cdot)\|_{L^1(I)}
	\leq
	C_\varphi\omega^G(|u-v|),
\]
so the map $u\mapsto Z_t^{\bunu,u}$ is continuous for every
$t\in\cT$. Since the state space and the horizon are finite, the
transition kernel and rewards are continuous in the action and
Lipschitz continuous in the aggregate, and
$u\mapsto\widehat\bpi^u$ is Lipschitz continuous, it follows by induction
that
$
	(u,\widetilde\bpi)
	\longmapsto
	\cJ^{\bunu,u}(\widetilde\bpi)
$
is continuous on $I\times\bfPi$. Moreover,
$\bfPi=\cA^{\cT\times\cX}$ is compact. Hence, by the maximum theorem,
$R$ is continuous on $I$.
By the GMFE property, $R(u)=0$ for a.e.\ $u\in I$. Since $R$ is
continuous and nonnegative, we conclude that $R(u)=0$, for all $u\in I$.

In particular, $\widehat\bupi^{\ioN}$ is optimal against the equilibrium
aggregate at the sampled label $i/N$.

By Theorem~\ref{thm_JioN_Jmu}, applied twice with tolerance
$\epsilon/2$ and exceptional proportion $p/2$, there exist subsets
$\cI_N^{(1)},\cI_N^{(2)}\subseteq[N]$ satisfying
\[
	|\cI_N^{(1)}|
	\geq
	\left(1-\frac p2\right)N,
	\qquad
	|\cI_N^{(2)}|
	\geq
	\left(1-\frac p2\right)N,
\]
such that the first and third terms below are bounded by $\epsilon/2$
on $\cI_N^{(1)}$ and $\cI_N^{(2)}$, respectively. Define $\cI_N := \cI_N^{(1)}\cap\cI_N^{(2)}$.
Since
\[
	|\cI_N|
	\geq
	N-
	|\cI_N^{(1)c}|-
	|\cI_N^{(2)c}|
	\geq
	(1-p)N,
\]
the following estimates hold simultaneously for every
$i\in\cI_N$:
\begin{align*}
	\sup_{\widetilde\bpi\in\bfPi}&
	\cJ^{N,i}
	(\widetilde\bpi;\Lambda^N(\widehat\bupi)^{-i})
	-
	\cJ^{N,i}(\Lambda^N(\widehat\bupi))
\leq
	\sup_{\widetilde\bpi\in\bfPi}
	\left(
		\cJ^{N,i}
		(\widetilde\bpi;\Lambda^N(\widehat\bupi)^{-i})
		-
		\cJ^{\bunu,\ioN}(\widetilde\bpi)
	\right)
	\\
	&\quad+
	\sup_{\widetilde\bpi\in\bfPi}
	\left(
		\cJ^{\bunu,\ioN}(\widetilde\bpi)
		-
		\cJ^{\bunu,\ioN}(\widehat\bupi^{\ioN})
	\right)
+
	\left(
		\cJ^{\bunu,\ioN}(\widehat\bupi^{\ioN})
		-
		\cJ^{N,i}(\Lambda^N(\widehat\bupi))
	\right)
\leq
	\frac{\epsilon}{2}
	+
	0
	+
	\frac{\epsilon}{2}
	=
	\epsilon.
\end{align*}
The first inequality follows from the triangle inequality. The first
and third terms are bounded by
Theorem~\ref{thm_JioN_Jmu}; for the third term, we apply the theorem
with $\widetilde\bpi=\widehat\bupi^{\ioN}$. The middle term vanishes
because $R(\ioN)=0$.
\end{proof}

\section{Numerical Example}
\label{sec:num1}

\subsection{Model}

We extend the classical optimal execution problem~\cite{almgren2001optimal} to a GMFG. Similar models have been considered in the MFG literature, see e.g.~\cite[Vol. 1, Section 4.7.1]{carmonadelarue2018} and the references therein. In the present model, the traders are indexed by $u\in I$, with interactions governed by a graphon $G$.

\noindent{\textbf{Agent Dynamics and Markov Decision Process Approximation.}}
An infinitesimal trader with index $u \in [0,1]$ has an inventory $X_t^u$ evolving according to:
\begin{equation*}
	dX_t^u = \alpha_t^u dt + \sigma dW_t^u,
\end{equation*}
where $X_t^u$ is her inventory at time $t\in[0,T]$, $\alpha_t^u$ is her trading rate, and $(W_t^u)_{t\in[0,T]}$ is an idiosyncratic Brownian motion.

Motivated by this continuous-time model, the numerical experiments use a
finite-state, discrete-time controlled Markov chain on the grid
$\mathcal{X}=\{x_{\min},x_{\min}+h,\dots,x_{\max}\}$, with bounded action
set $\mathcal A$. For a fixed decision step $\Delta t=T/N_t$, the
implementation first sets, at interior grid points,
\[
	\widetilde p_+(\alpha)
	=
	\frac{\sigma^2\Delta t}{2h^2}
	+
	\frac{\alpha\Delta t}{2h},
	\qquad
	\widetilde p_-(\alpha)
	=
	\frac{\sigma^2\Delta t}{2h^2}
	-
	\frac{\alpha\Delta t}{2h}.
\]
We then set $p_+=\max\{\widetilde p_+,0\}$ and
$p_-=\max\{\widetilde p_-,0\}$. If $p_++p_->1$, both probabilities are
renormalized by their sum. The transition probabilities are
\[
	P(x+h\mid x,\alpha)=p_+,\qquad
	P(x-h\mid x,\alpha)=p_-,\qquad
	P(x\mid x,\alpha)=1-p_+-p_-.
\]
At the boundaries, transition mass that would leave the grid is added to
the probability of staying at the current boundary state. This Markov chain
is used for the numerical illustration of the discrete GMFG model.

\noindent{\textbf{Interactions and Reward Functions.}}
In this network-dependent model, the price impact felt by trader $u$ is a weighted average of the trading rates of the other traders $v$, with weights determined by the graphon $G(u,v)$. Here, $G(u,v)$ represents the strength of the price interaction between traders at positions $u$ and $v$. Intuitively, this network effect can model situations in which traders who are closer in the financial network, for example because they use the same prime broker or trade assets in the same sector, are more sensitive to one another's order flows.

The interactions therefore occur through the controls, and the interaction function is
\[
	\varphi_t(x,a)=a.
\]
Accordingly, for a policy profile $\boldsymbol{\alpha}$ and its induced state-distribution flow $\boldsymbol{\mu}$, the aggregate perceived by trader $u$ is $Z_t^u = \int_I G(u,v) \sum_{x\in\mathcal X} \alpha_t^v(x)\mu_t^v(x)\,dv$ for  $t\in\mathcal T$. 
Thus, $Z_t^u$ is the graphon-weighted expected trading rate of the traders interacting with trader $u$.

For each decision time $t\in\mathcal T$, the running reward of trader $u\in I$ is 
$$r_t^u(x,a,z) = -\left( \frac{c_\alpha}{2}a^2-\gamma xz \right)\Delta t,$$
where $\Delta t=T/N_t$ is the scaling applied to the running rewards over the $N_t$ decision stages. The first term represents the trading cost and captures increasing marginal costs of trading more rapidly, while the second term represents the price impact generated by the other traders. If nearby traders buy more, prices increase; otherwise, they decrease. These price changes affect the mark-to-market value of the trader's inventory.
The terminal reward is $r_T^u(x)=-\frac{c_T}{2}x^2$,
which penalizes inventory remaining at the terminal time.

\subsection{Computational Method and Parameters}

To compute the solution, we use a fixed-point method. Given an aggregate
profile, we solve the Bellman equation backward in time to find the optimal
control $\alpha_t^*$ and the value function $V$, and then solve the forward
equation to update the distribution $\mu_t$. The aggregate
\[
	Z_t^u
	=
	\int_0^1 G(u,v)
	\sum_{x\in\cX}
	\alpha_t^v(x)\mu_t^v(x)\,dv
\]
is then recomputed from the updated policy and distribution, and the
fixed-point update is damped with parameter $\tau$.

The time horizon $T=10$ is represented by $N_t=50$ decision stages. The
integral in the aggregate is approximated on a uniform grid of $n_u=50$
agent indices. The fixed-point iteration is run with damping factor
$\tau=0.2$, maximum iteration count $500$, and tolerance $10^{-5}$; it is
stopped when the mean Euclidean change in the aggregate profile $Z$ falls
below this tolerance. In the figures below, the horizontal axis labeled
``Time (Physical)'' denotes the linearly rescaled decision-stage grid.
The action space is the interval $[a_{\min},a_{\max}]=[-5,5]$ discretized
with $50$ points. The model parameters used across all simulations are
detailed in Table~\ref{tab:parameters}.

\begin{table}[htbp]
	\centering
	\begin{tabular}{lcl}
		\hline
		\textbf{Parameter}         & \textbf{Symbol} & \textbf{Value} \\ \hline
		Time Horizon               & $T$             & 10.0           \\
		Diffusion Coefficient      & $\sigma$        & 0.5            \\
		Aggregate Impact           & $\gamma$        & 8.0            \\
		Trading Cost Coefficient   & $c_\alpha$      & 1.0            \\
		Terminal Cost Coefficient  & $c_T$           & 20.0           \\
		Damping Factor (Iteration) & $\tau$          & 0.2             \\ \hline
	\end{tabular}
	\caption{Parameter configuration for the numerical solver.}
	\label{tab:parameters}
\end{table}

The numerical experiments are intended as illustrations of the model.
Some of the regularity conditions used in Section~\ref{sec_existence} are
not satisfied in these computations; for instance, the Star graphon is
discontinuous.

In the following, we consider three specific graphon structures that
represent different interaction topologies:

\begin{itemize}
	\item \textbf{Mean-Field Graphon}: Every agent interacts equally with all other agents, i.e., $G_{\mathrm{MF}}(u,v)=1$ for all $(u,v)\in[0,1]^2$.

	\item \textbf{Star Graphon}: This models a hub-and-spoke topology. A core group $[0,\alpha]$ acts as a hub, interacting with the periphery $[\alpha,1]$, while peripheral agents do not interact with each other. We use $\alpha=0.2$:
	      \[
		      G_{\mathrm{Star}}(u,v)
		      =
		      \mathbf{1}_{\{u\in[0,0.2],\,v\in(0.2,1]\}
		      \cup
		      \{v\in[0,0.2],\,u\in(0.2,1]\}}.
	      \]

		\item \textbf{Min-Max Graphon}: This continuous, piecewise smooth kernel represents an interaction gradient, with higher influence near the center of the agent space:
	      \[
		      G_{\mathrm{MM}}(u,v)
		      =
		      \min(u,v)(1-\max(u,v)).
	      \]
\end{itemize}

\subsection{Experiment 1: Common Skewed Initial Distribution}

In the first setting, we initialize all agents across the graphon with a
common \textit{center-right} distribution, uniform over the grid points
within distance $1$ of $x=6$. 
The state space is $[0,10]$, discretized with $n_x = 101$ points.
We test this across three graphon structures: Mean-Field, Star, and Min-Max.

\begin{figure}[htbp]
	\centering
	\begin{subfigure}{1.0\textwidth}
		\includegraphics[width=\textwidth]{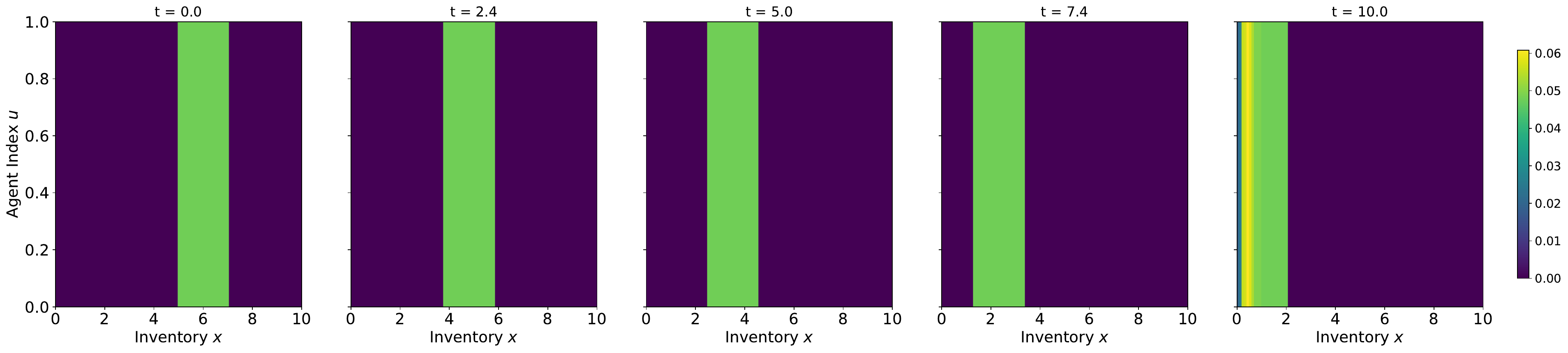}
	\end{subfigure}

	\begin{subfigure}{1.0\textwidth}
		\includegraphics[width=\textwidth]{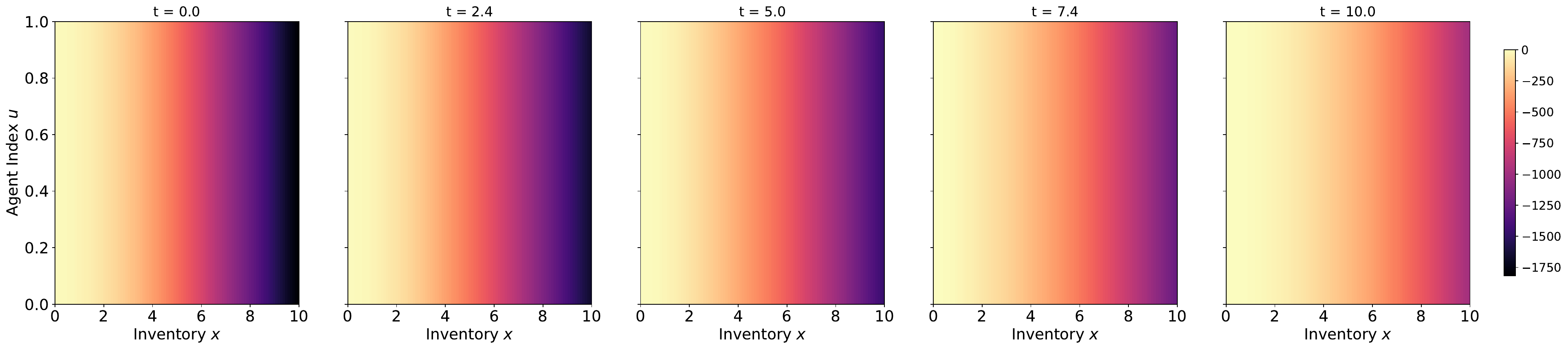}
	\end{subfigure}
	\caption{Evolution of the state distribution (top) and value function (bottom) for the Mean-Field graphon with a common initial distribution.}
	\label{fig:init_common_mf}
\end{figure}

With the Mean-Field graphon and a common initial distribution, all agents are identical. The results in Figure~\ref{fig:init_common_mf} show that agents actively trade to liquidate their portfolios, moving toward the state $x=0$ by terminal time $T$. The Mean-Field case serves as a baseline where every agent influences every other agent equally. For times close to $0$, the value function takes very negative values for large $x$, reflecting the high cost associated with starting with a large inventory.

The results for the Star and Min-Max graphons are respectively displayed in Figures~\ref{fig:init_common_star} and~\ref{fig:init_common_minmax}. Although the agents all start with the same initial distribution, we observe that, as time progresses, different indices have different distributions. With the Star graphon, at terminal time $T$, there is a clear split between high and low indices. With the Min-Max graphon, the terminal distribution exhibits a symmetric pattern about the axis $u=0.5$, which is consistent with the form of the graphon.

To better understand the pattern formation, we show the evolution of the aggregate, $Z_t^u$, (top row) and the control at time 0, $\alpha_0^u(x)$, (bottom row) for different graphons in Figure~\ref{fig:init_common_aggcont}. For the Mean-Field setting, the aggregate remains around $-2.5$ for most of the time interval, indicating a significant negative price impact. The policy takes negative values, consistent with the agents' objective to liquidate their portfolios. In the Star graphon case, we observe the formation of two sub-groups with distinct aggregates and policies. The hub group (low indices) experiences a more negative aggregate, while the effect on the peripheral group is less severe. Finally, with the Min-Max graphon, we observe that indices symmetric about $u=0.5$ (e.g., $u=0$ and $u=1$, or $u=0.24$ and $u=0.76$) share the same aggregate and policy, consistent with the graphon's symmetry. Agents at the extreme boundaries $(u=0, 1)$ have an aggregate impact near zero, as they are least connected.

\begin{figure}[htbp]
	\centering
	\begin{subfigure}{1.0\textwidth}
		\includegraphics[width=\textwidth]{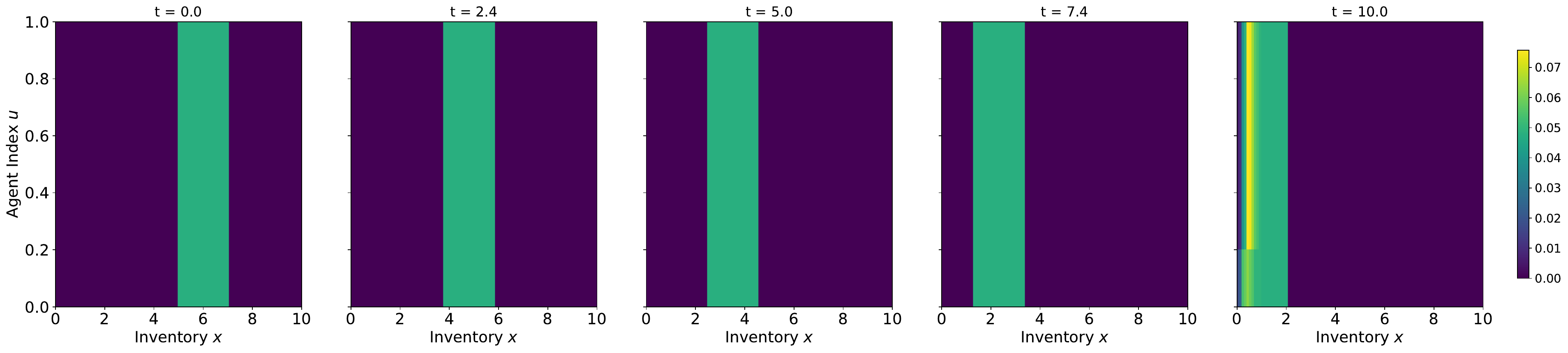}
	\end{subfigure}

	\begin{subfigure}{1.0\textwidth}
		\includegraphics[width=\textwidth]{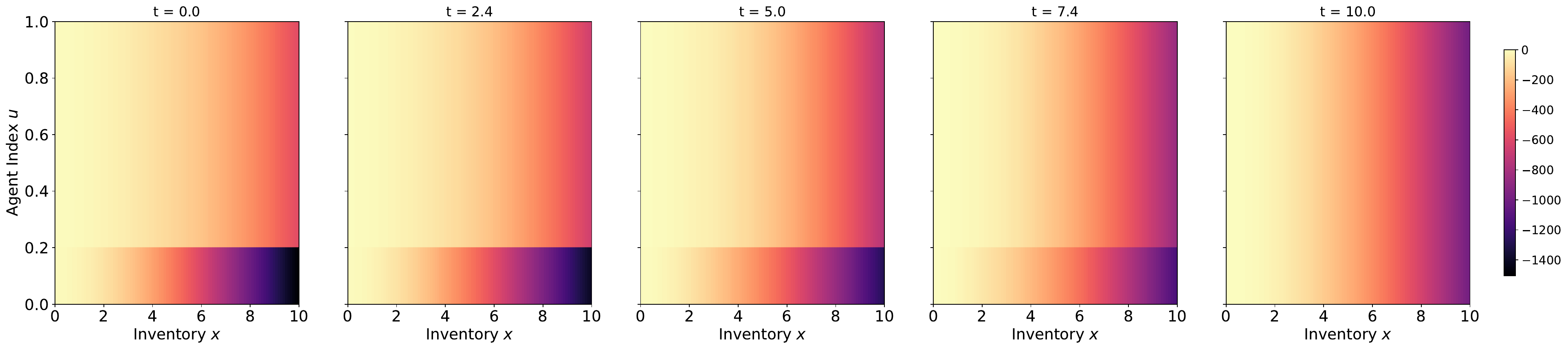}
	\end{subfigure}
	\caption{Evolution of the state distribution (top) and value function (bottom) for the Star graphon with a common initial distribution.}
	\label{fig:init_common_star}
\end{figure}

\begin{figure}[htbp]
	\centering
	\begin{subfigure}{1.0\textwidth}
		\includegraphics[width=\textwidth]{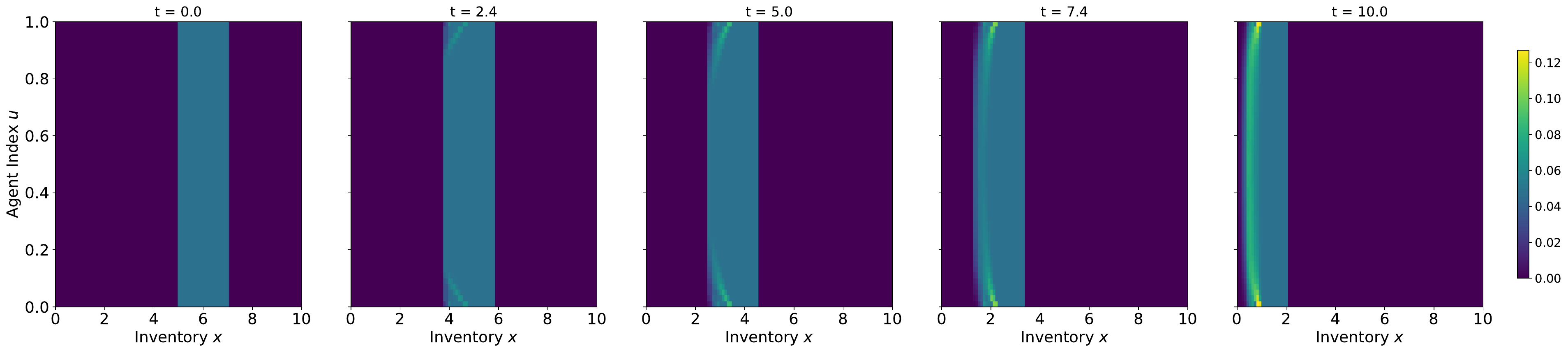}
	\end{subfigure}

	\begin{subfigure}{1.0\textwidth}
		\includegraphics[width=\textwidth]{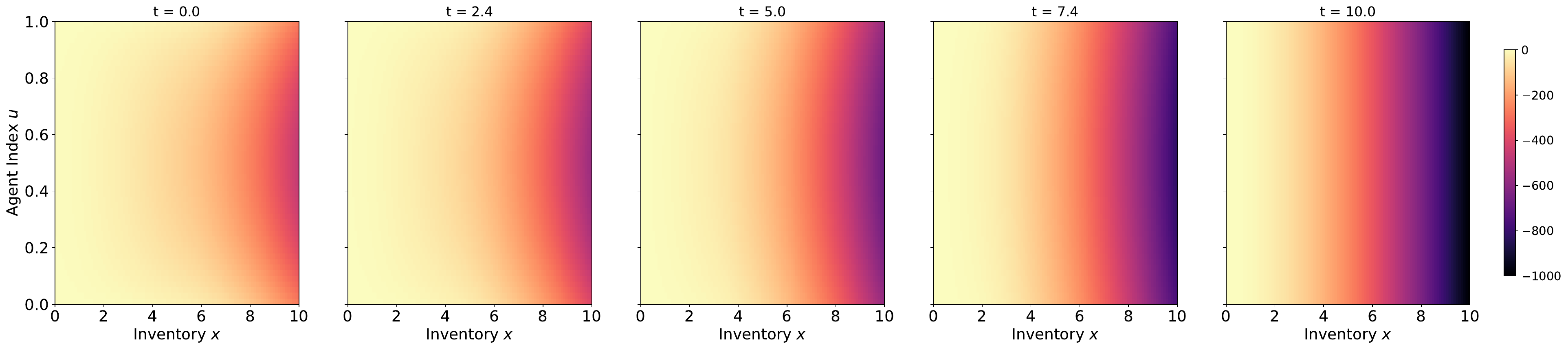}
	\end{subfigure}
	\caption{Evolution of the state distribution (top) and value function (bottom) for the Min-Max graphon with a common initial distribution.}
	\label{fig:init_common_minmax}
\end{figure}

\begin{figure}[htbp]
	\centering
	\begin{subfigure}{0.31\textwidth}
		\includegraphics[width=\textwidth]{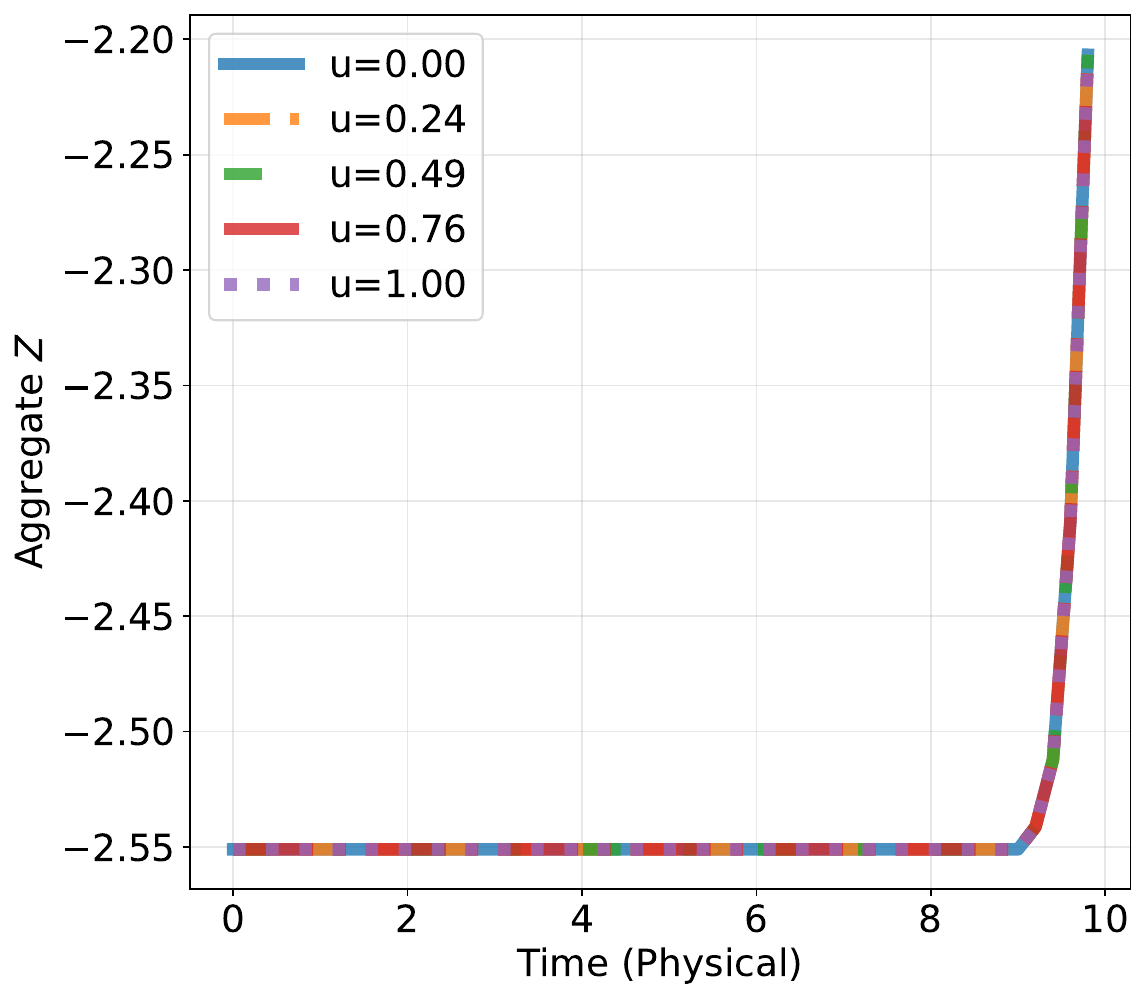}
	\end{subfigure}
	\hfill
	\begin{subfigure}{0.31\textwidth}
		\includegraphics[width=\textwidth]{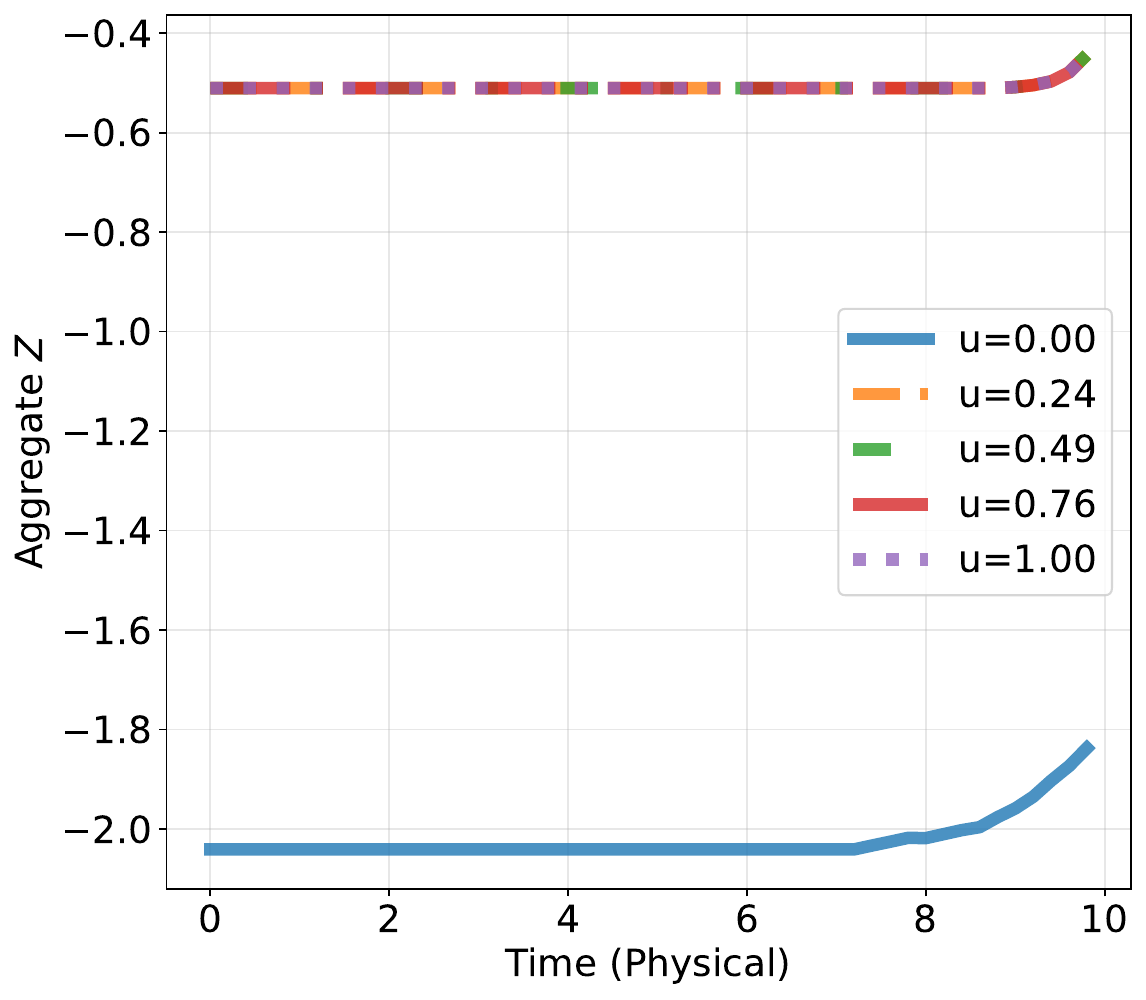}
	\end{subfigure}
	\hfill
	\begin{subfigure}{0.31\textwidth}
		\includegraphics[width=\textwidth]{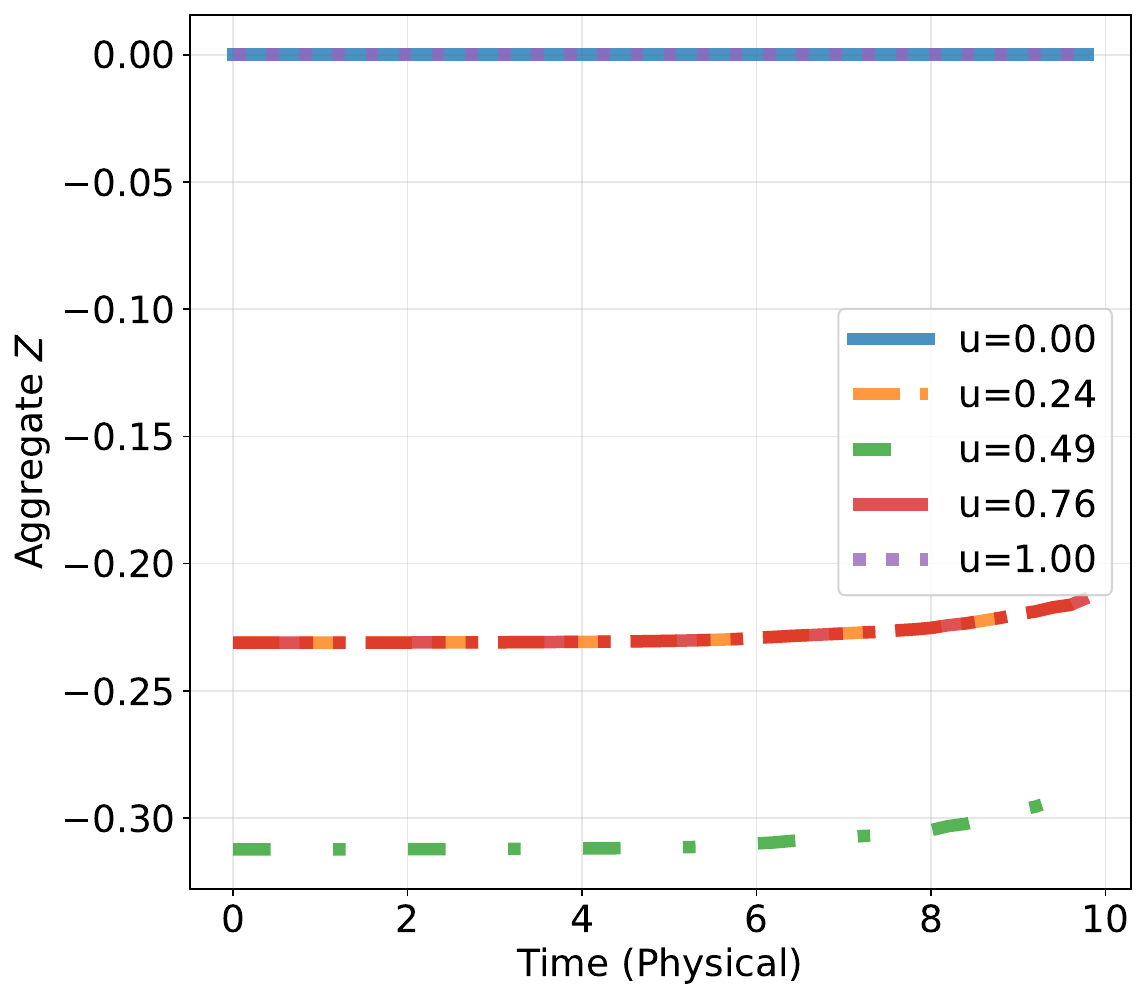}
	\end{subfigure}

	\begin{subfigure}{0.31\textwidth}
		\includegraphics[width=\textwidth]{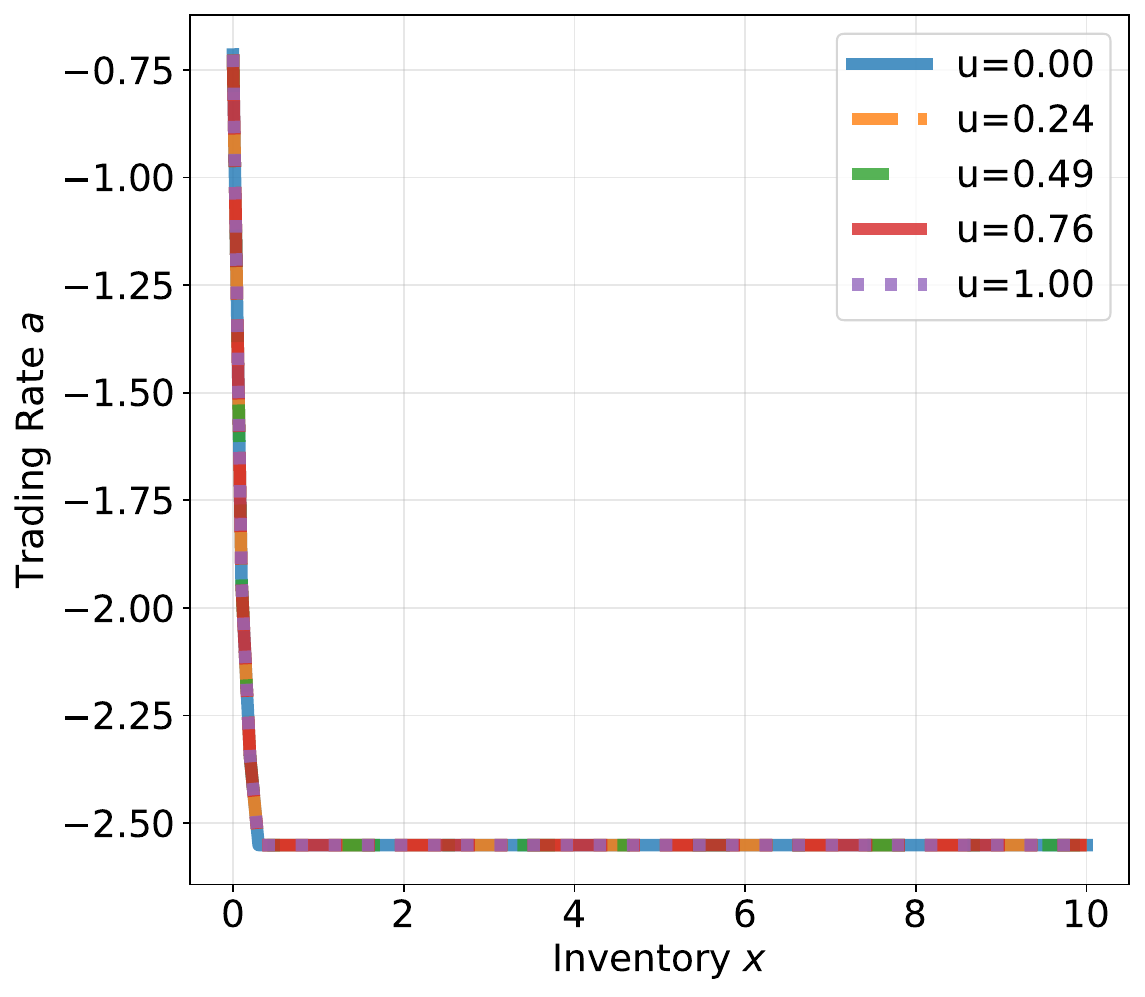}
	\end{subfigure}
	\hfill
	\begin{subfigure}{0.31\textwidth}
		\includegraphics[width=\textwidth]{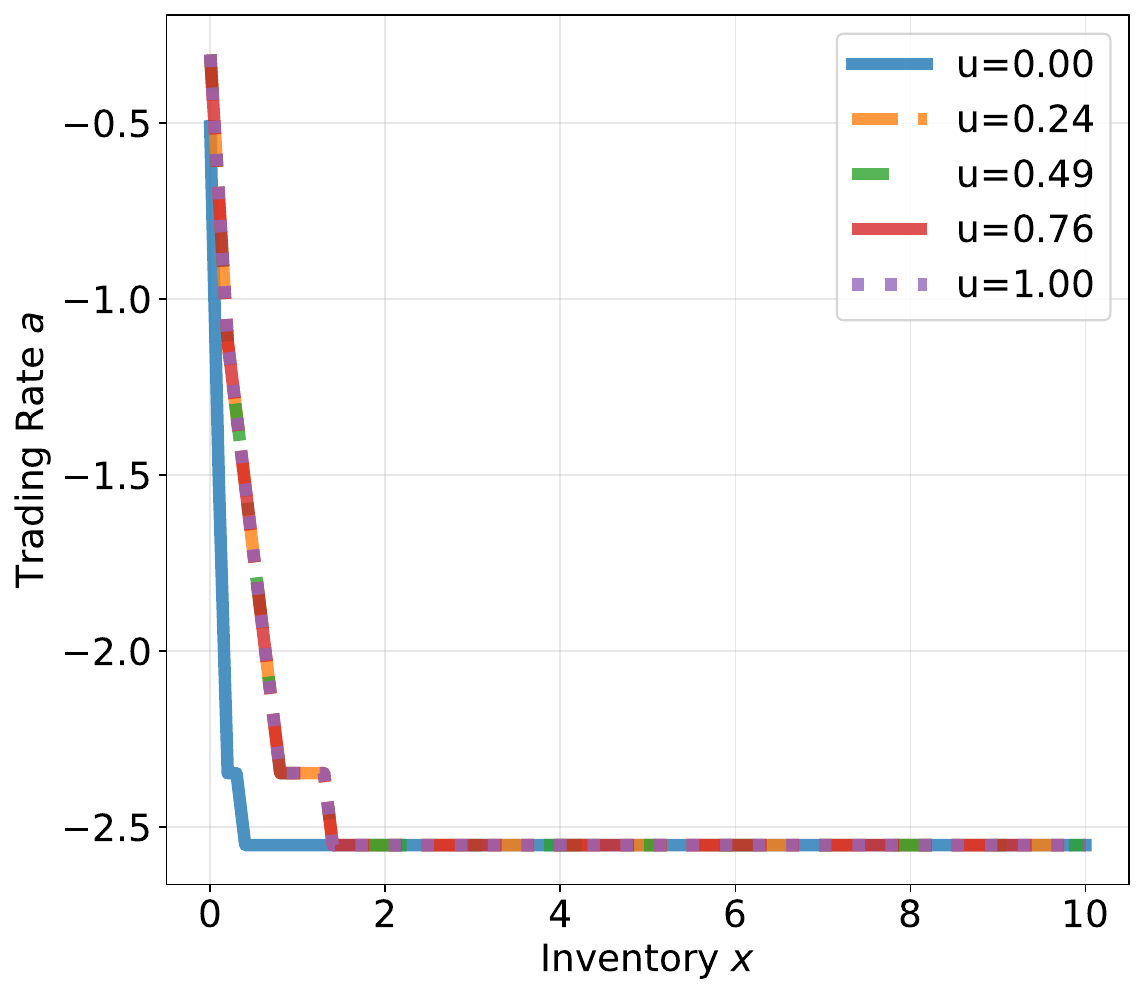}
	\end{subfigure}
	\hfill
	\begin{subfigure}{0.31\textwidth}
		\includegraphics[width=\textwidth]{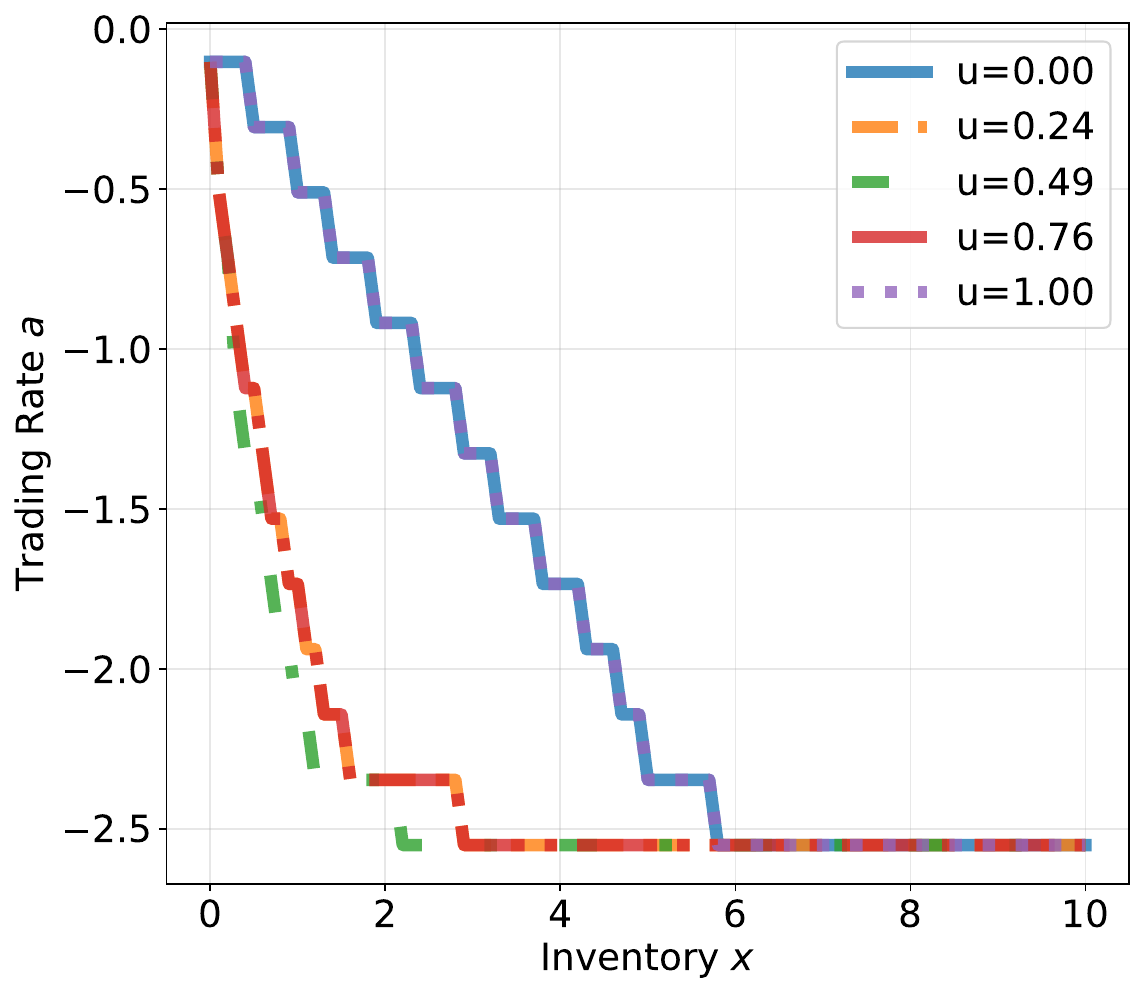}
	\end{subfigure}

	\begin{subfigure}{0.31\textwidth}
		\caption*{Mean-Field}
	\end{subfigure}
	\hfill
	\begin{subfigure}{0.31\textwidth}
		\caption*{Star}
	\end{subfigure}
	\hfill
	\begin{subfigure}{0.31\textwidth}
		\caption*{Min-Max}
	\end{subfigure}
	\caption{Aggregate's evolution (top row) and control at time $0$ (bottom row) for different graphons with a common initial distribution.}
	\label{fig:init_common_aggcont}
\end{figure}

\subsection{Experiment 2: Spatially Varying Initial Distributions}

In this setting, the initial distribution $\mu_0^u$ varies with the agent's
position $u \in [0,1]$. Specifically, for each $u$, we use a uniform
distribution over grid points within distance $1$ of a center that varies
linearly from $-9$ to $9$ as $u$ ranges from $0$ to $1$. This creates a
diagonal concentration of mass across the agent space while keeping the
initial support inside the state domain. We expand the state space to
$[-10,10]$ with $n_x = 201$ points to accommodate these initial conditions. 
We focus on the Star and Min-Max graphons, which exhibit the most significant structural heterogeneity.

The evolution of the state distribution and value function is displayed in Figures~\ref{fig:init_diago_star} and~\ref{fig:init_diago_minmax}. We again observe that the Star graphon leads to the formation of two distinct groups, while the Min-Max graphon yields a symmetric shape about $u=0.5$. Figure~\ref{fig:init_diago_aggcont} shows the corresponding evolution of the aggregate and the control at time $0$. We observe that the policies are positive for $x<0$ and negative for $x>0$, as agents trade to return their inventory toward zero.

These experiments illustrate that the graphon structure
can strongly affect the observed population dynamics and feedback policies. In
particular, in this example, both the network topology and the initial state heterogeneity
remain visible in the computed extended GMFG system.

\begin{figure}[htbp]
	\centering
	\begin{subfigure}{1.0\textwidth}
		\includegraphics[width=\textwidth]{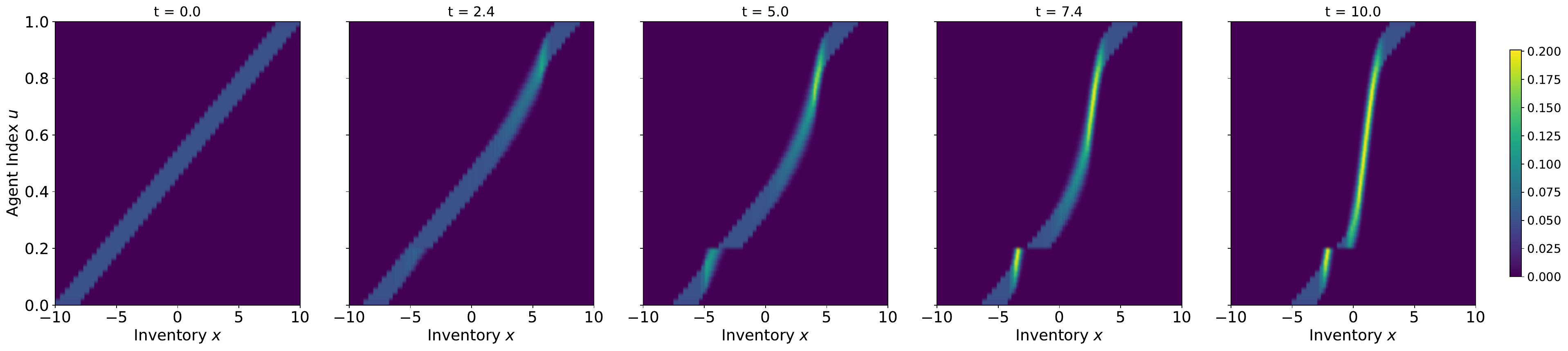}
	\end{subfigure}

	\begin{subfigure}{1.0\textwidth}
		\includegraphics[width=\textwidth]{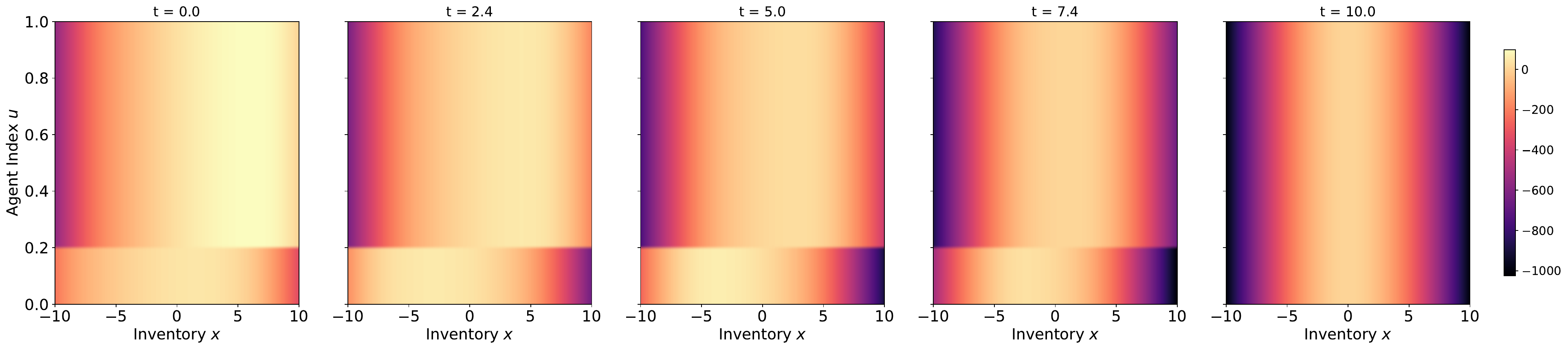}
	\end{subfigure}
	\caption{Evolution of the state distribution (top) and value function (bottom) for the Star graphon with a spatially varying initial distribution.}
	\label{fig:init_diago_star}
\end{figure}

\begin{figure}[htbp]
	\centering
	\begin{subfigure}{1.0\textwidth}
		\includegraphics[width=\textwidth]{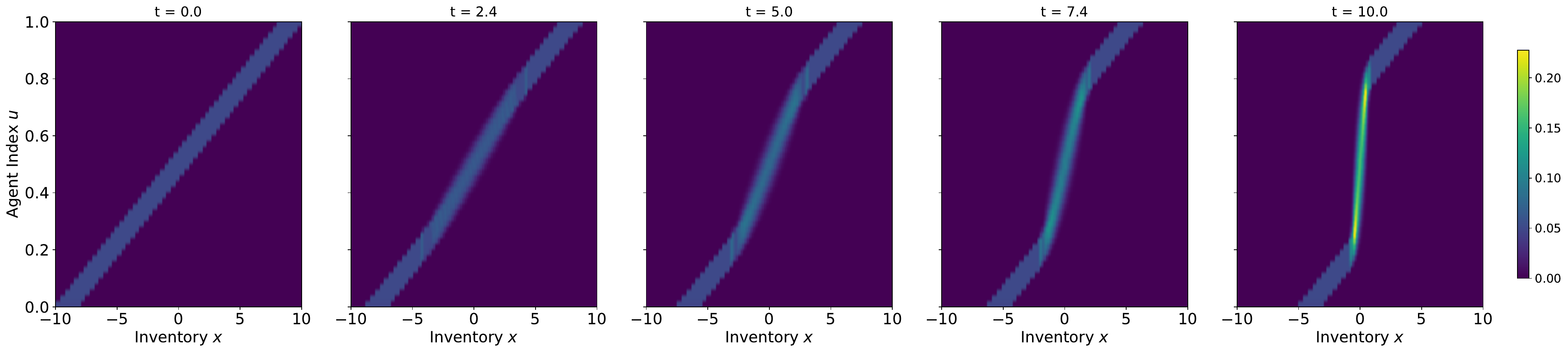}
	\end{subfigure}

	\begin{subfigure}{1.0\textwidth}
		\includegraphics[width=\textwidth]{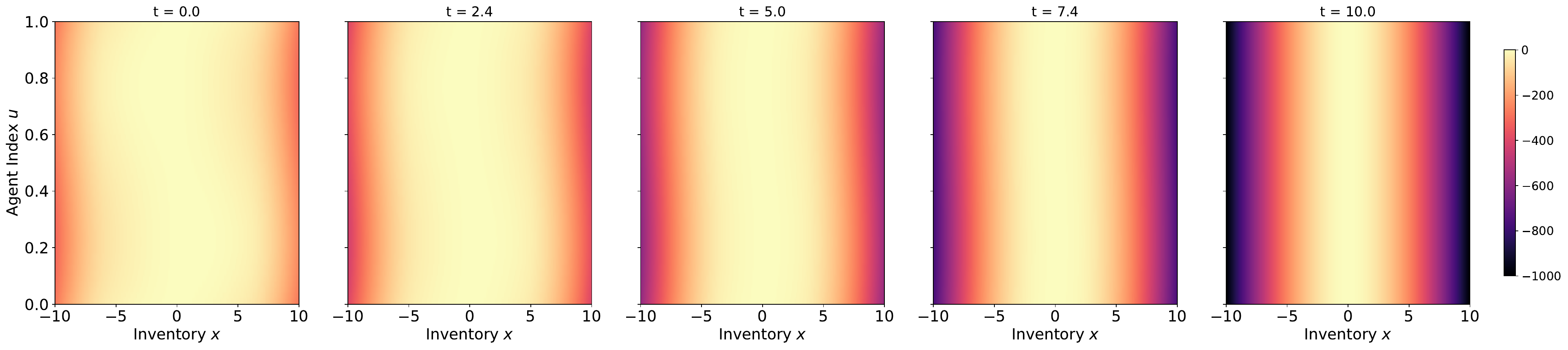}
	\end{subfigure}
	\caption{Evolution of the state distribution (top) and value function (bottom) for the Min-Max graphon with a spatially varying initial distribution.}
	\label{fig:init_diago_minmax}
\end{figure}

\begin{figure}[htbp]
	\centering
	\begin{subfigure}{0.48\textwidth}
		\includegraphics[width=0.5\textwidth]{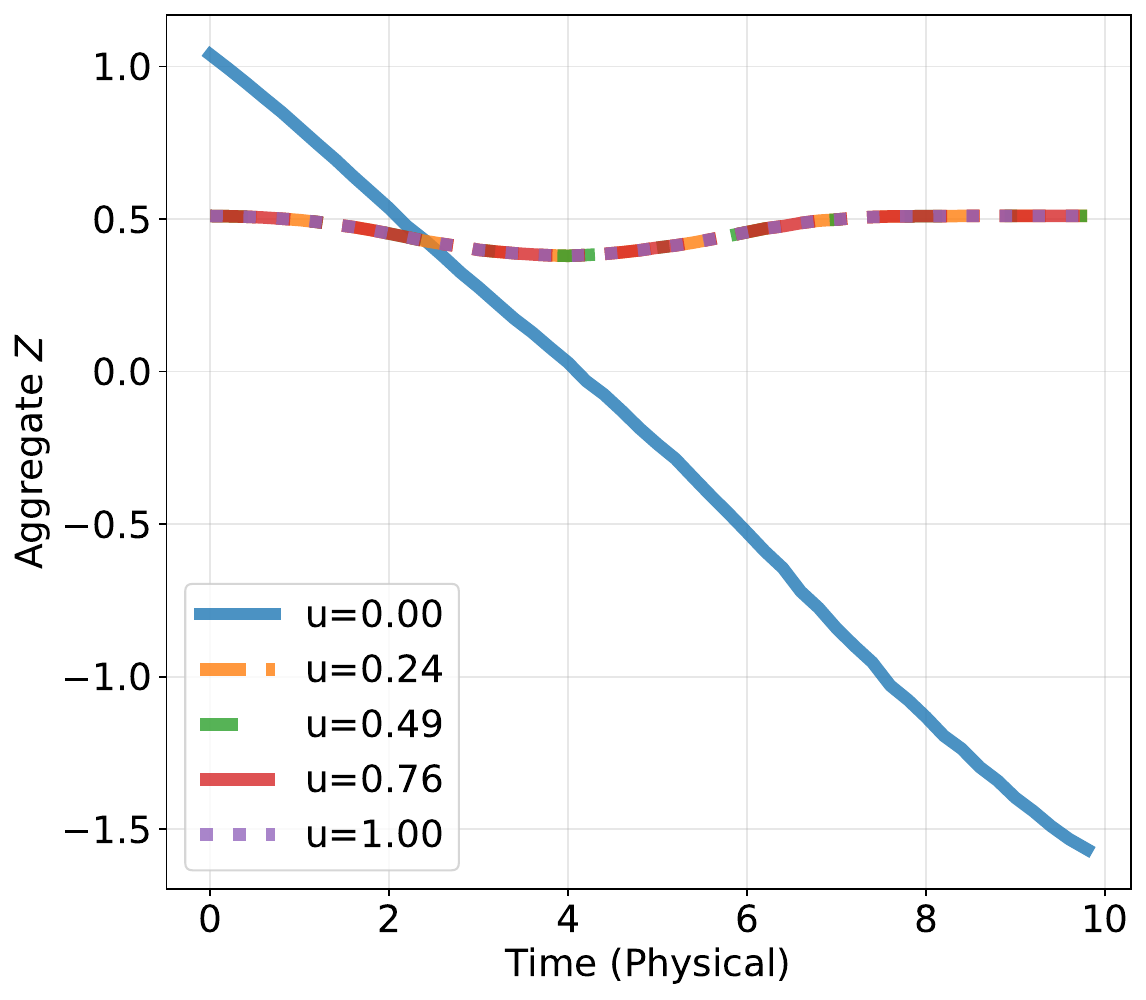}%
		\includegraphics[width=0.5\textwidth]{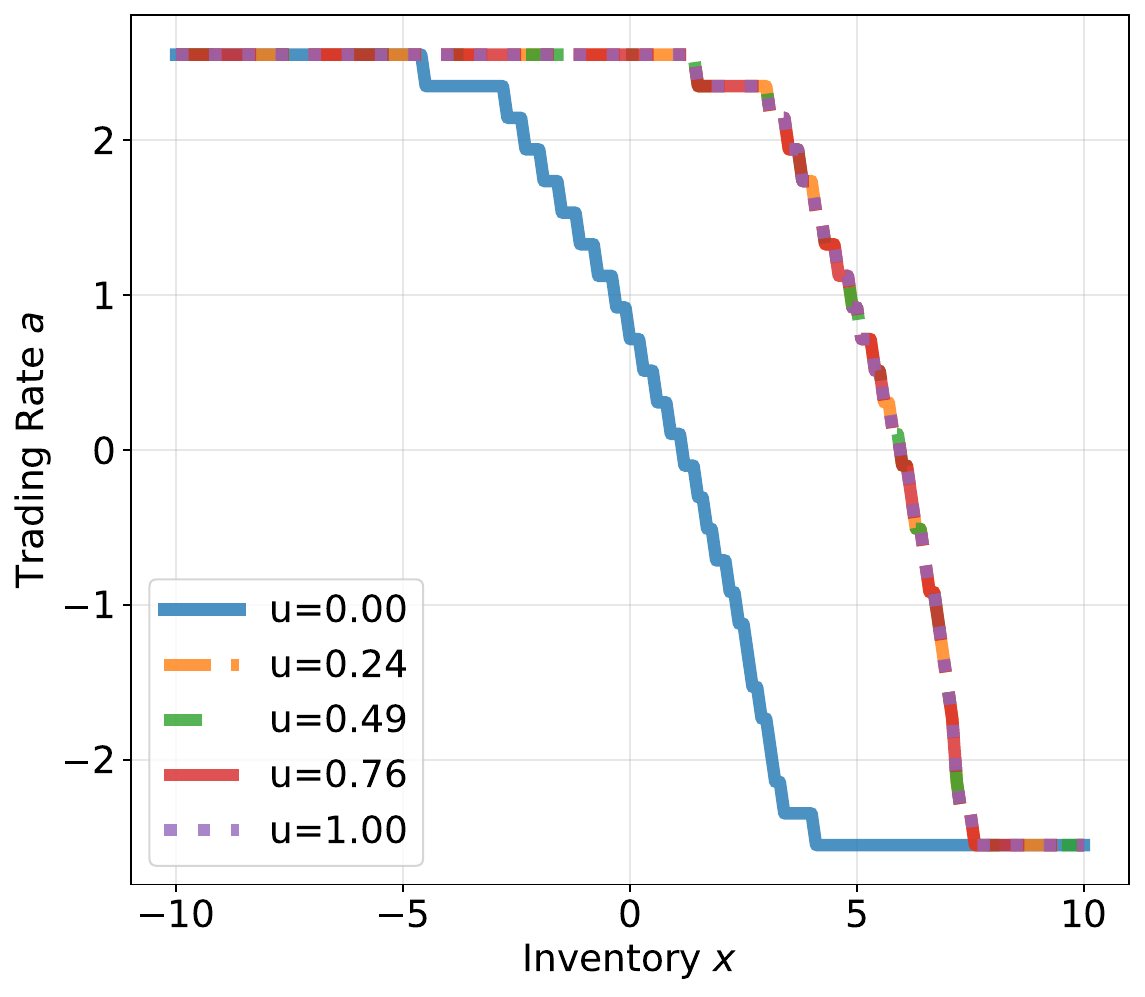}
		\caption*{Star}
	\end{subfigure}
	\hfill
	\begin{subfigure}{0.48\textwidth}
		\includegraphics[width=0.5\textwidth]{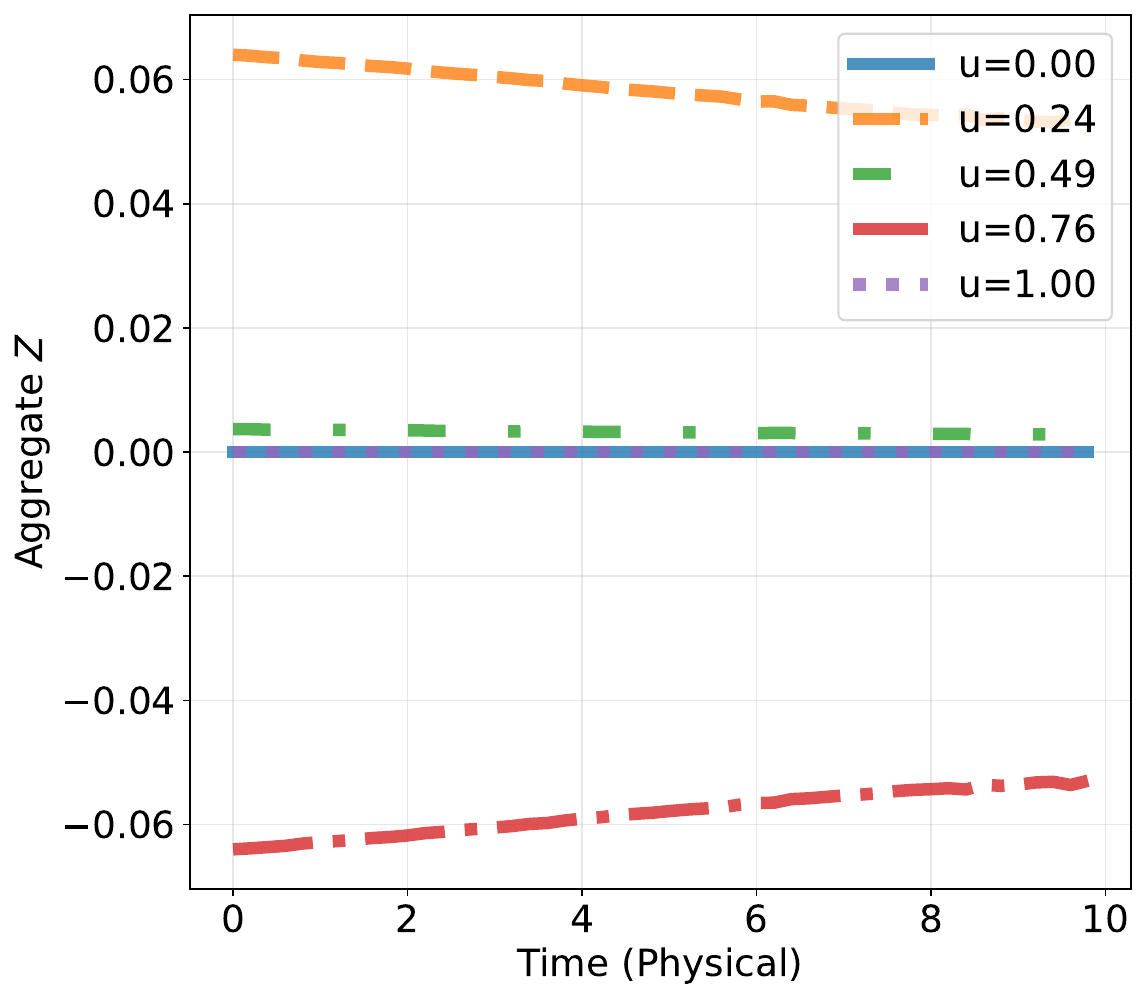}%
		\includegraphics[width=0.5\textwidth]{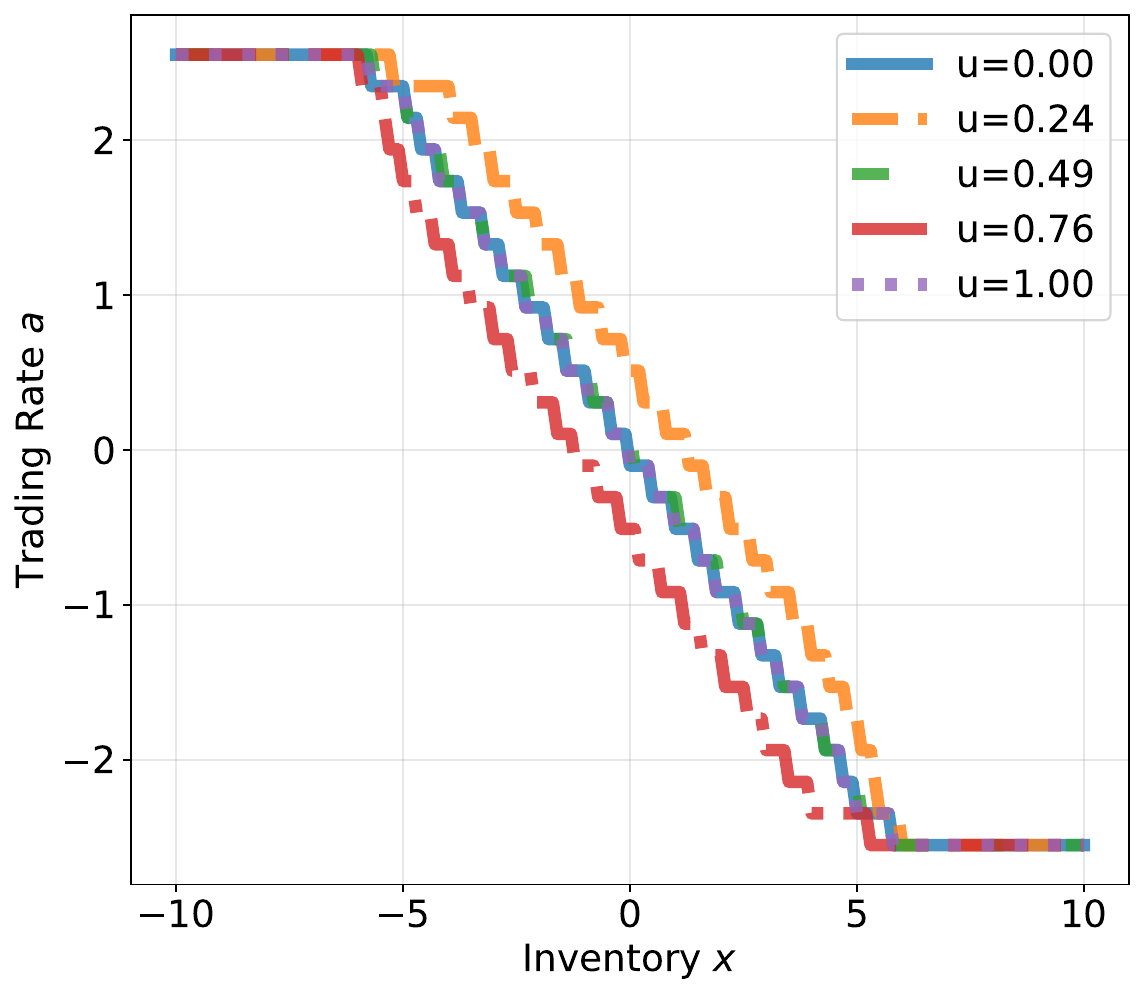}
		\caption*{Min-Max}
	\end{subfigure}
	\caption{Aggregate's evolution and control at time $0$ for the Star graphon (left) and the Min-Max graphon (right), with a spatially varying initial distribution.}
	\label{fig:init_diago_aggcont}
\end{figure}

\bibliographystyle{plain}
\bibliography{BiblioGMFG}

\appendix

\section{Proofs of Propositions and Lemmas}
\subsection{Proof of Lemma~\ref{lemma_f_stronglyConcave}}
\label{proof:lemma_f_stronglyConcave}

We show that $\partial_{aa}^2f(\theta,a) \leq -\lambda<0$
for some $\lambda>0$, uniformly in
$\theta=(t,u,Z,V,x)$.

By Assumption~\ref{assp_diff_concave_faz} and the definition of $f$,
\begin{align*}
	\partial_{aa}^2f(\theta,a)
	&=
	\partial_{aa}^2r_t^u(x,a,Z)
	+
	\partial_{aa}^2
	\left[
		\sum_{x'\in\cX}
		P_t^u(x'\mid x,a,Z)V(x')
	\right]
	\\
	&\leq
	-TC_rC_P|\cX|-\epsilon_r
	+
	\sum_{x'\in\cX}
	\left|
		\partial_{aa}^2
		P_t^u(x'\mid x,a,Z)
	\right|
	|V(x')|
	\\
	&\leq
	-TC_rC_P|\cX|-\epsilon_r
	+
	TC_rC_P|\cX|
	=
	-\epsilon_r.
\end{align*}
Therefore, the statement holds with $\lambda=\epsilon_r$.
\qed

\subsection{Proof of Lemma~\ref{lemma_ZtoA_lip}}
\label{proof:lemma_ZtoA_lip}

By Lemma~\ref{lemma_f_stronglyConcave}, the map $a\longmapsto f_t^u(Z,V,x,a)$
is uniformly strongly concave and hence admits a unique maximizer
$\widehat a_t^u(Z,V,x)$. To show Lipschitz continuity, let
$\theta:=(t,u,Z,V,x), \widetilde\theta := (t,u,\widetilde Z,\widetilde V,x)$,
and set $a=\widehat a(\theta), \widetilde a=\widehat a(\widetilde\theta)$.
The maximizer $\widehat a$ satisfies the variational inequality
\[
	\forall\,b\in\cA,
	\qquad
	(b-a)\partial_af_t^u(Z,V,x,a)
	\leq0.
\]
Choosing $b=\widetilde a$ and $b=a$ in the two corresponding
variational inequalities gives
\[
	(\widetilde a-a)\partial_af(\theta,a)\leq0,
	\qquad
	(a-\widetilde a)
	\partial_af(\widetilde\theta,\widetilde a)
	\leq0.
\]
Consequently,
\[
	(\widetilde a-a)
	\left[
		\partial_af(\theta,a)
		-
		\partial_af(\widetilde\theta,\widetilde a)
	\right]
	\leq0.
\]
Hence,
\begin{align*}
	(\widetilde a-a)
	\left[
		\partial_af(\theta,a)
		-
		\partial_af(\theta,\widetilde a)
	\right]
	\leq
	(\widetilde a-a)
	\left[
		\partial_af(\widetilde\theta,\widetilde a)
		-
		\partial_af(\theta,\widetilde a)
	\right].
\end{align*}

The strong concavity of $f$ implies, see, e.g.,
\cite[Lemma A.1]{montrucchio1987lipschitz}, that
\[
	\lambda|\widetilde a-a|^2
	\leq
	(\widetilde a-a)
	\left[
		\partial_af(\theta,a)
		-
		\partial_af(\theta,\widetilde a)
	\right].
\]
Combining the two inequalities, we deduce
\begin{align*}
	\lambda|\widetilde a-a|^2
	\leq
	(\widetilde a-a)
	\left[
		\partial_af(\theta,a)
		-
		\partial_af(\theta,\widetilde a)
	\right]
	\leq
	(\widetilde a-a)
	\left[
		\partial_af(\widetilde\theta,\widetilde a)
		-
		\partial_af(\theta,\widetilde a)
	\right].
\end{align*}
By
Assumption~\ref{assp_diff_concave_faz}%
\ref{assp_ZVtodf_lip},
\[
	|\widehat a_t^u(Z,V,x)
	-
	\widehat a_t^u(\widetilde Z,\widetilde V,x)|
	\leq
	\frac{L_{f,Z}}{\lambda}
	|Z-\widetilde Z|
	+
	\frac{L_{f,V}}{\lambda}
	\max_{y\in\cX}
	|V(y)-\widetilde V(y)|.
\]
Taking the maximum over $x\in\cX$ concludes the proof with
$L_{a,Z}=\frac{L_{f,Z}}{\lambda}$ and $L_{a,V}=\frac{L_{f,V}}{\lambda}$.
\qed

\subsection{Proof of Lemma~\ref{lemma_ztoVz_lip}}
\label{proof:lemma_ztoVz_lip}

The function $\bV^{\bZ,u}$ defined in~\eqref{eq_VZ} is well defined.
Indeed, the terminal value is prescribed by $r_T^u$, and at every
decision time the value function is the maximum of a continuous
function on the compact action space $\cA$. Joint measurability follows
by backward induction and the measurable maximum theorem.

We first note that
\begin{equation}\label{eq:preliminary-value-bound}
	|V_t^{\bZ,u}(x)|
	\leq
	(T-t+1)C_r,
	\qquad
	t\in\ocT,\quad x\in\cX.
\end{equation}
Indeed, at time $T$, $|V_T^{\bZ,u}(x)| = |r_T^u(x)| \leq C_r$.
If the bound holds at time $t+1$, then
\begin{align*}
	|V_t^{\bZ,u}(x)|
	&=
	\left|
		\max_{a\in\cA}
		\left\{
			r_t^u(x,a,Z_t)
			+
			\sum_{x'\in\cX}
			P_t^u(x'\mid x,a,Z_t)
			V_{t+1}^{\bZ,u}(x')
		\right\}
	\right|
	\\
	&\leq
	\max_{a\in\cA}
	\left\{
		|r_t^u(x,a,Z_t)|
		+
		\sum_{x'\in\cX}
		P_t^u(x'\mid x,a,Z_t)
		|V_{t+1}^{\bZ,u}(x')|
		\right\}
	\\
	&\leq
	C_r+(T-t)C_r
	=
	(T-t+1)C_r.
\end{align*}
This proves item~\ref{lemma_sub_V_lessTCr}.

We now prove item~\ref{lemma_sub_ZtoV_lip_infinf} by backward induction.
At time $T$, the value function is $r_T^u$ and does not depend on
$\bZ$, so we set $L_{V,Z,T}=0$.
Let $t\in\cT$ and assume that
$\bZ\mapsto V_{t+1}^{\bZ,u}$ is Lipschitz continuous with constant
$L_{V,Z,t+1}$. Let
$\bZ,\widetilde\bZ\in\RR^\cT$. Using the inequality between two
maxima, we have
\begin{align*}
	&
	\max_{x\in\cX}
	\left|
		V_t^{\bZ,u}(x)
		-
		V_t^{\widetilde\bZ,u}(x)
	\right|
	\\
	&\leq
	\max_{\substack{x\in\cX\\a\in\cA}}
	\Bigg|
		r_t^u(x,a,Z_t)
		-
		r_t^u(x,a,\widetilde Z_t)
	+
	\sum_{x'\in\cX}
	P_t^u(x'\mid x,a,Z_t)
	V_{t+1}^{\bZ,u}(x')
	-
	\sum_{x'\in\cX}
	P_t^u(x'\mid x,a,\widetilde Z_t)
	V_{t+1}^{\widetilde\bZ,u}(x')
	\Bigg|
	\\
	&\leq
	\max_{\substack{x\in\cX\\a\in\cA}}
	\Bigg(
		|r_t^u(x,a,Z_t)
		-
		r_t^u(x,a,\widetilde Z_t)|
	+
	\left|
		\sum_{x'\in\cX}
		\left[
			P_t^u(x'\mid x,a,Z_t)
			-
			P_t^u(x'\mid x,a,\widetilde Z_t)
		\right]
		V_{t+1}^{\bZ,u}(x')
	\right|
	\\
	&\qquad
	+
	\left|
		\sum_{x'\in\cX}
		P_t^u(x'\mid x,a,\widetilde Z_t)
		\left[
			V_{t+1}^{\bZ,u}(x')
			-
			V_{t+1}^{\widetilde\bZ,u}(x')
		\right]
	\right|
	\Bigg).
\end{align*}

For the first term, by
Assumption~\ref{assp_P_lip}%
\ref{assp_z_to_r_lip}, $|r_t^u(x,a,Z_t) - r_t^u(x,a,\widetilde Z_t)| \leq L_{r,Z} \|\bZ-\widetilde\bZ\|_\infty$.
For the second term, by
Assumption~\ref{assp_P_lip}%
\ref{assp_z_to_P_lip}
and~\eqref{eq:preliminary-value-bound},
\begin{align*}
	&
	\left|
		\sum_{x'\in\cX}
		\left[
			P_t^u(x'\mid x,a,Z_t)
			-
			P_t^u(x'\mid x,a,\widetilde Z_t)
		\right]
		V_{t+1}^{\bZ,u}(x')
	\right|
	\\
	& \qquad \leq
	L_{P,Z}|\cX|
	\max_{x'\in\cX}
	|V_{t+1}^{\bZ,u}(x')|
	\|\bZ-\widetilde\bZ\|_\infty
	\leq
	L_{P,Z}|\cX|(T-t)C_r
	\|\bZ-\widetilde\bZ\|_\infty.
\end{align*}
For the third term, since $P_t^u(\cdot\mid x,a,\widetilde Z_t)$
is a probability distribution,
\begin{align*}
	&
	\left|
		\sum_{x'\in\cX}
		P_t^u(x'\mid x,a,\widetilde Z_t)
		\left[
			V_{t+1}^{\bZ,u}(x')
			-
			V_{t+1}^{\widetilde\bZ,u}(x')
		\right]
	\right|
	\\
	&\qquad \leq
	\max_{x'\in\cX}
	\left|
		V_{t+1}^{\bZ,u}(x')
		-
		V_{t+1}^{\widetilde\bZ,u}(x')
	\right|
	\leq
	L_{V,Z,t+1}
\|\bZ-\widetilde\bZ\|_\infty.
\end{align*}
Collecting the terms and setting $L_{V,Z,t} := L_{r,Z} + L_{P,Z}|\cX|(T-t)C_r + L_{V,Z,t+1}$,
we obtain
\[
	\max_{x\in\cX}
	|V_t^{\bZ,u}(x)
	-
	V_t^{\widetilde\bZ,u}(x)|
	\leq
	L_{V,Z,t}
	\|\bZ-\widetilde\bZ\|_\infty.
\]
Thus, with $\overline K_{V,Z} := \max_{t\in\ocT}L_{V,Z,t}$,
item~\ref{lemma_sub_ZtoV_lip_infinf} holds with
$\overline K_{V,Z}$.

We next prove item~\ref{lemma_sub_ZtoV_lip_infinf2}. From the preceding
pointwise estimate,
\begin{align*}
	&
	\|\buV^{\buZ}
	-
	\buV^{\widetilde\buZ}\|_{\infty,\infty,2}
	\leq
	\overline K_{V,Z}
	\left(
		\int_I
		\max_{s\in\cT}
		|Z_s^u-\widetilde Z_s^u|^2\,du
	\right)^{1/2}
	\\
	&\qquad \leq
	\overline K_{V,Z}
	\left(
		\sum_{s\in\cT}
		\int_I
		|Z_s^u-\widetilde Z_s^u|^2\,du
	\right)^{1/2}
	\leq
	\sqrt{T}\,
	\overline K_{V,Z}
	\|\buZ-\widetilde\buZ\|_{\infty,2}.
\end{align*}
After enlarging the constant and setting $K_{V,Z} := \max\{1,\sqrt T\}\, \overline K_{V,Z}$,
both items~\ref{lemma_sub_ZtoV_lip_infinf}
and~\ref{lemma_sub_ZtoV_lip_infinf2} hold with the same constant.
\qed

\subsection{Proof of Lemma~\ref{lemma_ZtomuZ_lip}}
\label{proof:lemma_ZtomuZ_lip}

The argmax in~\eqref{eq_pi*ZVu} is uniquely defined by
Lemma~\ref{lemma_ZtoA_lip}, since the objective is uniformly strongly
concave. Moreover, Assumption~\ref{assp_P_lip} and the measurable
maximum theorem imply that the unique optimizer is measurable. For $(\bpi^u,\bZ^u) \in \bfPi\times\bcZ$,
let us introduce $\bmu^{u,\bpi^u,\bZ^u}$ defined by
\begin{equation*}
	\begin{dcases}
		\mu_{t+1}^{u,\bpi^u,\bZ^u}(x)
		=
		\displaystyle
		\sum_{x'\in\cX}
		\mu_t^{u,\bpi^u,\bZ^u}(x')
		P_t^u\left(
			x\mid x',
			\pi_t^u(x'),
			Z_t^u
		\right), \quad
		x\in\cX,\quad t\in\cT,
		\\[0.4em]
		\mu_0^{u,\bpi^u,\bZ^u}(x)
		=
		\mu_{\mathrm{init}}^u(x),
		\qquad x\in\cX.
	\end{dcases}
\end{equation*}

We first prove by induction that
$\bmu^{u,\bpi^u,\bZ^u}$ is Lipschitz continuous with respect to
$(\bpi^u,\bZ^u)$ for the norms
\begin{equation*}
	\|\bmu\|_{\infty,\infty}
	:=
	\max_{\substack{t\in\ocT\\x\in\cX}}
	|\mu_t(x)|,
	\qquad
	\|(\bpi,\bZ)\|_\infty
	:=
	\max\left\{
		\|\bpi\|_{\infty,\infty},
		\|\bZ\|_\infty
	\right\}.
\end{equation*}

Since $\mu_{\mathrm{init}}^u$ does not depend on
$(\bpi^u,\bZ^u)$, the induction starts with
$L_{m,0}=0$. Assume that $(\bpi^u,\bZ^u) \longmapsto \mu_t^{u,\bpi^u,\bZ^u}$
is Lipschitz continuous with constant $L_{m,t}$. For any
$(\bpi^u,\bZ^u)$ and
$(\widetilde\bpi^u,\widetilde\bZ^u)$, and any $x\in\cX$,
\begin{align}
	&
	\left|
		\mu_{t+1}^{u,\bpi^u,\bZ^u}(x)
		-
		\mu_{t+1}^{u,\widetilde\bpi^u,\widetilde\bZ^u}(x)
	\right|
	\notag
	\\
	&\quad \leq
	\sum_{x'\in\cX}
	\mu_t^{u,\bpi^u,\bZ^u}(x')
	\left|
		P_t^u\left(
			x\mid x',
			\pi_t^u(x'),
			Z_t^u
		\right)
		-
		P_t^u\left(
			x\mid x',
			\pi_t^u(x'),
			\widetilde Z_t^u
		\right)
	\right|
	\notag
	\\
	&\qquad
	+
	\sum_{x'\in\cX}
	\mu_t^{u,\bpi^u,\bZ^u}(x')
	\left|
		P_t^u\left(
			x\mid x',
			\pi_t^u(x'),
			\widetilde Z_t^u
		\right)
		-
		P_t^u\left(
			x\mid x',
			\widetilde\pi_t^u(x'),
			\widetilde Z_t^u
		\right)
	\right|
	\notag
	\\
	&\qquad
	+
	\left|
		\sum_{x'\in\cX}
		\left[
			\mu_t^{u,\bpi^u,\bZ^u}(x')
			-
			\mu_t^{u,\widetilde\bpi^u,\widetilde\bZ^u}(x')
		\right]
		P_t^u\left(
			x\mid x',
			\widetilde\pi_t^u(x'),
			\widetilde Z_t^u
		\right)
	\right|.
	\label{eq_unif-cont-mupiZ-diff}
\end{align}

For the first term in~\eqref{eq_unif-cont-mupiZ-diff}, by
Assumption~\ref{assp_P_lip}%
\ref{assp_z_to_P_lip},
\[
	\sum_{x'\in\cX}
	\mu_t^{u,\bpi^u,\bZ^u}(x')
	\left|
		P_t^u(x\mid x',\pi_t^u(x'),Z_t^u)
		-
		P_t^u(x\mid x',\pi_t^u(x'),\widetilde Z_t^u)
	\right|
	\leq
	L_{P,Z}
	\|\bZ^u-\widetilde\bZ^u\|_\infty.
\]
For the second term, by
Assumption~\ref{assp_P_lip}%
\ref{assp_a_to_P_lip},
\[
	\sum_{x'\in\cX}
	\mu_t^{u,\bpi^u,\bZ^u}(x')
	\left|
		P_t^u(x\mid x',\pi_t^u(x'),\widetilde Z_t^u)
		-
		P_t^u(
			x\mid x',
			\widetilde\pi_t^u(x'),
			\widetilde Z_t^u
		)
	\right|
	\leq
	L_{P,a}
	\|\bpi^u-\widetilde\bpi^u\|_{\infty,\infty}.
\]
For the third term, the induction hypothesis gives
\[
	\max_{x\in\cX}
	\left|
		\mu_t^{u,\bpi^u,\bZ^u}(x)
		-
		\mu_t^{u,\widetilde\bpi^u,\widetilde\bZ^u}(x)
	\right|
	\leq
	L_{m,t}
	\max\left\{
		\|\bpi^u-\widetilde\bpi^u\|_{\infty,\infty},
		\|\bZ^u-\widetilde\bZ^u\|_\infty
	\right\}.
\]
Consequently,
\begin{align*}
	&
	\left|
		\sum_{x'\in\cX}
		\left[
			\mu_t^{u,\bpi^u,\bZ^u}(x')
			-
			\mu_t^{u,\widetilde\bpi^u,\widetilde\bZ^u}(x')
		\right]
		P_t^u\left(
			x\mid x',
			\widetilde\pi_t^u(x'),
			\widetilde Z_t^u
		\right)
	\right|
	\\
	&\leq
	L_{m,t}|\cX|
	\max\left\{
		\|\bpi^u-\widetilde\bpi^u\|_{\infty,\infty},
		\|\bZ^u-\widetilde\bZ^u\|_\infty
	\right\}.
\end{align*}
Combining the bounds, we obtain
\begin{align*}
	\left|
		\mu_{t+1}^{u,\bpi^u,\bZ^u}(x)
		-
		\mu_{t+1}^{u,\widetilde\bpi^u,\widetilde\bZ^u}(x)
	\right|
	\leq
	\left(
		L_{m,t}|\cX|
		+
		L_{P,Z}
		+
		L_{P,a}
	\right)
	\max\left\{
		\|\bpi^u-\widetilde\bpi^u\|_{\infty,\infty},
		\|\bZ^u-\widetilde\bZ^u\|_\infty
	\right\}.
\end{align*}
Thus, setting $L_{m,t+1} := L_{m,t}|\cX| + L_{P,Z} + L_{P,a}$,
we deduce that
$(\bpi^u,\bZ^u)\mapsto
\mu_{t+1}^{u,\bpi^u,\bZ^u}$ is Lipschitz continuous. By induction,
\begin{equation}\label{eq_mudiff-piZ}
	\left\|
		\bmu^{u,\bpi^u,\bZ^u}
		-
		\bmu^{u,\widetilde\bpi^u,\widetilde\bZ^u}
	\right\|_{\infty,\infty}
	\leq
	K_m
	\left\|
		(\bpi^u,\bZ^u)
		-
		(\widetilde\bpi^u,\widetilde\bZ^u)
	\right\|_\infty,
\end{equation}
where $K_m := \max_{t\in\ocT}L_{m,t}.$

We now prove the Lipschitz continuity of
$\buZ\mapsto\bupi^{*,\buZ,\buV^{\buZ}}$ in
item~\ref{lemma_sub_ZtoPi_lip}. For each $u\in I$, Lemmas
\ref{lemma_ZtoA_lip} and~\ref{lemma_ztoVz_lip} give
\begin{align*}
	\left\|
		\bupi^{*,\buZ,\buV^{\buZ},u}
		-
		\bupi^{*,\widetilde\buZ,
		\buV^{\widetilde\buZ},u}
	\right\|_{\infty,\infty}
	\leq &
	L_{a,Z}
	\|\bZ^u-\widetilde\bZ^u\|_\infty
	+
	L_{a,V}
	\left\|
		\bV^{\bZ^u,u}
		-
		\bV^{\widetilde\bZ^u,u}
	\right\|_{\infty,\infty}
	\\
	\leq &
	\left(
		L_{a,Z}
		+
		L_{a,V}K_{V,Z}
	\right)
	\|\bZ^u-\widetilde\bZ^u\|_\infty.
\end{align*}
Therefore,
\begin{align*}
	\left\|
		\bupi^{*,\buZ,\buV^{\buZ}}
		-
		\bupi^{*,\widetilde\buZ,
		\buV^{\widetilde\buZ}}
	\right\|_{\infty,\infty,2}
	\leq&
	\left(
		\int_I
		\left\|
			\bupi^{*,\buZ,\buV^{\buZ},u}
			-
			\bupi^{*,\widetilde\buZ,
			\buV^{\widetilde\buZ},u}
		\right\|_{\infty,\infty}^2du
	\right)^{1/2}
	\\
	\leq&
	\left(
		L_{a,Z}
		+
		L_{a,V}K_{V,Z}
	\right)
	\left(
		\int_I
		\|\bZ^u-\widetilde\bZ^u\|_\infty^2du
	\right)^{1/2}
	\\
	\leq&
	\sqrt T
	\left(
		L_{a,Z}
		+
		L_{a,V}K_{V,Z}
	\right)
	\|\buZ-\widetilde\buZ\|_{\infty,2}.
\end{align*}
Thus, item~\ref{lemma_sub_ZtoPi_lip} holds with $K_{\pi,Z} = \sqrt T \left( L_{a,Z} + L_{a,V}K_{V,Z} \right).$

Notice that the preceding calculation also gives the stronger estimate
\begin{equation}\label{eq:strong-policy-estimate}
	\left(
		\int_I
		\left\|
			\bupi^{*,\buZ,\buV^{\buZ},u}
			-
			\bupi^{*,\widetilde\buZ,
			\buV^{\widetilde\buZ},u}
		\right\|_{\infty,\infty}^2du
	\right)^{1/2}
	\leq
	K_{\pi,Z}
	\|\buZ-\widetilde\buZ\|_{\infty,2}.
\end{equation}

Finally, we prove the Lipschitz continuity of
$\buZ\mapsto\bumu^{\buZ}$ in
item~\ref{lemma_sub_ZtoMu_lip}. By~\eqref{eq_mudiff-piZ},
\begin{align*}
	\|\bumu^{\buZ}
	-
	\bumu^{\widetilde\buZ}\|_{\infty,\infty,2}
	&\leq
	\left(
		\int_I
		\|\bmu^{\buZ,u}
		-
		\bmu^{\widetilde\buZ,u}\|_{\infty,\infty}^2du
	\right)^{1/2}
	\\
	&\leq
	K_m
	\left(
		\int_I
		\max\left\{
			\|\bZ^u-\widetilde\bZ^u\|_\infty,
			\left\|
				\bupi^{*,\buZ,\buV^{\buZ},u}
				-
				\bupi^{*,\widetilde\buZ,
				\buV^{\widetilde\buZ},u}
			\right\|_{\infty,\infty}
		\right\}^2du
	\right)^{1/2}
	\\
	&\leq
	K_m
	\left(
		\int_I
		\|\bZ^u-\widetilde\bZ^u\|_\infty^2du
	\right)^{1/2}
	+
	K_m
	\left(
		\int_I
		\left\|
			\bupi^{*,\buZ,\buV^{\buZ},u}
			-
			\bupi^{*,\widetilde\buZ,
			\buV^{\widetilde\buZ},u}
		\right\|_{\infty,\infty}^2du
	\right)^{1/2}
	\\
	&\leq
	K_m
	\left(
		\sqrt T+K_{\pi,Z}
	\right)
	\|\buZ-\widetilde\buZ\|_{\infty,2},
\end{align*}
where the last inequality follows from
\eqref{eq:strong-policy-estimate}. Therefore,
\[
	\|\bumu^{\buZ}
	-
	\bumu^{\widetilde\buZ}\|_{\infty,\infty,2}
	\leq
	K_{\mu,Z}
	\|\buZ-\widetilde\buZ\|_{\infty,2},
\]
with $K_{\mu,Z} = K_m \left( \sqrt T+K_{\pi,Z} \right)$.

\qed

\subsection{Proof of Proposition~\ref{lemma_XZconvergence}}\label{proof:lemma_XZconvergence}
The proof relies on several auxiliary results presented below.
Throughout this subsection, suppose that
Assumptions~\ref{aasp_graphon-assumption},
\ref{asm:varphi-regularity},
\ref{assp_diff_concave_faz},
\ref{assp_P_lip},
\ref{assp_varphi_lip}, and~\ref{assp_cutnorm_convergence} hold.

For each $i\in[N]$, let
$
	\upi_t^{N,i}
	=
	\upi_t^{N,(\widetilde\bpi,\Lambda^N(\bupi)^{-i})}
$
be defined, for all $x\in\cX$, $u\in I$, and $t\in\cT$, by
\begin{equation}\label{eq_piNlambda}
	\pi_t^{N,i,u}(x)
	:=
	\sum_{j\in[N]\setminus\{i\}}
	\mathds{1}_{u\in
	\left(\frac{j-1}{N},\joN\right]}
	\Lambda^N(\bupi)_t^j(x)
	+
	\mathds{1}_{u\in
	\left(\frac{i-1}{N},\ioN\right]}
	\widetilde\pi_t(x).
\end{equation}
When the deviating index $i$ is fixed, we write $\upi_t^N$ in place
of $\upi_t^{N,i}$.

Moreover, for $u\in I$, $G\in\cW$,
$\umu\in L^2(I;\cP(\cX))$,
$\upi\in L^2(I;\cA)^\cX$, and $t\in\cT$, we define
\begin{equation}\label{eq_Z_Gmupi_func}
	\mathfrak Z_t^u(G,\umu,\upi)
	:=
	\int_{v\in I}
	G(u,v)
	\sum_{x\in\cX}
	\varphi_t(x,\pi^v(x))
	\mu^v(x)\,dv.
\end{equation}
For a bounded measurable function
$f:\cX\times I\to\RR$, we also write
$
	\umu(f)
	:=
	\int_I
	\sum_{x\in\cX}
	f^u(x)\mu^u(x)\,du.
$
In particular, for a temporal sequence $\bumu$, the notation
$\umu_t(f)$ refers to the preceding expression evaluated at time $t$.

\begin{lemma}\label{lemma_ZtoZ_convergence}
	Let $\bupi\in\buPi^{L_\pi}$,
	$\widetilde\bpi\in\bfPi$, and $\bumu\in\bucM$. For each
	$i\in[N]$, consider the step policy $\upi_t^{N,i} = \upi_t^{N,(\widetilde\bpi,\Lambda^N(\bupi)^{-i})}$
	defined in~\eqref{eq_piNlambda}. If $(G_N)_N$ is a sequence of
	step-graphons converging to a graphon $G$ under
	Assumption~\ref{assp_cutnorm_convergence}, then, for every
	$\epsilon,p>0$, there exists $N_0$ such that, for all
	$N\geq N_0$, there exists a subset
	$\cI_N\subseteq[N]$ with
	$|\cI_N|\geq(1-p)N$ such that, for every $i\in\cI_N$,
	\[
		\left|
			\mathfrak Z_t^\ioN
			\left(
				G,\umu_t,\upi_t
			\right)
			-
			\mathfrak Z_t^\ioN
			\left(
				G_N,\umu_t,\upi_t^{N,i}
			\right)
		\right|
		\leq\epsilon,
		\qquad
		\forall\,t\in\cT.
	\]
	The integer $N_0$ can be chosen uniformly with respect to
	$\bupi\in\buPi^{L_\pi}$,
	$\widetilde\bpi\in\bfPi$, and $\bumu\in\bucM$.
\end{lemma}

\begin{proof}
	See Section~\ref{proof:lemma_ZtoZ_convergence}.
\end{proof}

\begin{lemma}\label{lemma_muNtomu}
	Let $\bupi\in\buPi^{L_\pi}$ and
	$\widetilde\bpi\in\bfPi$. For each $i\in[N]$, consider the state
	distribution sequences
	\[
		\bumu^{N,i}
		=
		\left(
			\mu_t^{N,(\widetilde\bpi,
			\Lambda^N(\bupi)^{-i}),u}
		\right)_{u\in I,t\in\ocT}
		\qquad\mbox{and}\qquad
		\bumu
		=
		\left(
			\mu_t^{\bupi,u}
		\right)_{u\in I,t\in\ocT},
	\]
	defined respectively in~\eqref{eq_muN} and
	Definition~\ref{def_GMFG_mu_Z}. For $M>0$, let $\cF_M$ be the set of measurable functions $\cX\times I\ni(x,u) \longmapsto f^u(x):=f(x,u)\in\RR$ satisfying $\sup_{x\in\cX,\,u\in I}|f^u(x)|\leq M$. 
	Then
	\[
		\sup_{f\in\cF_M}
		\E\left[
			\left|
				\umu_t^{N,i}(f)-\umu_t(f)
			\right|
		\right]
		\longrightarrow0
		\qquad\mbox{as }N\to+\infty,
		\qquad
		\forall\,t\in\ocT.
	\]
	The convergence is uniform with respect to
	$i\in[N]$, $\bupi\in\buPi^{L_\pi}$, and
	$\widetilde\bpi\in\bfPi$.
\end{lemma}

\begin{proof}
	See Section~\ref{proof:lemma_muNtomu}.
\end{proof}

\begin{lemma}\label{coro_ZNtoZ_with_mu}
Let $\bupi\in\buPi^{L_\pi}$ and $\widetilde\bpi\in\bfPi$. For each $i\in[N]$, let $\upi_t^{N,i} = \upi_t^{N,(\widetilde\bpi,\Lambda^N(\bupi)^{-i})} \in L^2(I;\cA)^\cX$ be defined as in~\eqref{eq_piNlambda}, and let $\bumu^{N,i} = \left( \mu_t^{N,(\widetilde\bpi, \Lambda^N(\bupi)^{-i}),u} \right)_{u\in I,t\in\ocT}$ be defined as in~\eqref{eq_muN}. Let $\bumu = \bumu^\bupi = \left( \mu_t^{\bupi,u} \right)_{u\in I,t\in\ocT}$ be the graphon mean-field generated by the background policy $\bupi$. Recall the notation $\mathfrak Z$ defined in~\eqref{eq_Z_Gmupi_func}.
	Then, for every $\epsilon,p>0$, there exists $N_0$ such that, for
	all $N\geq N_0$, there exists a subset
	$\cI_N\subseteq[N]$ with
	$|\cI_N|\geq(1-p)N$ such that, for every $i\in\cI_N$,
	\[
		\E\left[
			\left|
				\mathfrak Z_t^\ioN
				\left(
					G_N,\umu_t^{N,i},\upi_t^{N,i}
				\right)
				-
				\mathfrak Z_t^\ioN
				\left(
					G,\umu_t,\upi_t
				\right)
			\right|
		\right]
		\leq\epsilon,
		\qquad
		\forall\,t\in\cT.
	\]
	The integer $N_0$ can be chosen uniformly with respect to
	$\bupi\in\buPi^{L_\pi}$ and
	$\widetilde\bpi\in\bfPi$.
\end{lemma}

\begin{proof}
	See Section~\ref{proof:coro_ZNtoZ_with_mu}.
\end{proof}

We are now ready to prove Proposition~\ref{lemma_XZconvergence}.

\begin{proof}[Proof of Proposition~\ref{lemma_XZconvergence}]
	We split the proof into two steps.

	\noindent
	\textbf{Step 1:
	\eqref{eq_X_convergence} implies
	\eqref{eq_XZ_convergence}.}
	First, we prove that, for each $t\in\cT$,
	\eqref{eq_X_convergence} at time $t$ implies
	\eqref{eq_XZ_convergence} at time $t$.
	For any $t\in\cT$,
	\begin{align}
		 & \bigg |\,  \E \Big[\, \mathfrak{h} \big(X_t^i, Z^{N,(\tilde\bpi,\Lambda^N(\bupi)^{-i}),i}_t \,\big) \Big]
		- \E \Big[\, \mathfrak{h}(X_t^{\ioN}, Z^{\nu^\bupi,{\ioN}}_t\,  \big)\Big]  \bigg |
		\notag
		\\
		 & \leq  \bigg |  \E \Big[\, \mathfrak{h} \big(X_t^i, Z^{N,(\tilde\bpi,\Lambda^N(\bupi)^{-i}),i}_t \,\big) \Big] -  \E \Big[\, \mathfrak{h} \big( X_t^i, Z^{\nu^\bupi,{\ioN}}_t \,\big) \Big]\bigg |
		+ \bigg |\E \Big[\, \mathfrak{h} \big( X_t^i, Z^{\nu^\bupi,{\ioN}}_t \,\big) \Big] - \E \Big[\, \mathfrak{h}(X_t^{\ioN}, Z^{\nu^\bupi,{\ioN}}_t\,  \big)\Big] \bigg |
		\notag
		\\
		 & \leq  L_{\mathfrak{h}} \,
		\underbracket{\E \Big[ \Big|\,Z^{N,(\tilde\bpi,\Lambda^N(\bupi)^{-i}), i}_t - Z^{\nu^{\bupi},\ioN}_t \Big| \Big]}_{\text{Term 1}} + \underbracket{\bigg | \E \Big[\, \mathfrak{h} \big(X_t^i, Z^{\nu^\bupi,{\ioN}}_t\big) \Big]- \E \Big[\, \mathfrak{h} ( X_t^{\ioN},Z^{\nu^\bupi,{\ioN}}_t) \Big]   \bigg |}_{\text{Term 2}}.
		\label{eq:frakh-term1-term2}
	\end{align}

	Fix $\epsilon,p>0$. Term~2 in
	\eqref{eq:frakh-term1-term2} is bounded by
	$\epsilon/2$ using~\eqref{eq_X_convergence}, applied with
	tolerance $\epsilon/2$ and exceptional proportion $p/2$, and with
	the test function $h(x) = \mathfrak{h} \bigl( x,Z_t^{\bunu^\bupi,\ioN} \bigr)$.
	This function belongs to $\cH$ because
	$\mathfrak{h}$ is bounded by $C_h$. The estimate is uniform with
	respect to
	$\bupi\in\buPi^{L_\pi}$,
	$\widetilde\bpi\in\bfPi$, and
	$\mathfrak{h}\in\widetilde\cH$.
	We next consider Term~1 and apply
	Lemma~\ref{coro_ZNtoZ_with_mu}. We first match the notation in
	Term~1 with the notation $\mathfrak Z$ used in that lemma.
	Let $I_j^N := \bigl( \frac{j-1}{N},\frac{j}{N} \bigr]$, $j\in[N]$,
	and let $\bm{\pi}^{(N)} = \left( \widetilde\bpi, \Lambda^N(\bupi)^{-i} \right)$.
	Recalling the definition of
	$Z_t^{N,\bm{\pi}^{(N)},i}$ in~\eqref{eq:ZNpii} and the definition
	of $G_N$ in~\eqref{eq_step_graphon}, we have, for
	$u\in I_i^N$,
	\begin{align}
		Z_t^{N,(\widetilde\bpi,
		\Lambda^N(\bupi)^{-i}),i}
		&=
		\frac1N
		\sum_{j=1}^N
		\zeta_{ij}^N
		\varphi_t
		\left(
			X_t^j,
			\pi_t^{(N),j}(X_t^j)
		\right)
		\notag
		\\
		&=
		\sum_{j=1}^N
		\int_{v\in I_j^N}
		G_N(u,v)
		\sum_{x\in\cX}
		\varphi_t
		\left(
			x,\pi_t^{(N),j}(x)
		\right)
		\mathds{1}_{X_t^j=x}\,dv.
		\label{eq:tmp-ZN}
	\end{align}
	On the other hand, by
	\eqref{eq_Z_Gmupi_func},
	\eqref{eq_piNlambda}, and~\eqref{eq_muN},
	\begin{align}
		\mathfrak Z_t^u
		\left(
			G_N,\umu_t^{N,i},\upi_t^{N,i}
		\right)
		&=
		\int_{v\in I}
		G_N(u,v)
		\sum_{x\in\cX}
		\varphi_t
		\left(
			x,\pi_t^{N,i,v}(x)
		\right)
		\mu_t^{N,i,v}(x)\,dv
		\notag
		\\
		&=
		\sum_{j=1}^N
		\int_{v\in I_j^N}
		G_N(u,v)
		\sum_{x\in\cX}
		\varphi_t
		\left(
			x,\pi_t^{N,i,v}(x)
		\right)
		\mathds{1}_{X_t^j=x}\,dv.
		\label{eq:tmp-frakZN}
	\end{align}
	Hence, identifying
	\eqref{eq:tmp-ZN} and~\eqref{eq:tmp-frakZN} at
	$u=i/N$, we obtain
	\[
		\mathfrak Z_t^\ioN
		\left(
			G_N,\umu_t^{N,i},\upi_t^{N,i}
		\right)
		=
		Z_t^{N,(\widetilde\bpi,
		\Lambda^N(\bupi)^{-i}),i}.
	\]

	We now turn to
	$Z_t^{\bunu^\bupi,\ioN}$. By~\eqref{def_Znu}, recalling that
	$\varphi$ is independent of the player label in this section,
	\begin{align*}
		Z_t^{\bunu^\bupi,\ioN}
		&=
		\int_{v\in I}
		G(\ioN,v)
		\sum_{x\in\cX}
		\int_\cA
		\varphi_t(x,a)
		\nu_t^{\bupi,v}(x,da)\,dv
		\\
		&=
		\int_{v\in I}
		G(\ioN,v)
		\sum_{x\in\cX}
		\varphi_t
		\left(
			x,\pi_t^v(x)
		\right)
		\mu_t^{\bupi,v}(x)\,dv
		\\
		&=
		\mathfrak Z_t^\ioN
		\left(
			G,\umu_t^\bupi,\upi_t
		\right).
	\end{align*}

	We can therefore apply
	Lemma~\ref{coro_ZNtoZ_with_mu} to Term~1 with tolerance
	$\epsilon/(2L_h)$ and exceptional proportion $p/2$. By the
	uniformity in that lemma, the same subset can be chosen
	independently of
	$\widetilde\bpi\in\bfPi$. Intersecting the subsets used for
	Terms~1 and~2 gives a subset
	$\cI_N\subseteq[N]$ satisfying $|\cI_N| \geq (1-p)N$
	such that, for every $i\in\cI_N$,
	\[
		\sup_{\widetilde\bpi\in\bfPi}
		\sup_{\mathfrak{h}\in\widetilde\cH}
		\left|
			\E\left[
				\mathfrak{h}
				\left(
					X_t^i,
					Z_t^{N,(\widetilde\bpi,
					\Lambda^N(\bupi)^{-i}),i}
				\right)
			\right]
			-
			\E\left[
				\mathfrak{h}
				\left(
					X_t^{\ioN},
					Z_t^{\bunu^\bupi,\ioN}
				\right)
			\right]
		\right|
		<\epsilon.
	\]
	This proves~\eqref{eq_XZ_convergence} at time $t$.

	\medskip

	\noindent
	\textbf{Step 2: Proof of~\eqref{eq_X_convergence}.}
	We prove~\eqref{eq_X_convergence} by induction.

	At $t=0$, since $X_0^i$ and $X_0^{\ioN}$ have the same
	distribution,
	\[
		\sup_{\widetilde\bpi\in\bfPi}
		\sup_{h\in\cH}
		\left|
			\E[h(X_0^i)]
			-
			\E[h(X_0^{\ioN})]
		\right|
		=
		0,
		\qquad
		\forall\,i\in[N].
	\]

	Let $t\in\cT$ and assume that
	\eqref{eq_X_convergence} holds at time $t$, in the sense that,
	for every $\epsilon,p>0$, there exists $N_0$ such that, for all
	$N\geq N_0$, there exists a subset
	$\cI_N\subseteq[N]$ with
	$|\cI_N|\geq(1-p)N$ on which the estimate holds uniformly over
	$\widetilde\bpi\in\bfPi$ and $h\in\cH$.

	For $h\in\cH$ and
	$\widetilde\bpi\in\bfPi$, define
	\[
		\mathfrak{h}(x,z)
		:=
		\sum_{y\in\cX}
		h(y)
		P_t
		\left(
			y\mid
			x,\widetilde\pi_t(x),z
		\right).
	\]
	Since $P_t(\cdot\mid x,a,z)$ is a probability distribution,
	$\mathfrak{h}$ is bounded by $C_h$. Moreover, by
	Assumption~\ref{assp_P_lip}%
	\ref{assp_z_to_P_lip},
	\begin{align*}
		|\mathfrak{h}(x,z)-\mathfrak{h}(x,\widetilde z)|
		&\leq
		C_h
		\sum_{y\in\cX}
		\left|
			P_t
			\left(
				y\mid
				x,\widetilde\pi_t(x),z
			\right)
			-
			P_t
			\left(
				y\mid
				x,\widetilde\pi_t(x),\widetilde z
			\right)
		\right|
		\\
		&\leq
		C_h|\cX|L_{P,Z}
		|z-\widetilde z|.
	\end{align*}
	Thus, $\mathfrak{h}$ belongs to the class
	$\widetilde\cH$ with Lipschitz constant $L_h = C_h|\cX|L_{P,Z}$,
	uniformly in
	$h\in\cH$ and
	$\widetilde\bpi\in\bfPi$. By the Markov transition dynamics,
	\[
		\E[h(X_{t+1}^i)]
		=
		\E\left[
			\mathfrak{h}
			\left(
				X_t^i,
				Z_t^{N,(\widetilde\bpi,
				\Lambda^N(\bupi)^{-i}),i}
			\right)
		\right],
	\]
	and similarly,
	\[
		\E[h(X_{t+1}^{\ioN})]
		=
		\E\left[
			\mathfrak{h}
			\left(
				X_t^{\ioN},
				Z_t^{\bunu^\bupi,\ioN}
			\right)
		\right].
	\]
	By Step~1, the induction hypothesis at time $t$ implies
	\eqref{eq_XZ_convergence} at time $t$. Therefore, for every
	$\epsilon,p>0$, there exists $N_0$ such that, for all
	$N\geq N_0$, there exists a subset
	$\cI_N\subseteq[N]$ with
	$|\cI_N|\geq(1-p)N$ such that, for every $i\in\cI_N$,
	\begin{align*}
		&
		\sup_{\widetilde\bpi\in\bfPi}
		\sup_{h\in\cH}
		\left|
			\E[h(X_{t+1}^i)]
			-
			\E[h(X_{t+1}^{\ioN})]
		\right|
		\\
		&=
		\sup_{\widetilde\bpi\in\bfPi}
		\sup_{h\in\cH}
		\left|
			\E\left[
				\mathfrak{h}
				\left(
					X_t^i,
					Z_t^{N,(\widetilde\bpi,
					\Lambda^N(\bupi)^{-i}),i}
				\right)
			\right]
			-
			\E\left[
				\mathfrak{h}
				\left(
					X_t^{\ioN},
					Z_t^{\bunu^\bupi,\ioN}
				\right)
			\right]
		\right|
		<\epsilon.
	\end{align*}
	This proves~\eqref{eq_X_convergence} at time $t+1$ and completes
	the induction.

	We have thus proved~\eqref{eq_X_convergence} for every
	$t\in\ocT$, and Step~1 gives
	\eqref{eq_XZ_convergence} for every $t\in\cT$. To obtain a single
	subset on which all these estimates hold simultaneously, apply each
	fixed-time estimate with exceptional proportion $\frac{p}{2T+1}$
	and intersect the resulting $T+1$ subsets corresponding to
	\eqref{eq_X_convergence} and the $T$ subsets corresponding to
	\eqref{eq_XZ_convergence}. The resulting subset
	$\cI_N\subseteq[N]$ satisfies $|\cI_N| \geq (1-p)N$.
	Taking the maximum of the finitely many corresponding values of
	$N_0$ proves the proposition.
\end{proof}

\subsubsection{Proof of Lemma~\ref{lemma_ZtoZ_convergence}}
\label{proof:lemma_ZtoZ_convergence}

For all $t\in\cT$ and $i\in[N]$, by the triangle inequality, we have
\begin{align*}
	&
	\left|
		\mathfrak{Z}^{\frac{i}{N}}_t
		\left(G,\umu_t,\upi_t\right)
		-
		\mathfrak{Z}^{\frac{i}{N}}_t
		\left(G_N,\umu_t,\upi^N_t\right)
	\right|
	\\
	&\qquad
	\leq
	\underbracket{
	\left|
		\mathfrak{Z}^{\frac{i}{N}}_t
		\left(G,\umu_t,\upi_t\right)
		-
		\mathfrak{Z}^{\frac{i}{N}}_t
		\left(G_N,\umu_t,\upi_t\right)
	\right|
	}_{\text{Term 1}}
+
	\underbracket{
	\left|
		\mathfrak{Z}^{\frac{i}{N}}_t
		\left(G_N,\umu_t,\upi_t\right)
		-
		\mathfrak{Z}^{\frac{i}{N}}_t
		\left(G_N,\umu_t,\upi^N_t\right)
	\right|
	}_{\text{Term 2}}.
\end{align*}

We start by bounding Term~1:
\begin{align*}
	&
	\left|
		\mathfrak{Z}^{\frac{i}{N}}_t
		\left(G,\umu_t,\upi_t\right)
		-
		\mathfrak{Z}^{\frac{i}{N}}_t
		\left(G_N,\umu_t,\upi_t\right)
	\right|
	\\
	&=
	\left|
		\int_{v\in I}
		\left(
			G\left(\tfrac{i}{N},v\right)
			-
			G_N\left(\tfrac{i}{N},v\right)
		\right)
		\sum_{x\in\cX}
		\varphi_t(x,\pi_t^v(x))
		\mu_t^v(x)\,dv
	\right|
	\\
	&\leq
	\underbracket{
	\left|
		\int_{v\in I}
		\left(
			G\left(\tfrac{i}{N},v\right)
			-
			N\int_{w\in
			(\frac{i-1}{N},\frac{i}{N}]}
			G(w,v)\,dw
		\right)
		\sum_{x\in\cX}
		\varphi_t(x,\pi_t^v(x))
		\mu_t^v(x)\,dv
	\right|
	}_{\text{Term 1.a}}
	\\
	&\quad+
	\underbracket{
	\left|
		\int_{v\in I}
		\left(
			N\int_{w\in
			(\frac{i-1}{N},\frac{i}{N}]}
			G(w,v)\,dw
			-
			G_N\left(\tfrac{i}{N},v\right)
		\right)
		\sum_{x\in\cX}
		\varphi_t(x,\pi_t^v(x))
		\mu_t^v(x)\,dv
	\right|
	}_{\text{Term 1.b}}.
\end{align*}

First, we bound Term~1.a:
\begin{align*}
	&
	\left|
		\int_{v\in I}
		\left(
			G\left(\tfrac{i}{N},v\right)
			-
			N\int_{w\in
			(\frac{i-1}{N},\frac{i}{N}]}
			G(w,v)\,dw
		\right)
		\sum_{x\in\cX}
		\varphi_t(x,\pi_t^v(x))
		\mu_t^v(x)\,dv
	\right|
	\\
	&\leq
	\left|
		\int_{v\in I}
		N\int_{w\in
		(\frac{i-1}{N},\frac{i}{N}]}
		\left(
			G\left(\tfrac{\lceil Nw\rceil}{N},v\right)
			-
			G(w,v)
		\right)
		\sum_{x\in\cX}
		\varphi_t(x,\pi_t^v(x))
		\mu_t^v(x)\,dw\,dv
	\right|
	\\
	&\leq
	|\cX|\,\omega^G(1/N)\,C_\varphi,
\end{align*}
where we used the Cauchy--Schwarz inequality in the last step. Due to
Assumption~\ref{aasp_graphon-assumption}, this term converges to zero
as $N\to\infty$, uniformly with respect to
$\bupi\in\buPi^{L_\pi}$, $\widetilde\bpi\in\bfPi$, and
$\bumu\in\bucM$.

We next consider Term~1.b. Since $G_N$ is constant in its first
variable on each interval
$\left(\frac{i-1}{N},\frac{i}{N}\right]$, we have, for a.e.\ $v\in I$,
$
	G_N\left(\tfrac{i}{N},v\right)
	=
	N\int_{w\in
	(\frac{i-1}{N},\frac{i}{N}]}
	G_N(w,v)\,dw.
$
Therefore,
\begin{align*}
	\text{Term 1.b}
	&\leq
	N\int_{w\in
	(\frac{i-1}{N},\frac{i}{N}]}
	\left|
		\int_{v\in I}
		\left(
			G(w,v)-G_N(w,v)
		\right)
		\sum_{x\in\cX}
		\varphi_t(x,\pi_t^v(x))
		\mu_t^v(x)\,dv
	\right|dw.
\end{align*}

We define
\[
	U_i^N
	:=
	\max_{t\in\cT}
	N\int_{w\in
	(\frac{i-1}{N},\frac{i}{N}]}
	\left|
		\int_{v\in I}
		\left(
			G(w,v)-G_N(w,v)
		\right)
		\sum_{x\in\cX}
		\varphi_t(x,\pi_t^v(x))
		\mu_t^v(x)\,dv
	\right|dw.
\]
Then Term~1.b is bounded by $U_i^N$ for every $t\in\cT$. Moreover,
since
\[
	\left|
		\sum_{x\in\cX}
		\varphi_t(x,\pi_t^v(x))
		\mu_t^v(x)
	\right|
	\leq C_\varphi,
\]
we have
\begin{align*}
	\frac1N\sum_{i=1}^NU_i^N
	&\leq
	\sum_{t\in\cT}
	\int_{w\in I}
	\left|
		\int_{v\in I}
		\left(
			G(w,v)-G_N(w,v)
		\right)
		\sum_{x\in\cX}
		\varphi_t(x,\pi_t^v(x))
		\mu_t^v(x)\,dv
	\right|dw
	\\
	&\leq
	|\cT|\,C_\varphi
	\|G-G_N\|_{\infty\to1}
	\longrightarrow0
\end{align*}
by Assumption~\ref{assp_cutnorm_convergence}. This convergence is
uniform with respect to
$\bupi\in\buPi^{L_\pi}$ and $\bumu\in\bucM$.

Fix $\epsilon_1,p>0$, and define
\[
	B_N
	:=
	\left\{
		i\in[N]:
		U_i^N\geq\epsilon_1
	\right\}.
\]
By Markov's inequality,
\[
	\frac{|B_N|}{N}
	\leq
	\frac{1}{\epsilon_1}
	\frac1N\sum_{i=1}^NU_i^N
	\longrightarrow0.
\]
Hence, there exists $N_0\in\mathbb N$ such that, for all $N\geq N_0$, $|B_N|\leq pN$. Setting $\cI_N:=[N]\setminus B_N$, we conclude that $|\cI_N|\geq(1-p)N$ and $U_i^N<\epsilon_1$, for all $i\in\cI_N$.

We continue by bounding Term~2:
\begin{align*}
	&
	\left|
		\mathfrak{Z}^{\frac{i}{N}}_t
		\left(G_N,\umu_t,\upi_t\right)
		-
		\mathfrak{Z}^{\frac{i}{N}}_t
		\left(G_N,\umu_t,\upi_t^N\right)
	\right|
	\\
	&=
	\left|
		\int_{v\in I}
		G_N\left(\tfrac{i}{N},v\right)
		\sum_{x\in\cX}
		\left(
			\varphi_t(x,\pi_t^{N,v}(x))
			-
			\varphi_t(x,\pi_t^v(x))
		\right)
		\mu_t^v(x)\,dv
	\right|
	\\
	&\leq
	\int_{v\in I}
		\left|
			G_N\left(\tfrac{i}{N},v\right)
		\right|
		\sum_{x\in\cX}
		\left|
			\varphi_t(x,\pi_t^{N,v}(x))
			-
			\varphi_t(x,\pi_t^v(x))
		\right|
		\mu_t^v(x)\,dv
	\\
	&\leq
	L_\varphi
	\int_{v\in I}
	\sum_{x\in\cX}
	\left|
		\pi_t^{N,v}(x)-\pi_t^v(x)
	\right|
	\mu_t^v(x)\,dv.
\end{align*}

By the definition of the step policy $\upi^N$, the interval
$\left(\frac{i-1}{N},\frac{i}{N}\right]$ corresponds to the unique
block where the deviating policy $\widetilde\bpi$ is used, whereas on
every other block
$\left(\frac{j-1}{N},\frac{j}{N}\right]$, $j\neq i$, the policy
coincides with the sampled graphon policy. Therefore,
\begin{align*}
	&
	\int_{v\in I}
	\sum_{x\in\cX}
	\left|
		\pi_t^{N,v}(x)-\pi_t^v(x)
	\right|
	\mu_t^v(x)\,dv
	\\
	&=
	\sum_{j\in[N]\setminus\{i\}}
	\int_{v\in
	(\frac{j-1}{N},\frac{j}{N}]}
	\sum_{x\in\cX}
	\left|
		\pi_t^{\frac{\lceil Nv\rceil}{N}}(x)
		-
		\pi_t^v(x)
	\right|
	\mu_t^v(x)\,dv+
	\int_{v\in
	(\frac{i-1}{N},\frac{i}{N}]}
	\sum_{x\in\cX}
	\left|
		\widetilde\pi_t(x)-\pi_t^v(x)
	\right|
	\mu_t^v(x)\,dv
	\\
	&\leq
	\sum_{j\in[N]\setminus\{i\}}
	\int_{v\in
	(\frac{j-1}{N},\frac{j}{N}]}
	L_\pi
	\left|
		\frac{\lceil Nv\rceil}{N}-v
	\right|
	\sum_{x\in\cX}\mu_t^v(x)\,dv+
	2C_\cA
	\int_{v\in
	(\frac{i-1}{N},\frac{i}{N}]}
	\sum_{x\in\cX}\mu_t^v(x)\,dv.
\end{align*}
Since $|\lceil Nv\rceil-Nv|\leq1$, it follows that $\left| \frac{\lceil Nv\rceil}{N}-v \right| \leq\frac1N$, and thus
\[
	\int_{v\in I}
	\sum_{x\in\cX}
	\left|
		\pi_t^{N,v}(x)-\pi_t^v(x)
	\right|
	\mu_t^v(x)\,dv
	\leq
	\frac{L_\pi+2C_\cA}{N}.
\]
Combining the bounds yields
\[
	\left|
		\mathfrak{Z}^{\frac{i}{N}}_t
		\left(G_N,\umu_t,\upi_t\right)
		-
		\mathfrak{Z}^{\frac{i}{N}}_t
		\left(G_N,\umu_t,\upi_t^N\right)
	\right|
	\leq
	\frac{L_\varphi(L_\pi+2C_\cA)}{N},
\]
which converges to zero as $N\to\infty$, uniformly with respect to
$\bupi\in\buPi^{L_\pi}$,
$\widetilde\bpi\in\bfPi$, and $\bumu\in\bucM$.

Finally, set $\epsilon_1=\epsilon/3$. Increasing $N_0$ if necessary, we may ensure that $|\cX|C_\varphi\omega^G(1/N) \leq\frac{\epsilon}{3}$ and $L_\varphi(L_\pi+2C_\cA)/N \leq \epsilon/3$. Therefore, for every $N\geq N_0$, there exists a subset $\cI_N\subseteq[N]$ of indices with $|\cI_N|\geq(1-p)N$ such that, for every $i\in\cI_N$,
\[
	\left|
		\mathfrak{Z}^{\frac{i}{N}}_t
		\left(G,\umu_t,\upi_t\right)
		-
		\mathfrak{Z}^{\frac{i}{N}}_t
		\left(G_N,\umu_t,\upi_t^N\right)
	\right|
	\leq\epsilon,
	\qquad
	\forall\,t\in\cT.
\]
\qed

\subsubsection{Proof of Lemma~\ref{lemma_muNtomu}}
\label{proof:lemma_muNtomu}

Let $C_{\varphi,M}:=\max\{C_\varphi,M\}$.
Since $\cF_M\subseteq\cF_{C_{\varphi,M}}$, it is sufficient to prove
the result for the larger class $\cF_{C_{\varphi,M}}$. For each
deviating index $i\in[N]$, we write $\bumu^N$ and $\upi_t^N$ in place
of $\bumu^{N,i}$ and $\upi_t^{N,i}$, respectively. We prove by
induction on $t$ the stronger assertion
\[
	\sup_{i\in[N]}
	\sup_{\bupi\in\buPi^{L_\pi}}
	\sup_{\widetilde\bpi\in\bfPi}
	\sup_{f\in\cF_{C_{\varphi,M}}}
	\E\left[
		\left|
			\umu_t^N(f)-\umu_t(f)
		\right|
	\right]
	\longrightarrow0
	\qquad\text{as }N\to\infty.
\]

For the initial step $t=0$, note that the random variables
$(X_0^j)_{j\in[N]}$ are independent and have common distribution
$\mu_{\mathrm{init}}$. By the Cauchy--Schwarz inequality and
$|f|\leq C_{\varphi,M}$,
\begin{align*}
	&
	\E\left[
		\left|
			\umu_0^N(f)-\umu_0(f)
		\right|
	\right]
	\\
	&=
	\E\Bigg[
	\Bigg|
		\int_{u\in I}
		\left(
			\sum_{x\in\cX}
			f^u(x)\mu_0^{N,u}(x)
			-
			\sum_{x\in\cX}
			f^u(x)\mu_{\mathrm{init}}(x)
		\right)du
	\Bigg|
	\Bigg]
	\\
	&=
	\E\Bigg[
	\Bigg|
		\sum_{j\in[N]}
		\int_{u\in
		(\frac{j-1}{N},\frac{j}{N}]}
		\left(
			f^u(X_0^j)
			-
			\E[f^u(X_0^j)]
		\right)du
	\Bigg|
	\Bigg]
	\\
	&\leq
	\Bigg(
	\E\Bigg[
	\Bigg|
		\sum_{j\in[N]}
		\left(
			\int_{u\in
			(\frac{j-1}{N},\frac{j}{N}]}
			f^u(X_0^j)\,du
			-
			\E\left[
				\int_{u\in
				(\frac{j-1}{N},\frac{j}{N}]}
				f^u(X_0^j)\,du
			\right]
		\right)
	\Bigg|^2
	\Bigg]
	\Bigg)^{1/2}
	\\
	&=
	\Bigg(
	\sum_{j\in[N]}
	\E\Bigg[
	\Bigg|
		\int_{u\in
		(\frac{j-1}{N},\frac{j}{N}]}
		f^u(X_0^j)\,du
		-
		\E\left[
			\int_{u\in
			(\frac{j-1}{N},\frac{j}{N}]}
			f^u(X_0^j)\,du
		\right]
	\Bigg|^2
	\Bigg]
	\Bigg)^{1/2}
	\\
	&\leq
	\sqrt{
		N
		\left(
			\frac{2C_{\varphi,M}}{N}
		\right)^2
	}
	=
	\frac{2C_{\varphi,M}}{\sqrt N}.
\end{align*}
The bound is independent of $i$, $\bupi$, $\widetilde\bpi$, and $f$.
Hence, the claim holds at $t=0$.

For the induction step, fix $t\in\cT$ and assume that
\[
	\sup_{i\in[N]}
	\sup_{\bupi\in\buPi^{L_\pi}}
	\sup_{\widetilde\bpi\in\bfPi}
	\sup_{f\in\cF_{C_{\varphi,M}}}
	\E\left[
		\left|
			\umu_t^N(f)-\umu_t(f)
		\right|
	\right]
	\longrightarrow0.
\]

Given
$\ueta,\widetilde\ueta\in L^2(I;\cP(\cX))$ and
$\upi,\widetilde\upi\in L^2(I;\cA)^\cX$, define the operator
$
	\bfT^\upi_{G,\widetilde\ueta,\widetilde\upi}
	:
	L^2(I;\cP(\cX))
	\longrightarrow
	L^2(I;\cP(\cX))
$
by
\[
	\left(
		\ueta\,
		\bfT^\upi_{G,\widetilde\ueta,\widetilde\upi}
	\right)^u
	=
	\sum_{x\in\cX}
	\eta^u(x)
	P_t
	\left(
		\cdot\mid
		x,\pi^u(x),
		\mathfrak Z_t^u
		(G,\widetilde\ueta,\widetilde\upi)
	\right),
	\qquad
	u\in I,
\]
so that
$
	\umu_{t+1}
	=
	\umu_t\,
	\bfT^{\upi_t}_{G,\umu_t,\upi_t}.
$

Let
$
	\upi_t^N
	=
	\upi_t^{N,(\widetilde\bpi,\Lambda^N(\bupi)^{-i})}
$
be defined as in~\eqref{eq_piNlambda}. By the triangle inequality,
\begin{align*}
	\E\left[
		\left|
			\umu_{t+1}^N(f)-\umu_{t+1}(f)
		\right|
	\right]
	&\leq
	\E\left[
		\left|
			\umu_{t+1}^N(f)
			-
			\umu_t^N
			\bfT^{\upi_t^N}_{G_N,\umu_t^N,\upi_t^N}(f)
		\right|
	\right]
	\\
	&\quad+
	\E\left[
		\left|
			\umu_t^N
			\bfT^{\upi_t^N}_{G_N,\umu_t^N,\upi_t^N}(f)
			-
			\umu_t^N
			\bfT^{\upi_t}_{G_N,\umu_t^N,\upi_t}(f)
		\right|
	\right]
	\\
	&\quad+
	\E\left[
		\left|
			\umu_t^N
			\bfT^{\upi_t}_{G_N,\umu_t^N,\upi_t}(f)
			-
			\umu_t^N
			\bfT^{\upi_t}_{G,\umu_t^N,\upi_t}(f)
		\right|
	\right]
	\\
	&\quad+
	\E\left[
		\left|
			\umu_t^N
			\bfT^{\upi_t}_{G,\umu_t^N,\upi_t}(f)
			-
			\umu_t^N
			\bfT^{\upi_t}_{G,\umu_t,\upi_t}(f)
		\right|
	\right]
	\\
	&\quad+
	\E\left[
		\left|
			\umu_t^N
			\bfT^{\upi_t}_{G,\umu_t,\upi_t}(f)
			-
			\umu_{t+1}(f)
		\right|
	\right]
	\\
	&=:I_1+I_2+I_3+I_4+I_5.
\end{align*}

For $I_1$, conditionally on
$\bm X_t=(X_t^j)_{j\in[N]}$, the variables
$(X_{t+1}^j)_{j\in[N]}$ are independent. Repeating the argument used
for $t=0$ gives
\begin{align*}
	I_1:= & \E \Bigl[\Bigl| \umu^N_{t+1}(f) - \umu^N_{t}\,\bfT^{\pi^N_t}_{\,  G_N, \,\umu^N_t, \,\pi^N_t } (f) \Bigr|\Bigr]                                                                                                                                              \\
	      & =
	\E \Bigl[\Bigl|\int_{u\in I}\sum_{x\in\cX} f^u(x) \,\mu_{t+1}^{N, u}(x) \, du
	\Bigr.\Bigr.                                                                                                                                                                                                                                                         \\
	      & \qquad\quad - \Bigl.\Bigl. \int_{u\in I}\sum_{x\in\cX} f^u(x) \sum_{x^\prime\in\cX} \mu^{N,u}_t(x^\prime) \,  P_t \Bigl(x \mid x^\prime,   \pi^{N,u}_t(x^\prime), \mathfrak{Z}^u_t\Bigl( G_N, \,\umu^N_t, \, \pi^N_t \Bigr) \Bigr)du
	\,\Bigr|\Bigr]                                                                                                                                                                                                                                                       \\
	      & = \E \Bigl[\Bigl|  \sum_{i\in [N]} \Bigl( \int_{u\in \big(\frac{i-1}{N},\ioN\big]} f^u(X^i_{t+1}) \, du - \E \bigg[\int_{u\in \big(\frac{i-1}{N},\ioN\big]} f^u(X^i_{t+1}) \,du \,\bigg|\, \bm{X}_t \bigg]  \Bigr)\Bigr|\Bigr]                               \\
	      & \leq \Bigl(\E \Bigl[\Bigl|  \sum_{i\in [N]} \Bigl( \int_{u\in \big(\frac{i-1}{N},\ioN\big]} f^u(X^i_{t+1}) \, du - \E \bigg[\int_{u\in \big(\frac{i-1}{N},\ioN\big]} f^u(X^i_{t+1}) \,du \,\bigg|\, \bm{X}_t \bigg]  \Bigr)\Bigr|^2\,\Bigr]\Bigr)^{\frac12}  \\
	      & = \Bigl(  \sum_{i\in [N]}  \E \Bigl[\Bigl|   \Bigl( \int_{u\in \big(\frac{i-1}{N},\ioN\big]} f^u(X^i_{t+1}) \, du - \E \bigg[\int_{u\in \big(\frac{i-1}{N},\ioN\big]} f^u(X^i_{t+1}) \,du \,\bigg|\, \bm{X}_t \bigg]  \Bigr)\Bigr|^2\,\Bigr]\Bigr)^{\frac12} \\
	      & \leq \sqrt{ N\cdot\Bigl(\frac{2 C_{\varphi,M}}{N}\Bigr)^2} = \frac{2C_{\varphi,M}}{\sqrt{N}},
\end{align*}
which converges to zero uniformly with respect to
$i$, $\bupi$, $\widetilde\bpi$, and $f$.

For $I_2$, by
Assumption~\ref{assp_P_lip}~\ref{assp_z_to_P_lip}%
~\ref{assp_a_to_P_lip}, the triangle inequality, and
$|f|\leq C_{\varphi,M}$,
\begin{align*}
	I_2
	&:=
	\E\left[
		\left|
			\umu_t^N
			\bfT^{\pi_t^N}_{G_N,\umu_t^N,\pi_t^N}(f)
			-
			\umu_t^N
			\bfT^{\upi_t}_{G_N,\umu_t^N,\upi_t}(f)
		\right|
	\right]
	\\
	&\leq
	\E\left[
		\left|
			\umu_t^N
			\bfT^{\pi_t^N}_{G_N,\umu_t^N,\pi_t^N}(f)
			-
			\umu_t^N
			\bfT^{\pi_t^N}_{G_N,\umu_t^N,\upi_t}(f)
		\right|
	\right]+
	\E\left[
		\left|
			\umu_t^N
			\bfT^{\pi_t^N}_{G_N,\umu_t^N,\upi_t}(f)
			-
			\umu_t^N
			\bfT^{\upi_t}_{G_N,\umu_t^N,\upi_t}(f)
		\right|
	\right]
	\\
	&\leq
	|\cX|C_{\varphi,M}L_{P,Z}
	\E\left[
		\int_{u\in I}
		\left|
			\mathfrak Z_t^u
			\left(G_N,\umu_t^N,\upi_t^N\right)
			-
			\mathfrak Z_t^u
			\left(G_N,\umu_t^N,\upi_t\right)
		\right|du
	\right]
	\\
	&\qquad+
	|\cX|C_{\varphi,M}L_{P,a}
	\E\left[
		\int_{u\in I}
		\sum_{x'\in\cX}
		\mu_t^{N,u}(x')
		\left|
			\pi_t^{N,u}(x')-\pi_t^u(x')
		\right|du
	\right].
\end{align*}

For the first integral, since $0\leq G_N\leq1$ and $\varphi$ is
$L_\varphi$-Lipschitz continuous,
\begin{align*}
	&
	\E\left[
		\int_{u\in I}
		\left|
			\mathfrak Z_t^u
			\left(G_N,\umu_t^N,\upi_t^N\right)
			-
			\mathfrak Z_t^u
			\left(G_N,\umu_t^N,\upi_t\right)
		\right|du
	\right]
	\\
	&\leq
	L_\varphi
	\E\left[
		\int_{u\in I}
		\int_{v\in I}
		G_N(u,v)
		\sum_{x\in\cX}
		\mu_t^{N,v}(x)
		\left|
			\pi_t^{N,v}(x)-\pi_t^v(x)
		\right|dv\,du
	\right]
	\\
	&\leq
	L_\varphi
	\E\left[
		\int_{v\in I}
		\sum_{x\in\cX}
		\mu_t^{N,v}(x)
		\left|
			\pi_t^{N,v}(x)-\pi_t^v(x)
		\right|dv
	\right].
\end{align*}

For the policy discrepancy, using that
$u\mapsto\pi_t^u(x)$ is $L_\pi$-Lipschitz and that actions are bounded
by $C_\cA$, we have
\begin{align*}
	&
	\E\left[
		\int_{u\in I}
		\sum_{x'\in\cX}
		\mu_t^{N,u}(x')
		\left|
			\pi_t^{N,u}(x')-\pi_t^u(x')
		\right|du
	\right]
	\\
	&=
	\E\left[
		\sum_{j\in[N]\setminus\{i\}}
		\int_{u\in
		(\frac{j-1}{N},\frac{j}{N}]}
		\left|
			\pi_t^{j/N}(X_t^j)
			-
			\pi_t^u(X_t^j)
		\right|du
	\right]+
	\E\left[
		\int_{u\in
		(\frac{i-1}{N},\frac{i}{N}]}
		\left|
			\widetilde\pi_t(X_t^i)
			-
			\pi_t^u(X_t^i)
		\right|du
	\right]
	\\
	&\leq
	\sum_{j\in[N]\setminus\{i\}}
	\int_{u\in
	(\frac{j-1}{N},\frac{j}{N}]}
	L_\pi
	\left|
		\frac{j}{N}-u
	\right|du
	+
	\frac{2C_\cA}{N}
	\leq
	\frac{(N-1)L_\pi}{N^2}
	+
	\frac{2C_\cA}{N}
	\leq
	\frac{L_\pi+2C_\cA}{N}.
\end{align*}

Combining the above estimates yields
\[
	I_2
	\leq
	|\cX|C_{\varphi,M}
	\left(
		L_{P,Z}L_\varphi+L_{P,a}
	\right)
	\frac{L_\pi+2C_\cA}{N},
\]
which converges to zero as $N\to\infty$, uniformly with respect to
$i\in[N]$, $\bupi\in\buPi^{L_\pi}$, and
$\widetilde\bpi\in\bfPi$.

For $I_3$, by Assumption~\ref{assp_P_lip}~\ref{assp_z_to_P_lip} and since $|f|\le C_{\varphi,M}$,
\begin{align*}
	 & I_3:=\E \Bigl[\Bigl| \umu^N_{t}\,\bfT^{\upi_t}_{\, G_N, \,\umu^N_t, \,\upi_t} (f) - \umu^N_{t}\,\bfT^{\upi_t}_{\, G, \,\umu^N_t, \,\upi_t}  (f) \Bigr|\Bigr]                                                                  \\
	 & =
	\E \Bigl[\Bigl| \int_{u\in I}\sum_{x\in\cX} f^u(x) \sum_{x^\prime\in\cX} \mu^{N,u}_t(x^\prime) \,  P_t \Bigl(x \mid x^\prime,\pi^{u}_t(x^\prime), \mathfrak{Z}^{u}_t\Bigl( G_N, \,\umu^N_t, \,\upi_t \Bigr) \Bigr) du
	\Bigr.\Bigr.                                                                                                                                                                                                                     \\
	 & \qquad\quad - \Bigl.\Bigl. \int_{u\in I}\sum_{x\in\cX} f^u(x) \sum_{x^\prime\in\cX} \mu^{N,u}_t(x^\prime) \,  P_t \Bigl(x \mid x^\prime,\pi^{u}_t(x^\prime), \mathfrak{Z}^{u}_t\Bigl( G, \,\umu^N_t, \,\upi_t \Bigr) \Bigr)du
	\,\Bigr|\Bigr]                                                                                                                                                                                                                   \\
	 & \leq
	C_{\varphi,M}L_{P,Z}\, |\cX| \, \E \Bigl[ \int_{u\in I} \Big| \mathfrak{Z}^{u}_t\Bigl( G_N, \,\umu^N_t, \,\upi_t\Bigr)  - \mathfrak{Z}^{u}_t\Bigl( G, \,\umu^N_t, \,\upi_t\Bigr)  \Big| \, du \,\Bigr].
\end{align*}

For every realization, since $\mu_t^{N,v}$ is a probability
distribution and $|\varphi_t(y,a)|\leq C_\varphi$, we have
\begin{align*}
	&
	\int_{u\in I}
	\left|
		\mathfrak{Z}^{u}_t
		\left(G_N,\umu^N_t,\upi_t\right)
		-
		\mathfrak{Z}^{u}_t
		\left(G,\umu^N_t,\upi_t\right)
	\right|du
	\\
	&=
	\int_{u\in I}
	\left|
		\int_{v\in I}
		\left(G_N(u,v)-G(u,v)\right)
		\sum_{y\in\cX}
		\varphi_t(y,\pi_t^v(y))
		\mu_t^{N,v}(y)\,dv
	\right|du
	\\
	&\leq
	C_\varphi
	\|G_N-G\|_{\infty\to1}.
\end{align*}
Therefore,
\[
	I_3
	\leq
	C_{\varphi,M}L_{P,Z}|\cX|C_\varphi
	\|G_N-G\|_{\infty\to1}
	\longrightarrow0
\]
as $N\to\infty$, uniformly with respect to
$i\in[N]$, $\bupi\in\buPi^{L_\pi}$,
$\widetilde\bpi\in\bfPi$, and $f\in\cF_{C_{\varphi,M}}$.

For $I_4$, again by
Assumption~\ref{assp_P_lip}~\ref{assp_z_to_P_lip} and
$|f|\leq C_{\varphi,M}$,
\begin{align*}
	I_4
	:=&\
	\E \Bigl[\Bigl|
	\umu^N_t\,
	\bfT^{\upi_t}_{G,\umu^N_t,\upi_t}(f)
	-
	\umu^N_t\,
	\bfT^{\upi_t}_{G,\umu_t,\upi_t}(f)
	\Bigr|\Bigr]
	\\
	=&\
	\E \Bigl[\Bigl|
	\int_{u\in I}
	\sum_{x\in\cX}
	f^u(x)
	\sum_{x'\in\cX}
	\mu_t^{N,u}(x')
	P_t\Bigl(
		x\mid x',
		\pi_t^u(x'),
		\mathfrak Z_t^u(G,\umu_t^N,\upi_t)
	\Bigr)du
	\\
	&\qquad -
	\int_{u\in I}
	\sum_{x\in\cX}
	f^u(x)
	\sum_{x'\in\cX}
	\mu_t^{N,u}(x')
	P_t\Bigl(
		x\mid x',
		\pi_t^u(x'),
		\mathfrak Z_t^u(G,\umu_t,\upi_t)
	\Bigr)du
	\Bigr|\Bigr]
	\\
	\leq&\
	|\cX|C_{\varphi,M}L_{P,Z}
	\E\Bigl[
	\int_{u\in I}
	\Bigl|
	\mathfrak Z_t^u(G,\umu_t^N,\upi_t)
	-
	\mathfrak Z_t^u(G,\umu_t,\upi_t)
	\Bigr|du
	\Bigr]
	\\
	=&\
	|\cX|C_{\varphi,M}L_{P,Z}
	\int_{u\in I}
	\E\Bigl[
	\Bigl|
	\int_{v\in I}
	G(u,v)
	\sum_{y\in\cX}
	\varphi_t(y,\pi_t^v(y))
	\left(
	\mu_t^{N,v}(y)-\mu_t^v(y)
	\right)dv
	\Bigr|
	\Bigr]du,
\end{align*}
where the last equality follows from Fubini's theorem.

For every $u\in I$, define $\widehat f^u(v,y) := G(u,v)\varphi_t(y,\pi_t^v(y))$,  $(v,y)\in I\times\cX$. By Assumption~\ref{assp_varphi_lip}, $|\widehat f^u(v,y)| \leq C_\varphi \leq C_{\varphi,M}$, so $\widehat f^u\in\cF_{C_{\varphi,M}}$. Hence,

\begin{align*}
	I_4
	&\leq
	|\cX|C_{\varphi,M}L_{P,Z}
	\int_{u\in I}
	\E\left[
	\left|
	\umu_t^N(\widehat f^u)
	-
	\umu_t(\widehat f^u)
	\right|
	\right]du
	\\
	&\leq
	|\cX|C_{\varphi,M}L_{P,Z}
	\sup_{g\in\cF_{C_{\varphi,M}}}
	\E\left[
	\left|
	\umu_t^N(g)-\umu_t(g)
	\right|
	\right].
\end{align*}
By the induction hypothesis, $I_4\to0$ as $N\to\infty$, uniformly
with respect to $i\in[N]$, $\bupi\in\buPi^{L_\pi}$,
$\widetilde\bpi\in\bfPi$, and
$f\in\cF_{C_{\varphi,M}}$.

Finally, for $I_5$,
\begin{align*}
	I_5
	:=\,
	&\E \Bigl[\Bigl|
	\umu^N_t\,
	\bfT^{\upi_t}_{G,\umu_t,\upi_t}(f)
	-
	\umu_{t+1}(f)
	\Bigr|\Bigr]
	\\
	=\,
	&\E \Bigl[\Bigl|
	\int_{u\in I}
	\sum_{x\in\cX}
	f^u(x)
	\sum_{x^\prime\in\cX}
	\mu_t^{N,u}(x^\prime)
	P_t\Bigl(
		x\mid x^\prime,
		\pi_t^u(x^\prime),
		\mathfrak{Z}_t^u(G,\umu_t,\upi_t)
	\Bigr)du
	\Bigr.\Bigr.
	\\
	&\qquad\quad
	-
	\Bigl.\Bigl.
	\int_{u\in I}
	\sum_{x\in\cX}
	f^u(x)
	\sum_{x^\prime\in\cX}
	\mu_t^u(x^\prime)
	P_t\Bigl(
		x\mid x^\prime,
		\pi_t^u(x^\prime),
		\mathfrak{Z}_t^u(G,\umu_t,\upi_t)
	\Bigr)du
	\Bigr|\Bigr]
	\\
	=\,
	&\E \Bigl[\Bigl|
	\int_{u\in I}
	\sum_{x\in\cX}
	f^u(x)
	\sum_{x^\prime\in\cX}
	\Bigl(
		\mu_t^{N,u}(x^\prime)
		-
		\mu_t^u(x^\prime)
		\Bigr)
	P_t\Bigl(
		x\mid x^\prime,
		\pi_t^u(x^\prime),
		\mathfrak{Z}_t^u(G,\umu_t,\upi_t)
	\Bigr)du
	\Bigr|\Bigr]
	\\
	=\,
	&\E \Bigl[\Bigl|
	\int_{u\in I}
	\sum_{x^\prime\in\cX}
	\widetilde f^u(x^\prime)
	\Bigl(
		\mu_t^{N,u}(x^\prime)
		-
		\mu_t^u(x^\prime)
		\Bigr)du
	\Bigr|\Bigr]
	\\
	=\,
	&\E\Bigl[
	\Bigl|
	\umu_t^N(\widetilde f)
	-
	\umu_t(\widetilde f)
	\Bigr|
	\Bigr],
\end{align*}
where
$
	\widetilde f^u(x^\prime)
	=
	\sum_{x\in\cX}
	f^u(x)\,
	P_t\Bigl(
		x\mid x^\prime,
		\pi_t^u(x^\prime),
		\mathfrak{Z}_t^u(G,\umu_t,\upi_t)
	\Bigr).
$
Since
$P_t(\cdot\mid x^\prime,\pi_t^u(x^\prime),
\mathfrak{Z}_t^u(G,\umu_t,\upi_t))$
is a probability distribution and
$|f^u(x)|\leq C_{\varphi,M}$, we have
\[
	\left|
	\widetilde f^u(x^\prime)
	\right|
	\leq
	\sum_{x\in\cX}
	|f^u(x)|
	P_t\Bigl(
		x\mid x^\prime,
		\pi_t^u(x^\prime),
		\mathfrak{Z}_t^u(G,\umu_t,\upi_t)
	\Bigr)
	\leq
	C_{\varphi,M}.
\]
Thus, $\widetilde f\in\cF_{C_{\varphi,M}}$. By the induction
hypothesis, $I_5\to0$ as $N\to\infty$, uniformly with respect to
$i\in[N]$, $\bupi\in\buPi^{L_\pi}$,
$\widetilde\bpi\in\bfPi$, and
$f\in\cF_{C_{\varphi,M}}$.

All five bounds are uniform with respect to
$i\in[N]$, $\bupi\in\buPi^{L_\pi}$,
$\widetilde\bpi\in\bfPi$, and
$f\in\cF_{C_{\varphi,M}}$. Combining them yields
\[
	\sup_{i\in[N]}
	\sup_{\bupi\in\buPi^{L_\pi}}
	\sup_{\widetilde\bpi\in\bfPi}
	\sup_{f\in\cF_{C_{\varphi,M}}}
	\E\left[
		\left|
			\umu_{t+1}^N(f)-\umu_{t+1}(f)
		\right|
	\right]
	\longrightarrow0,
\]
which completes the induction. Since
$\cF_M\subseteq\cF_{C_{\varphi,M}}$, the result follows.
\qed

\subsubsection{Proof of Lemma~\ref{coro_ZNtoZ_with_mu}}
\label{proof:coro_ZNtoZ_with_mu}

By the triangle inequality,
\begin{align}
	&
	\mathbb{E}\left[
		\left|
			\mathfrak{Z}^\ioN_t
			\left(G_N,\umu^N_t,\upi^N_t\right)
			-
			\mathfrak{Z}^\ioN_t
			\left(G,\umu_t,\upi_t\right)
		\right|
	\right]
	\notag
	\\
	&\quad\leq
	\mathbb{E}\left[
		\left|
			\mathfrak{Z}^\ioN_t
			\left(G_N,\umu^N_t,\upi^N_t\right)
			-
			\mathfrak{Z}^\ioN_t
			\left(G_N,\umu_t,\upi^N_t\right)
		\right|
	\right]+
	\left|
		\mathfrak{Z}^\ioN_t
		\left(G_N,\umu_t,\upi^N_t\right)
		-
		\mathfrak{Z}^\ioN_t
		\left(G,\umu_t,\upi_t\right)
	\right|.
	\label{eq_coroZNtoZ_with_mu}
\end{align}

\noindent
\textbf{First term.}
We have
\begin{align*}
	&
	\mathbb{E}\left[
		\left|
			\mathfrak{Z}^\ioN_t
			\left(G_N,\umu^N_t,\upi^N_t\right)
			-
			\mathfrak{Z}^\ioN_t
			\left(G_N,\umu_t,\upi^N_t\right)
		\right|
	\right]
	\\
	&=
	\E\left[
		\left|
			\int_{v\in I}
			G_N(\ioN,v)
			\sum_{y\in\cX}
			\varphi_t
			\left(y,\pi_t^{N,v}(y)\right)
			\left(
				\mu_t^{N,v}(y)-\mu_t^v(y)
			\right)dv
		\right|
	\right]
	\\
	&=
	\E\left[
		\left|
			\umu_t^N
			\left(\widehat f^{N,\ioN}\right)
			-
			\umu_t
			\left(\widehat f^{N,\ioN}\right)
		\right|
	\right],
\end{align*}
where
\[
	\widehat f^{N,u}(v,y)
	:=
	G_N(u,v)\,
	\varphi_t
	\left(y,\pi_t^{N,v}(y)\right),
	\qquad
	(v,y)\in I\times\cX.
\]
Since $0\leq G_N\leq1$ and
$|\varphi_t(y,a)|\leq C_\varphi$, we have
$
	\left|
		\widehat f^{N,u}(v,y)
	\right|
	\leq C_\varphi.
$
Although $\widehat f^{N,\ioN}$ depends on $N$, $i$, and the policies,
it belongs to $\cF_{C_\varphi}$. Hence, the uniform convergence in
Lemma~\ref{lemma_muNtomu} applies. Therefore, for every
$\epsilon>0$, there exists $N_1$ such that, for all $N\geq N_1$,
\[
	\mathbb{E}\left[
		\left|
			\mathfrak{Z}^\ioN_t
			\left(G_N,\umu^N_t,\upi^N_t\right)
			-
			\mathfrak{Z}^\ioN_t
			\left(G_N,\umu_t,\upi^N_t\right)
		\right|
	\right]
	\leq\frac{\epsilon}{2},
\]
for every $i\in[N]$ and every $t\in\cT$, uniformly with respect to
$\bupi\in\buPi^{L_\pi}$ and
$\widetilde\bpi\in\bfPi$.

\noindent
\textbf{Second term.}
By Lemma~\ref{lemma_ZtoZ_convergence}, for every
$\epsilon,p>0$, there exists $N_2$ such that, for all
$N\geq N_2$, there exists a subset
$\cI_N\subseteq[N]$ with $|\cI_N|\geq(1-p)N$
such that, for every $i\in\cI_N$,
\[
	\left|
		\mathfrak{Z}^\ioN_t
		\left(G_N,\umu_t,\upi^N_t\right)
		-
		\mathfrak{Z}^\ioN_t
		\left(G,\umu_t,\upi_t\right)
	\right|
	\leq\frac{\epsilon}{2},
	\qquad
	\forall\,t\in\cT.
\]
The integer $N_2$ can be chosen uniformly with respect to
$\bupi\in\buPi^{L_\pi}$ and
$\widetilde\bpi\in\bfPi$.

Taking $N_0:=\max\{N_1,N_2\}$,
we conclude from~\eqref{eq_coroZNtoZ_with_mu} that, for all
$N\geq N_0, i\in\cI_N, t\in\cT$,
\[
	\E\left[
		\left|
			\mathfrak{Z}^\ioN_t
			\left(G_N,\umu^N_t,\upi^N_t\right)
			-
			\mathfrak{Z}^\ioN_t
			\left(G,\umu_t,\upi_t\right)
		\right|
	\right]
	\leq\epsilon.
\]
The estimates are uniform with respect to
$\bupi\in\buPi^{L_\pi}$ and
$\widetilde\bpi\in\bfPi$.
\qed

\newpage

\end{document}